\documentclass[11pt]{amsart}
\usepackage{amsmath,amssymb,amsthm}
\usepackage[all]{xy}
\usepackage{bm}
\usepackage{pinlabel}
\usepackage{srcltx}
\usepackage[hypcap=true,font=small,margin=10pt]{caption}
\usepackage[hypcap=true,format=hang,labelfont=up]{subcaption}
\usepackage[colorlinks=true,linkcolor=blue,citecolor=blue,urlcolor=blue]{hyperref}

\title[Joining and gluing SFH]
{Joining and gluing sutured Floer homology}
\author[Rumen Zarev]{Rumen Zarev}
\address{Department  of Mathematics\\Columbia University\\
New York\\NY 10027}
\email{rzarev@math.columbia.edu}
\thanks{The author was partially supported by NSF grant number DMS-0804121.}

\def\co{\colon\thinspace} 
\newcommand\lu[1]{\mbox{}^{#1}} 
\newcommand\li[1]{\mbox{}#1} 

\def\Figure{Figure}
\def\Figures{Figures}
\def\Equation{Eq.}

\def\Theorem{Theorem}
\def\Theorems{Theorems}
\def\Proposition{Proposition}
\def\Propositions{Propositions}
\def\Lemma{Lemma}

\def\Section{Section}

\def\Definition{Definition}
\def\Definitions{Definitions}
\def\Appendix{Appendix}

\newcommand{\del}{\partial}

\newcommand{\sqtens}{\boxtimes} 

\newcommand{\ol}{\overline}
\newcommand{\into}{\hookrightarrow}

\DeclareMathOperator{\Int}{Int}

\DeclareMathOperator{\id}{id}

\DeclareMathOperator{\op}{op}
\DeclareMathOperator{\Hom}{Hom}
\DeclareMathOperator{\Mor}{Mor}
\DeclareMathOperator{\Mod}{Mod}
\DeclareMathOperator{\Db}{D}	

\DeclareMathOperator{\BBar}{Bar}

\newcommand{\A}{\mathcal{A}}

\newcommand{\D}{\mathcal{D}}
\newcommand{\F}{\mathcal{F}}
\newcommand{\G}{\mathcal{G}}
\newcommand{\I}{\mathcal{I}}

\newcommand{\M}{\mathcal{M}}
\newcommand{\Q}{\mathcal{Q}}

\newcommand{\W}{\mathcal{W}}
\newcommand{\Y}{\mathcal{Y}}
\newcommand{\Z}{\mathcal{Z}}
\newcommand{\HH}{\mathcal{H}}
\newcommand{\PP}{\mathcal{P}}
\newcommand{\TW}{\mathcal{TW}}

\newcommand{\ZZ}{\mathbb{Z}}
\newcommand{\II}{\mathbb{I}}
\newcommand{\RR}{\mathbb{R}}

\newcommand{\ZZZ}{\mathbf{Z}}
\newcommand{\aaa}{\mathbf{a}}
\newcommand{\xxx}{\mathbf{x}}
\newcommand{\yyy}{\mathbf{y}}

\newcommand{\balpha}{\boldsymbol\alpha}
\newcommand{\bbeta}{\boldsymbol\beta}

\newcommand{\Ainf}{\mathcal{A}_{\infty}}

\newcommand{\HFhat}{\widehat{\textit{HF}}}
\newcommand{\CFD}{\widehat{\textit{CFD}}}
\newcommand{\CFA}{\widehat{\textit{CFA}}}
\newcommand{\BSD}{\widehat{\textit{BSD}}}
\newcommand{\BSDA}{\widehat{\textit{BSDA}}}
\newcommand{\BSAA}{\widehat{\textit{BSAA}}}
\newcommand{\BSDD}{\widehat{\textit{BSDD}}}
\newcommand{\BSAD}{\widehat{\textit{BSAD}}}

\newcommand{\BSA}{\widehat{\textit{BSA}}}
\newcommand{\SFH}{\textit{SFH}}
\newcommand{\SFC}{\textit{SFC}}

\newcommand{\dtens}{\mathbin{\widetilde{\otimes}}} 

\newcommand{\Zbar}{\overline{\Z}}

\newtheorem{THM}{Theorem}

\newtheorem{CONJ}[THM]{Conjecture}
\newtheorem{thm}{Theorem}[section]
\newtheorem{prop}[thm]{Proposition}

\newtheorem{lem}[thm]{Lemma}

\newtheorem{defn}[thm]{Definition}

\theoremstyle{remark}

\newtheorem*{rmk}{Remark}

\begin{document}

\begin{abstract}
	We give a partial characterization of bordered Floer homology in terms of sutured Floer homology. The bordered
	algebra and modules are direct sums of certain sutured Floer complexes. The algebra multiplication and
	algebra action correspond to a new gluing map on $\SFH$. It is defined algebraically, and is a special case of a
	more general ``join'' map.
	
	In a follow-up paper we show that this gluing map can be identified with the contact cobordism map of
	Honda-Kazez-Mati\'c. The join map is conjecturally equivalent to the cobordism maps on $\SFH$ defined by Juh\'asz. 
\end{abstract}

\maketitle

\section{Introduction}
\label{sec:introduction}

Heegaard Floer homology is a family of invariants for 3 and 4--manifold invariants defined by counting
pseudo-holomorphic curves, originally introduced by Ozsv\'ath and Szab\'o. The most simple form associates to an
oriented 3--manifold $Y$ a graded homology group $\HFhat(Y)$~\cite{OS:HF_definition,OS:HF_properties}.

While Heegaard Floer theory for closed 3--manifolds has been very successful, a lot of the applications involve
manifolds with boundary. In~\cite{Juh:SFH} Juh\'asz introduced sutured Floer homology, or $\SFH$, which
generalizes $\HFhat$ to sutured manifolds.
Introduced by Gabai in~\cite{Gab:foliations}, they are 3--manifolds with boundary, and some extra structure.
In the context of Heegaard Floer homology, the extra structure can be considered to be a multicurve
$\Gamma$, called a \emph{dividing set}, on the boundary of the 3--manifold $Y$. Sutured Floer homology associates to such
a pair $(Y,\Gamma)$ a homology group $\SFH(Y,\Gamma)$.

Among other applications, sutured Floer homology has been used to solve problems in contact topology, via a contact
invariant for contact manifolds with boundary, and a map associated to contact cobordisms, defined by Honda, Kazez, and
Mati\'c in~\cite{HKM:EH, HKM:TQFT}. This map has been used by Juh\'asz in~\cite{Juh:cobordisms} to define a map on $\SFH$ associated to a
cobordism (with corners) between two sutured manifolds.

A shortcoming of sutured Floer homology is that there is little relationship between the groups $\SFH(Y,\Gamma_1)$ and
$\SFH(Y,\Gamma_2)$, where $\Gamma_1$ and $\Gamma_2$ are two dividing sets on the same manifold $Y$. For example one can
find many examples where one of the groups vanishes, while the other does not. Moreover, the groups $\SFH(Y_1,\Gamma_1)$ and
$\SFH(Y_2,\Gamma_2)$ are not sufficient to reconstruct $\HFhat(Y)$, where $Y=Y_1\cup Y_2$ is a closed manifold.

To overcome these shortcomings, Lipshitz, Ozsv\'ath, and Thurston introduced in~\cite{LOT:pairing} a new Heegaard Floer
invariant for 3--manifolds with boundary called \emph{bordered Floer homology}. To a parametrized closed connected surface $F$
they associate a DG-algebra $\A(F)$. To a 3--manifold $Y$ with boundary $\del Y\cong F$ they associate
an $\Ainf$--module $\CFA(Y)$ over $\A(F)$ (defined up to $\Ainf$--homotopy equivalence). This invariant overcomes both
of the above shortcomings of $\SFH$. On the one hand, given two parametrizations of the surface $F$, the modules $\CFA(Y)$
associated to these parametrizations can be computed from each other. On the other hand, if $Y_1$ and $Y_2$ are two
manifolds with boundary diffeomorphic to $F$, the group $\HFhat(Y_1\cup_F Y_2)$ can be computed from $\CFA(Y_1)$ and $\CFA(Y_2)$.

The natural question arises: How are these two theories for 3--manifolds with boundary related to each other? Can $\SFH(Y,\Gamma)$
be computed from $\CFA(Y)$, and if yes, how? Can $\CFA(Y)$ be computed from the sutured homology of $Y$, and if yes, how?

In~\cite{Zar:BSFH} we introduced \emph{bordered sutured Floer homology}, to serve as a bridge between
the two worlds. We used it to answer the first part of the above question---to each dividing set $\Gamma$ on $F$ we can
associate a module $\CFD(\Gamma)$ over $\A(F)$, such that $\SFH(Y,\Gamma)$ is simply the homology of the derived tensor product
$\CFA(Y)\dtens\CFD(\Gamma)$.

In the current paper we answer the second half of this question. We show that for a given parametrization of $F$, the
homologies of the bordered algebra $\A(F)$ and the module $\CFA(Y)$ associated to a 3--manifold $Y$ are direct sums of
finitely many sutured Floer homology groups. Moreover we identify multiplication in $H_*(\A(F))$ and the action of
$H_*(\A(F))$ on $H_*(\CFA(Y))$ with a certain \emph{gluing map} $\Psi$ on sutured Floer homology.

\subsection{Results}
\label{sec:intro_results}

The first result of this paper is to define the gluing map $\Psi$ discussed above. Suppose $(Y_1,\Gamma_1)$ and
$(Y_2,\Gamma_2)$ are two sutured manifolds. We say that we can glue them if there are subsets $F_1$ and $F_2$
of their boundaries, where $F_1$ can be identified with the mirror of $F_2$, such that the multicurve $\Gamma_1\cap F_1$
is identified with $\Gamma_2\cap F_2$, preserving the orientations on $\Gamma_i$. This means that the regions $R_+$ and
$R_-$ on the two boundaries are interchanged.
We will only talk of gluing in the case when $F_i$ have no closed components, and all components of $\del
F_i$ intersect the dividing sets $\Gamma_i$.
\begin{defn}
	\label{def:intro_gluing}
	Suppose $(Y_1,\Gamma_1)$, $(Y_2,\Gamma_2)$, $F_1$ and $F_2$ are as above. The \emph{gluing of $(Y_1,\Gamma_1)$ and
	$(Y_2,\Gamma_2)$ along $F_i$} is the sutured manifold $(Y_1\cup_{F_i} Y_2,\Gamma_{1+2})$. The dividing set
	$\Gamma_{1+2}$
	is obtained from $(\Gamma_1\setminus F_1)\cup_{\del F_i}(\Gamma_2\setminus F_2)$ as follows.
	Along each component $f$ of $\del F_i$ the orientations of $\Gamma_1$ and $\Gamma_2$ disagree. We apply the
	minimal possible positive fractional Dehn twist along $f$ that gives a consistent orientation.
\end{defn}

An illustration of gluing is given in \Figure~\ref{fig:intro_gluing_example}. We define a gluing map
$\Psi$ on $\SFH$ corresponding to this topological construction.

\begin{figure}
	\labellist
	\pinlabel $\bigcup$ at 172 126
	\pinlabel $\longrightarrow$ at 386 126
	\hair=4pt
	\pinlabel $\Y_1$ [t] at 68 24
	\pinlabel $\Y_2$ [t] at 276 24
	\pinlabel $\Y_1\cup_{F_i}\Y_2$ [t] at 570 24
	\pinlabel $F_1$ [b] at 136 228
	\pinlabel $F_2$ [b] at 208 228
	\endlabellist
	\includegraphics[scale=.5]{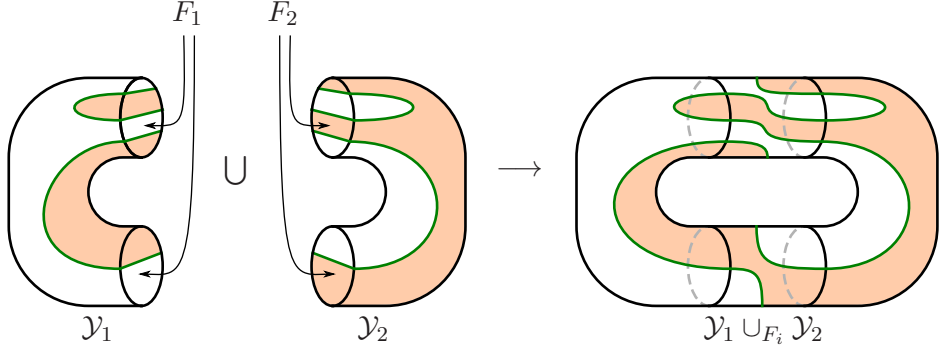}
	\caption{Gluing two solid balls along $F=D^2\cup D^2$, to obtain a solid torus. The $R_+$ regions have been shaded.}
	\label{fig:intro_gluing_example}
\end{figure}

\begin{THM}
	\label{thm:intro_join_and_gluing}
	Let $(Y_1,\Gamma_1)$ and $(Y_2,\Gamma_2)$ be two balanced sutured manifolds, that can be glued along $F$.
	Then there is a well defined map
	\begin{equation*}
		\Psi_F \co \SFH(Y_1,\Gamma_1)\otimes\SFH(Y_2,\Gamma_2)\to\SFH( (Y_1,\Gamma_1)\cup_{F}(Y_2,\Gamma_2)),
	\end{equation*}
	satisfying the following properties:
	\begin{enumerate}
		\item Symmetry: The map $\Psi_F$ for gluing $Y_1$ to  $Y_2$ is equal to that for gluing $Y_2$ to $Y_1$.
		\item Associativity: Suppose that we can glue $Y_1$ to $Y_2$ along $F_1$, and $Y_2$ to $Y_3$ along $F_2$, such that
			$F_1$ and $F_2$ are disjoint in $\del Y_2$. Then the order of gluing is irrelevant:
			\begin{equation*}
				\Psi_{F_2}\circ\Psi_{F_1}=\Psi_{F_1}\circ\Psi_{F_2}=\Psi_{F_1\cup F_2}.
			\end{equation*}
		\item Identity: Given a dividing set $\Gamma$ on $F$, there is a dividing set
			$\Gamma'$ on $F\times[0,1]$, and an element $\Delta_\Gamma\in\SFH(F\times[0,1],\Gamma')$, satisfying the
			following. For any sutured manifold $(Y,\Gamma'')$ with $F\subset\del Y$ and $\Gamma''\cap F=\Gamma$, there
			is a diffeomorphism $(Y,\Gamma'')\cup_F(F\times[0,1],\Gamma')\cong(Y,\Gamma'')$.
			Moreover, the map $\Psi_F(\cdot,\Delta_\Gamma)$ is the identity of $\SFH(Y,\Gamma'')$.
	\end{enumerate}
\end{THM}

One application of this result is the following characterization of bordered Floer homology in terms of $\SFH$ and the
gluing map.
Fix a parametrized closed surface $F$, with bordered algebra $A=\A(F)$. Let $F'$ be $F$ with
a disc removed, and let $p,q\in\del F'$ be two points. We can find $2^{2g(F)}$ distinguished dividing sets on $F$, which
we denote $\Gamma_I$ for
$I\subset\{1,\ldots,2g\}$, and corresponding dividing sets $\Gamma'_I=\Gamma_I\cap F'$ on $F'$.
Let $\Gamma_{I\to J}$ be a dividing set on $F'\times[0,1]$ which is $\Gamma'_I$ along $F'\times\{0\}$,
$\Gamma'_J$ along $F'\times\{1\}$, and half of a negative Dehn twist of $\{p,q\}\times[0,1]$ along $\del
F'\times[0,1]$.

\begin{THM}
	\label{thm:intro_bordered_via_SFH}
	Suppose the surfaces $F$ and $F'$, the algebra $A$, and the dividing sets $\Gamma_I$, $\Gamma'_I$, and $\Gamma_{I\to
	J}$ are as described above.
	Then there is an isomorphism
	\begin{equation*}
		H_*(A) \cong\bigoplus_{I,J\subset\{1,\ldots,2g\}}\SFH(F'\times[0,1],\Gamma_{I\to J}),\\
	\end{equation*}
	and the multiplication map $\mu_2$ on $H_*(A)$ can be identified with the gluing map $\Psi_{F'}$. It maps 
	$\SFH(F'\times[0,1],\Gamma_{I\to J})\otimes\SFH(F'\times[0,1],\Gamma_{J\to K})$ to
		$\SFH(F'\times[0,1],\Gamma_{I\to K})$, and sends all other summands to 0.
\end{THM}

The module $\CFA$ can be similarly described.

\begin{THM}
	\label{thm:intro_modules}
	Suppose $Y$ is a 3--manifold with boundary $\del Y\cong F$. There is an isomorphism
	\begin{align*}
		H_*(\CFA(Y)_A)\cong\bigoplus_{I\subset\{1,\ldots,2g\}} \SFH(Y,\Gamma_I),
	\end{align*}
	and the action $m_2$ of $H_*(A)$ on $H_*(\CFA(Y))$ can be identified with the gluing map $\Psi_{F'}$. It maps
	$\SFH(Y,\Gamma_I)\otimes\SFH(F'\times[0,1],\Gamma_{I\to J})$ to $\SFH(Y,\Gamma_J)$, and sends all other summands to 0.
\end{THM}

The gluing construction and the gluing map readily generalize to a more general \emph{join} construction, and \emph{join
map}, which are 3--dimensional analogs. Suppose that $(Y_1,\Gamma_1)$ and $(Y_2,\Gamma_2)$ are two sutured
manifolds, and $F_1$ and $F_2$ are subsets of their boundaries, satisfying the conditions for gluing. Suppose further
that the diffeomorphism $F_1\to F_2$ extends to $W_1\to W_2$, where $W_i$ is a compact codimension--0 submanifold of
$Y_i$, and $\del W_i\cap\del Y_i=F_i$. Instead of gluing $Y_1$ and $Y_2$ along $F_i$, we can \emph{join} them along
$W_i$.

\begin{defn}
	\label{def:intro_join}
	The \emph{join of $(Y_1,\Gamma_1)$ and $(Y_2,\Gamma_2)$ along $W_i$} is the sutured manifold
	\begin{equation*}
		( (Y_1\setminus W_1)\cup_{\del W_i\setminus F_i} (Y_2\setminus W_2),\Gamma_{1+2}),
	\end{equation*}
	where the dividing set $\Gamma_{1+2}$ is constructed exactly as in \Definition~\ref{def:intro_gluing}. We denote the
	join by $(Y_1,\Gamma_1)\Cup_{W_i}(Y_2,\Gamma_2)$.
\end{defn}

An example of a join is shown in \Figure~\ref{fig:intro_join_example}. Notice that if $W_i$ is a collar neighborhood of
$F_i$, then the notions of join and gluing coincide. That is, the join operation is indeed a generalization of gluing.
In fact, throughout the paper we work almost exclusively with joins, while only regarding gluing as a special case.

\begin{THM}
	\label{thm:intro_join_map}
	There is a well-defined join map
	\begin{equation*}
		\Psi_W\co\SFH(Y_1,\Gamma_1)\otimes\SFH(Y_2,\Gamma_2)\to\SFH((Y_1,\Gamma_1)\Cup_{W}(Y_2,\Gamma_2)),
	\end{equation*}
	satisfying properties of symmetry, associativity, and identity, analogous to those listed in
	\Theorem~\ref{thm:intro_join_and_gluing}.
\end{THM}

\begin{figure}
	\labellist
	\pinlabel $\bigcup$ at 230 344
	\hair=4pt
	\pinlabel $\Y_1$ [t] at 100 260
	\pinlabel $\Y_2$ [t] at 360 260
	\pinlabel $\Y_1\Cup_{\W_i}\Y_2$ [t] at 232 24
	\hair=0.5pt
	\pinlabel $\W_1$ [b] at 168 452
	\pinlabel $\W_2$ [b] at 292 452
	\endlabellist
	\includegraphics[scale=.6]{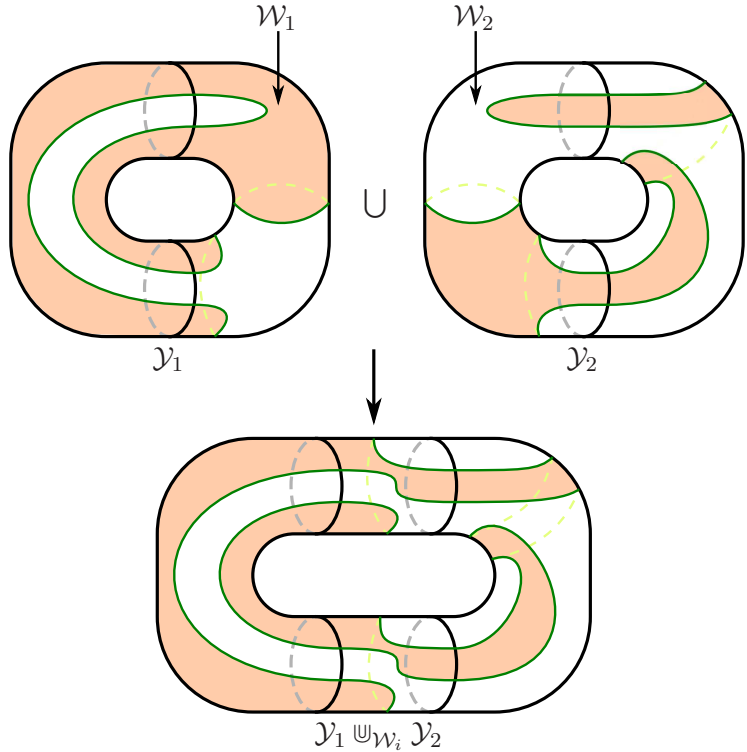}
	\caption{Join of two solid tori along $D^2\times[0,1]$, to obtain another solid torus. The $R_+$ regions have been shaded.}
	\label{fig:intro_join_example}
\end{figure}

The join map is constructed as follows. We cut out $W_1$ and $W_2$ from $Y_1$ and $Y_2$, respectively, and regard the
complements as bordered sutured manifolds. The join operation corresponds to replacing $W_1$ and $W_2$ by an interpolating piece $\TW_{F,+}$.
We define a map between the bordered sutured invariants, from the product
$\BSA(W_1)\otimes\BSA(W_2)$ to the bimodule $\BSAA(\TW_{F,+})$. We show that for an appropriate choice of
parametrizations, the modules $\BSA(W_1)$ and $\BSA(W_2)$ are duals, while $\BSAA(\TW_{F,+})$ is the
dual of the bordered algebra for $F$. The map is then an $\Ainf$--version of the natural pairing between a module and its dual. The proof of invariance and the
properties from \Theorems~\ref{thm:intro_join_and_gluing} and~\ref{thm:intro_join_map} is purely algebraic. Most of the arguments involve
$\Ainf$--versions of standard facts in commutative algebra.

The proofs of \Theorems~\ref{thm:intro_bordered_via_SFH} and~\ref{thm:intro_modules} involve several steps. First, we find a
manifold whose bordered sutured invariant is the bordered algebra, as a bimodule over itself. Second, we find
manifolds whose bordered sutured invariants are all possible simple modules over the algebra. Finally, we compute the gluing map
$\Psi$ explicitly in several cases.

\subsection{Further implications and speculations}
\label{sec:speculations}

In a follow-up paper~\cite{Zar:gluing2} we prove that the gluing map map $\Psi_F$ can be identified with the contact cobordism map
from~\cite{HKM:TQFT}. This is somewhat surprising as that the definition of that map uses contact structures and open
books, while our definition uses bordered sutured Floer homology and is purely algebraic. The equivalence of the two
maps also gives a purely contact-geometric interpretation of the bordered algebra.

There is no analog of the join map in the setting of Honda, Kazez, and Mati\'c. However, there is a natural
pair-of-pants cobordism
\begin{equation*}
	Z_W\co (Y_1,\Gamma_1)\sqcup (Y_2,\Gamma_2)\to (Y_1,\Gamma_1)\Cup_W(Y_2,\Gamma_2),
\end{equation*}
and we conjecture that the join map $\Psi_W$ is equivalent to the cobordism map $F_{Z_W}$ that Juh\'asz associates to such a
cobordism, by counting pseudo-holomorphic triangles.

Though \Theorems~\ref{thm:intro_bordered_via_SFH} and~\ref{thm:intro_modules} give a pretty good description of bordered Floer
homology in terms of sutured Floer homology, it is not complete. For instance, to be able to recover the pairing theorem
for bordered Floer homology, we need to work either with the full bordered DG-algebra $\A(F)$, or with its homology
$H_*(\A(F))$, considered as an $\Ainf$--algebra. That is, $H_*(\A(F))$ inherits higher multiplication maps $\mu_i$,
for $i\geq 2$ from the DG-structure on $\A(F)$. \Theorem~\ref{thm:intro_bordered_via_SFH} only recovers $\mu_2$. Similarly,
$H_*(\CFA(Y))$ has higher actions $m_i$, for $i\geq 2$ by $H_*(\A(F))$, while \Theorem~\ref{thm:intro_modules} only recovers
$m_2$.

We believe that these higher structures can be recovered by a similar construction. There are
maps, $\ol\Psi_i$, for $i\geq 2$, defined algebraically, similar to $\Psi$, of the following form:
\begin{equation*}
	\ol\Psi_i\co\SFC(Y_1)\otimes\cdots\otimes\SFC(Y_i)\to\SFC(Y_1\Cup\ldots\Cup Y_i).
\end{equation*}
Here $\SFC$ denotes the chain complex defining the homology group $\SFH$. The first term $\ol\Psi_2$ induces the usual
join $\Psi$ on homology.

\begin{CONJ}
	The following two statements hold:
	\begin{enumerate}
		\item The collection of maps $\ol\Psi_i$, for $i\geq 2$ can be used to recover the higher multiplications $\mu_i$ on
			$H_*(\A(F))$, and the higher actions $m_i$ of $H_*(\A(F))$ on $H_*(\CFA(Y))$.
		\item The map $\ol\Psi_i$ can be computed by counting pseudo-holomorphic $(i+1)$--gons in a sutured Heegaard
			multidiagram.
	\end{enumerate}
\end{CONJ}

Analogs of sutured Floer homology have been defined in settings other than Heegaard Floer homology---for instanton
and monopole Floer homology in~\cite{KM:sutured}, and for embedded contact homology in~\cite{CGHH:sutured_ECH}. We
believe that analogs of the join and gluing maps can be used to extend bordered Floer homology to those settings.

\subsection*{Organization}
\label{sec:intro_organization}
We start by introducing in more detail the topological constructions of the gluing join operations in \Section~\ref{sec:topology}.
In \Section~\ref{sec:bordered} we recall briefly the definitions of the bordered sutured
invariants $\BSA$ and $\BSD$. We also discuss how the original definitions involving only $\alpha$--arcs can be extended
to diagrams using both $\beta$--arcs, and to some mixed diagrams using both. Section~\ref{sec:BSA_computations} contains
computations of several $\BSA$ invariants needed later.

We define the join map in \Section~\ref{sec:join}, on the level of chain complexes.
The same section contains the proof that it descends to a unique map on homology.
In the following \Section~\ref{sec:join_properties} we prove the properties
from \Theorems~\ref{thm:intro_join_and_gluing} and~\ref{thm:intro_join_map}.
Finally, \Section~\ref{sec:reinterpretation} contains the statement and the proof of a slightly more general version of
\Theorems~\ref{thm:intro_bordered_via_SFH} and~\ref{thm:intro_modules}.

Throughout the paper, we make use of a diagrammatic calculus to compute $\Ainf$--morphisms, which greatly simplifies
the arguments.  \Appendix~\ref{sec:diagrams} contains a brief description of this calculus, and the necessary algebraic
assumptions. \Appendix~\ref{sec:algebra} gives an overview of $\Ainf$--bimodules in terms of the diagrammatic calculus,
as they are used in the paper.

\subsection*{Acknowledgments}
\label{sec:intro_acknowledgments}
The author is grateful to his advisor Peter Ozsv\'ath, and to Robert Lipshitz, Dylan Thurston, and Shea Vela-Vick for
many productive discussions about this work. Shea Vela-Vick, Robert Lipshitz, and Peter Ozsv\'ath also gave much appreciated feedback on
earlier versions of this paper. A significant portion of the work described here was carried out at the
Mathematical Sciences Research Institute, which kindly hosted the author as a program associate in the program
``Homology theories of knots and links'', during the Spring of 2010.

\section{Topological preliminaries}
\label{sec:topology}
We recall the definition of a sutured manifold and some auxiliary notions, and define what we mean by gluing and surgery.

\begin{rmk}
	Throughout the paper all manifolds are oriented. We use $-M$ to denote the manifold $M$ with
	its orientation reversed.
\end{rmk}

\subsection{Sutured manifolds and surfaces}
\label{sec:sutured_defs}
\begin{defn}
	\label{def:sutured_manifold}
	As defined in~\cite{Juh:SFH}, a \emph{balanced sutured manifold} is a pair $\Y=(Y,\Gamma)$ consisting of the following:
	\begin{itemize}
		\item An oriented 3--manifold $Y$ with boundary.
		\item A collection $\Gamma$ of disjoint oriented simple closed curves in $\del Y$, called \emph{sutures}.
	\end{itemize}

	They are required to satisfy the following conditions:
	\begin{itemize}
		\item $Y$ can be disconnected but cannot have any closed components.
		\item $\del Y$ is divided by $\Gamma$ into two complementary regions $R_+(\Gamma)$ and $R_-(\Gamma)$ such that $\del
			R_\pm(Y)=\pm\Gamma$. ($R_+$ and $R_-$ may be disconnected.)
		\item Each component of $\del Y$ contains a suture. Equivalently, $R_+$ and $R_-$ have no closed components.
		\item $\chi(R_+)=\chi(R_-)$.
	\end{itemize}
\end{defn}

In~\cite{Zar:BSFH} we introduced the notion of a sutured surface.

\begin{defn}
	\label{def:sutured_surface}
	A \emph{sutured surface} is a pair $\F=(F,\Lambda)$ consisting of the following:
	\begin{itemize}
		\item A compact oriented surface $F$.
		\item A finite collection $\Lambda\subset\del F$ of points with sign, called \emph{sutures}.
	\end{itemize}

	They are required to satisfy the following conditions:
	\begin{itemize}
		\item $F$ can be disconnected but cannot have any closed components.
		\item $\del F$ is divided by $\Lambda$ into two complementary regions $S_+(\Gamma)$ and $S_-(\Gamma)$ such that $\del
			S_\pm(Y)=\pm\Lambda$. ($S_+$ and $S_-$ may be disconnected.)
		\item Each component of $\del F$ contains a suture. Equivalently, $S_+$ and $S_-$ have no closed components.
	\end{itemize}
\end{defn}

A sutured surface is precisely the 2--dimensional equivalent of a balanced sutured manifold. The requirement
$\chi(S_+)=\chi(S_-)$ follows automatically from the other conditions.

From $\F=(F,\Lambda)$ we can construct two other sutured surfaces: $-\F=(-F,-\Lambda)$, and $\ol{\F}=(-F,\Lambda)$.
In both of $-\F$ and $\ol\F$, the orientation of the underlying surface $F$ is reversed.
The difference between the two is that in $-\F$ the roles of $S_+$ and $S_-$ are preserved, while in $\ol\F$ they are
reversed.

\begin{defn}
	\label{def:dividing_set}
	Suppose $\F=(F,\Lambda)$ is a sutured surface. A \emph{dividing set} $\Gamma$ for $\F$ is a finite collection
	$\Gamma$ of disjoint embedded oriented arcs and simple closed curves in $F$, with the following properties:
	\begin{itemize}
		\item $\del \Gamma=-\Lambda$, as an oriented boundary.
		\item $\Gamma$ divides $F$ into (possibly disconnected) regions $R_+$ and $R_-$ with $\del R_\pm=(\pm\Gamma)\cup
			S_\pm$.
	\end{itemize}
\end{defn}

We can extend the definition of a dividing set to pairs $(F,\Lambda)$ which do not quite satisfy the conditions for a sutured surface.
We can allow some or all of the components $F$ to be closed. We call such a pair \emph{degenerate}. In that case we
impose the extra condition that each closed component contains a component of $\Gamma$.

Note that the sutures $\Gamma$ of a sutured manifold $(Y,\Gamma)$ can be regarded as a dividing set for the (degenerate)
sutured surface $(\del Y,\varnothing)$.

\begin{defn}
	\label{def:partially_sutured_manifold}
	A \emph{partially sutured manifold} is a triple $\Y=(Y,\Gamma,\F)$ consisting of the following:
	\begin{itemize}
		\item A 3--manifold $Y$ with boundary and 1--dimensional corners.
		\item A sutured surface $\F=(F,\Lambda)$, such that $F\subset\del Y$, and such that the 1--dimensional corner of $Y$ is
			$\del F$.
		\item A dividing set $\Gamma$ for $(\del Y\setminus F,-\Lambda)$ (which might be degenerate).
	\end{itemize}
\end{defn}

Note that a partially sutured manifold $\Y=(Y,\Gamma,\F_1\sqcup\F_2)$ can be thought of as a cobordism between $-\F_1$
and $\F_2$. On the other hand, the partially sutured manifold $\Y=(Y,\Gamma,\varnothing)$ is just a sutured
manifold, although it may not be balanced. We can \emph{concatenate} $\Y=(Y,\Gamma,\F_1\sqcup\F_2)$ and
$\Y'=(Y',\Gamma',-\F_2\sqcup\F_3)$ along $\F_2=(F_2,\Lambda_2)$ and $-\F_2=(-F_2,-\Lambda_2)$ to obtain
\begin{equation*}
	\Y\cup_{\F_2}\Y'=(Y\cup_{F_2} Y',\Gamma\cup_{\Lambda_2}\Gamma',\F_1\sqcup\F_3).
\end{equation*}
We use the term concatenate to distinguish from the operation of \emph{gluing} of two sutured manifolds described
in \Definition~\ref{def:gluing}.

A partially sutured manifold whose sutured surface is parametrized by an arc diagram is a \emph{bordered sutured manifold},
as defined in ~\cite{Zar:BSFH}. We will return to this point in section~\ref{sec:bordered}, where we give the precise
definitions.

An important special case is when $Y$ is a thickening of $F$.

\begin{defn}
	\label{def:cap}
	Suppose $\Gamma$ is a dividing set for the sutured surface $\F=(F,\Lambda)$.
	Let $W=F\times[0,1]$, and $W'=F\times[0,1]/\sim$, where $(p,t)\sim(p,t')$ whenever $p\in\del F$, and
	$t,t'\in[0,1]$. We will refer to the partially sutured manifolds
	\begin{align*}
		\W_{\Gamma}&=(W,\Gamma\times\{1\}\cup\Lambda\times[0,1],(-F\times\{0\},-\Lambda\times\{0\})),\\
		\W'_{\Gamma}&=(W',\Gamma\times\{1\},(-F\times\{0\},-\Lambda\times\{0\}))
	\end{align*}
	as the \emph{caps} for $\F$ associated to $\Gamma$.
\end{defn}

Since $\W'_{\Gamma}$ is just a smoothing of $\W_{\Gamma}$ along the corner $\del F\times\{1\}$, we will not distinguish between them.
An illustration of a dividing set and a cap is shown in
\Figure~\ref{fig:dividing_sets_and_caps}. In this and in all other figures we use the convention that the dividing set
is colored in green, to avoid confusion with Heegaard diagrams later. We also shade the $R_+$ regions.

Notice that the sutured surface for $\W_\Gamma$ is $-\F$. This means that if $\Y=(Y,\Gamma',\F)$ is a partially sutured
manifold, we can concatenate $\Y$ and $\W$ to obtain $(Y,\Gamma'\cup\Gamma)$. That is, the effect is that of ``filling
in'' $F\subset\del Y$ by $\Gamma$.

\begin{figure}
	\begin{subfigure}[b]{.36\linewidth}
		\centering
		\labellist
		\tiny\hair=1pt
		\pinlabel $\boldsymbol{+}$ [r] at 32 160
		\pinlabel $\boldsymbol{-}$ [r] at 16 100
		\pinlabel $\boldsymbol{+}$ [r] at 32 48
		\pinlabel $\boldsymbol{-}$ [r] at 172 100
		\pinlabel $\boldsymbol{-}$ [l] at 216 160
		\pinlabel $\boldsymbol{+}$ [l] at 232 100
		\pinlabel $\boldsymbol{-}$ [l] at 216 48
		\pinlabel $\boldsymbol{-}$ [l] at 76 100
		\pinlabel $\boldsymbol{+}$ [b] at 124 56
		\pinlabel $\boldsymbol{+}$ [t] at 124 152
		\endlabellist
		\includegraphics[scale=.5]{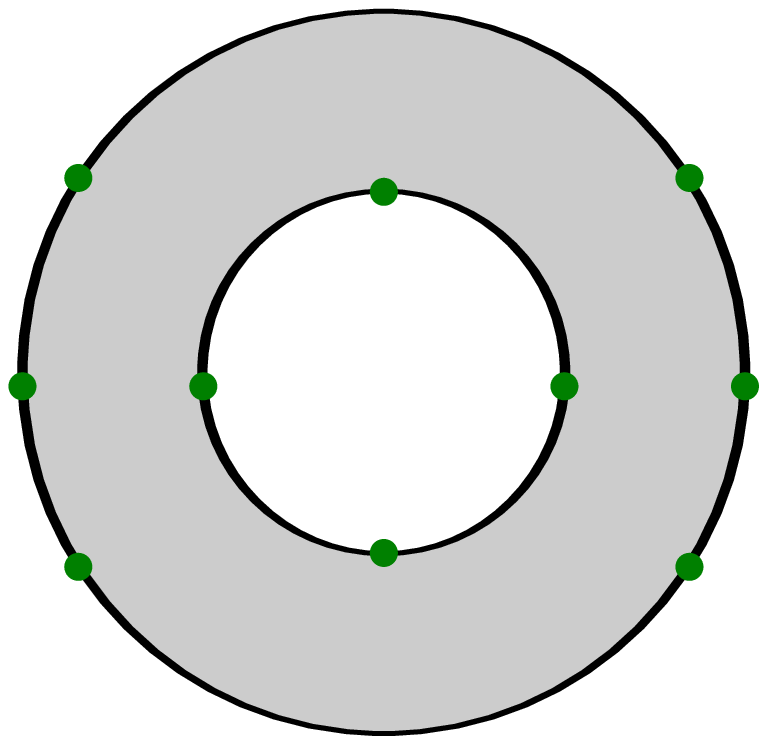}
		\caption{The sutured surface $\F$.}
		\label{subfig:sutured_surface}
	\end{subfigure}
	\begin{subfigure}[b]{.36\linewidth}
		\centering
		\labellist
		\tiny\hair=1pt
		\pinlabel $\boldsymbol{+}$ [r] at 32 160
		\pinlabel $\boldsymbol{-}$ [r] at 16 100
		\pinlabel $\boldsymbol{+}$ [r] at 32 48
		\pinlabel $\boldsymbol{-}$ [r] at 172 100
		\pinlabel $\boldsymbol{-}$ [l] at 216 160
		\pinlabel $\boldsymbol{+}$ [l] at 232 100
		\pinlabel $\boldsymbol{-}$ [l] at 216 48
		\pinlabel $\boldsymbol{-}$ [l] at 76 100
		\pinlabel $\boldsymbol{+}$ [b] at 124 56
		\pinlabel $\boldsymbol{+}$ [t] at 124 152
		\endlabellist
		\includegraphics[scale=.5]{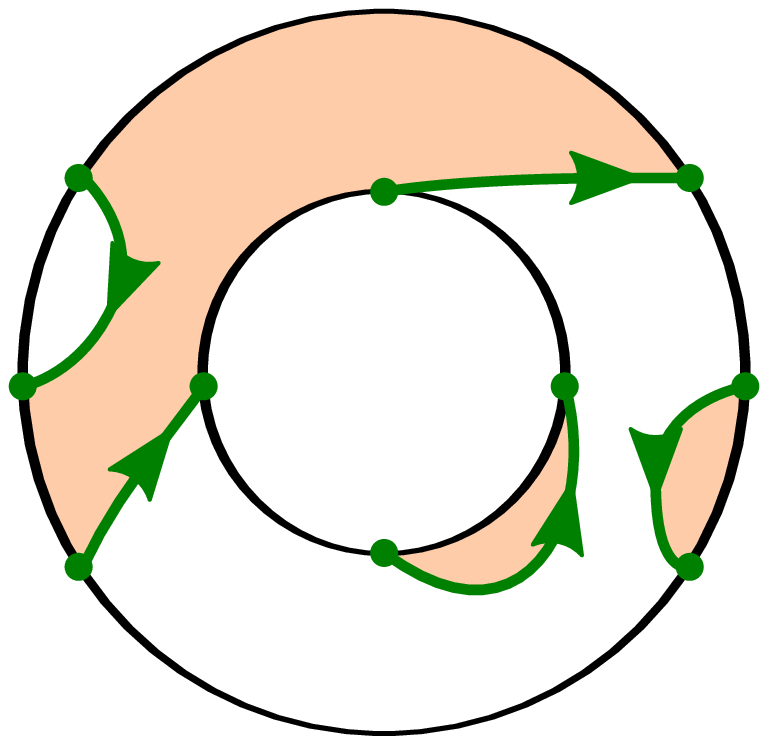}
		\caption{A dividing set $\Gamma$ of $\F$.}
		\label{subfig:dividing_set}
	\end{subfigure}
	\begin{subfigure}[b]{.25\linewidth}
		\centering
		\labellist
		\small\hair=2.5pt
		\pinlabel $-\F$ [b] at 12 244
		\pinlabel $[0,1]$ [b] at 76 228
		\pinlabel $(\F,\Gamma)$ [b] at 132 244
		\endlabellist
		\includegraphics[scale=.5]{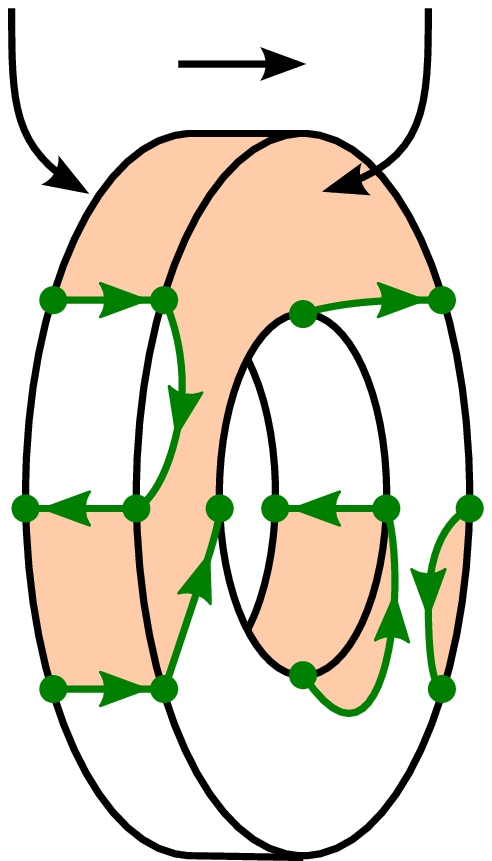}
		\caption{The cap for $\Gamma$.}
		\label{subfig:cap}
	\end{subfigure}
	\caption{A sutured annulus $\F$, with a cap associated to a dividing set.}
	\label{fig:dividing_sets_and_caps}
\end{figure}

\begin{defn}
	\label{def:partially_sutured_embedding}
	Suppose $\F=(F,\Lambda)$ is a sutured surface.
	An \emph{embedding} $\W\into\Y$ of the partially sutured $\W=(W,\Gamma_W,\F)$ into the sutured $\Y=(Y,\Gamma_Y)$ is an embedding
	$W\into Y$ with the following properties:
	\begin{itemize}
		\item $F\subset\del W$ is properly embedded in $Y$ as a separating surface.
		\item $\del W\setminus F=\del Y\cap W$.
		\item $\Gamma_W=\Gamma_Y\cap\del W$.
	\end{itemize}
\end{defn}

The complement $Y\setminus W$ also inherits a partially sutured structure.  We define
\begin{equation*}
	\Y\setminus\W=(Y\setminus W,\Gamma_Y\setminus\Gamma_W,-\F).
\end{equation*}

The definition of embeddings easily extends to $\W\into\Y$ where both $\W=(W,\Gamma_W,\F)$ and $\Y=(Y,\Gamma_Y,\F')$ 
are partially sutured.  In this case we require that $W$ is disjoint from a collar neighborhood of $F'$. Then there is still a complement 
\begin{equation*}
	\Y\setminus\W=(Y\setminus W,\Gamma_Y\setminus\Gamma_W,\F'\cup-\F).
\end{equation*}

In both cases $\Y$ is diffeomorphic to the concatenation $\W\cup_\F(\Y\setminus\W)$.
Examples of a partial sutured manifold and of an embedding are given in \Figure~\ref{fig:sutured_and_partial_example}.

\begin{figure}
	\labellist
	\small\hair=1.5pt
	\pinlabel $\W$ [r] at 16 212
	\pinlabel $\Y$ [r] at 16 52
	\pinlabel $\Y\setminus\W$ [r] at 242 212
	\pinlabel $\F$ [l] at 132 132
	\pinlabel $-\F$ [l] at 262 132
	\endlabellist
	\includegraphics[scale=.6]{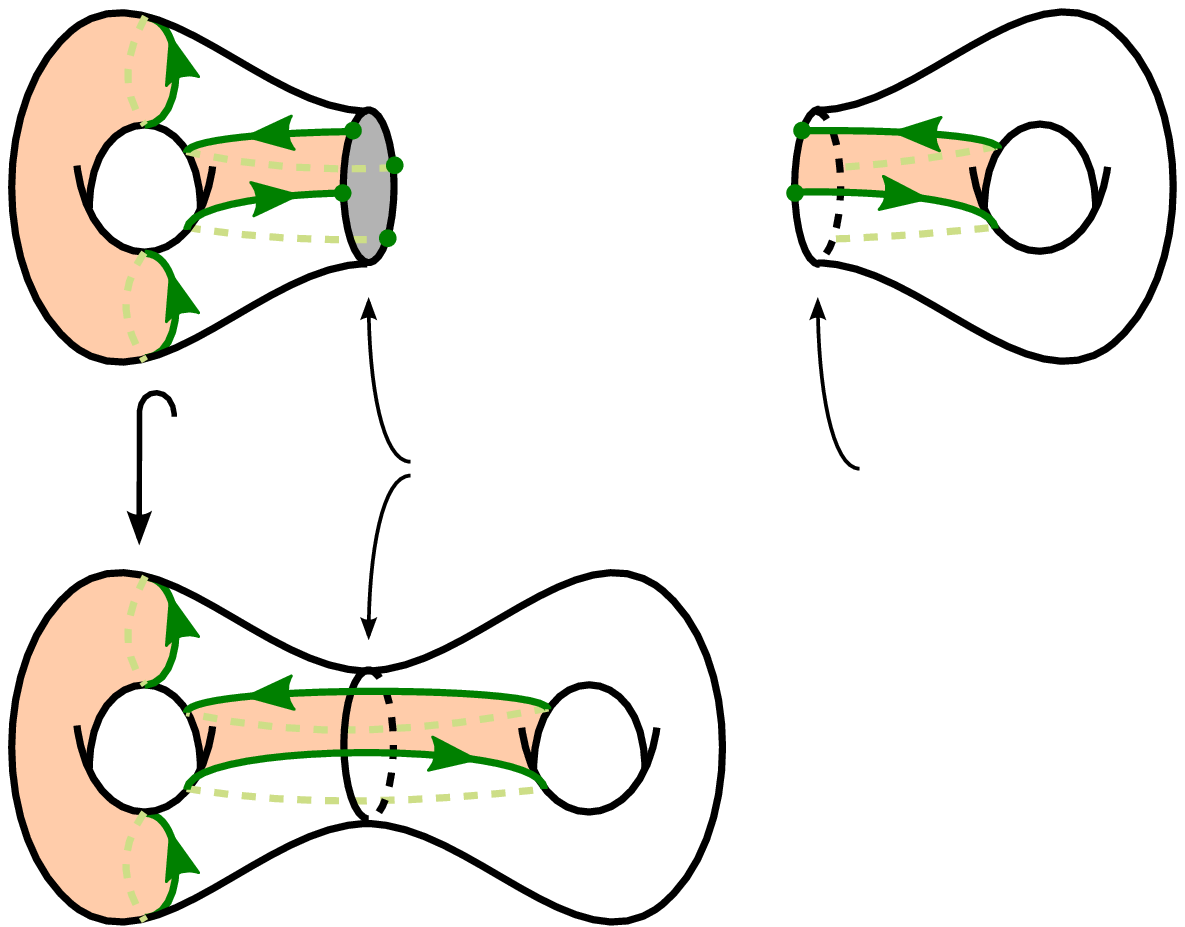}
	\caption{Examples of a partially sutured manifold $\W$ embedding into the sutured manifold $\Y$, and the
	complement $\Y\setminus\W$, which is also partially sutured.}
	\label{fig:sutured_and_partial_example}
\end{figure}

\subsection{Mirrors and doubles; joining and gluing}
\label{sec:joining_def}

We want to define a gluing operation which takes two sutured manifolds $(Y_1,\Gamma_1)$ and $(Y_2,\Gamma_2)$, and surfaces
$F\subset\del Y_1$ and $-F\subset\del Y_2$, and produces a new sutured manifold $(Y_1\cup_{F}Y_2,\Gamma_3)$. To do
that we have to decide how to match up the dividing sets on and around $F$ and $-F$. One solution is to require that we
glue $F\cap R_+(\Gamma_1)$ to $-F\cap R_+(\Gamma_2)$, and $F\cap R_-(\Gamma_1)$ to $-F\cap R_-(\Gamma_2)$. Then
$(\Gamma_1\setminus F)\cup(\Gamma_2\setminus -F)$ is a valid dividing set, and candidate for $\Gamma_3$. The problem with
this approach is that even if we glue two balanced sutured manifolds, the result is not guaranteed to be balanced.

Another approach, suggested by contact topology is the following. We glue $F\cap R_+$ to $-F\cap R_-$, and vice versa.
To compensate for the fact that the dividing sets $\Gamma_1\setminus F$ and $\Gamma_2 \setminus -F$ do not match up
anymore, we introduce a slight twist along $\del F$. In contact topology this twist appears
when we smooth the corner between two convex surfaces meeting at an angle.

It turns out that the same approach is the correct one, from the bordered sutured point of view. To be able to define a
gluing map on $\SFH$ with nice formal properties, the underlying topological operation should employ the same kind of
twist. However, its direction is opposite from the one in the contact world. This is not unexpected, as orientation
reversal is the norm when defining any contact invariant in Heegaard Floer homology.

As we briefly explained in \Section~\ref{sec:introduction}, we will also define a surgery procedure which we call
\emph{joining}, and which generalizes this gluing operation. We will
associate a map on sutured Floer homology to such a surgery in \Section~\ref{sec:join_def_geometric}. 

First we define some preliminary notions.

\begin{defn}
	\label{def:mirror}
	The \emph{mirror} of a partially sutured manifold $\W=(W,\Gamma,\F)$, where $\F=(F,\Lambda)$ is $-\W=(-W,\Gamma,\ol\F)$. Alternatively, it is a
	partially sutured manifold $(W',\Gamma',\F')$, with an orientation reversing diffeomorphism $\varphi\co W\to W'$,
	such that:
	\begin{itemize}
		\item $F$ is sent to $-F'$ (orientation is reversed).
		\item $\Gamma$ is sent to $\Gamma'$ (orientation is preserved).
		\item $R_+(\Gamma)$ is sent to $R_-(\Gamma')$, and vice versa.
		\item $S_+(\Lambda)$ is sent to $S_-(\Lambda')$, and vice versa.
	\end{itemize}
\end{defn}

Whenever we talk about a pair of mirrors, we will implicitly assume that a specific diffeomorphism between them has been
chosen. An example is shown in \Figure~\ref{fig:mirror_example}.

\begin{figure}
	\labellist
	\small\hair=-1pt
	\pinlabel $\W=(W,\Gamma,\F)$ [b] at 84 0
	\pinlabel $-\W=(-W,\Gamma,\ol\F)$ [b] at 304 0
	\pinlabel $\F$ at 84 264
	\pinlabel $\ol\F$ at 304 264
	\hair=1.5pt
	\pinlabel $\boldsymbol{-}$ [b] at 25 159
	\pinlabel $\boldsymbol{+}$ [b] at 50 158
	\pinlabel $\boldsymbol{-}$ [b] at 121 154
	\pinlabel $\boldsymbol{+}$ [b] at 159 165
	\pinlabel $\boldsymbol{-}$ [b] at 367 159
	\pinlabel $\boldsymbol{+}$ [b] at 344 158
	\pinlabel $\boldsymbol{-}$ [b] at 270 154
	\pinlabel $\boldsymbol{+}$ [b] at 233 165
	\endlabellist
	\includegraphics[scale=.5]{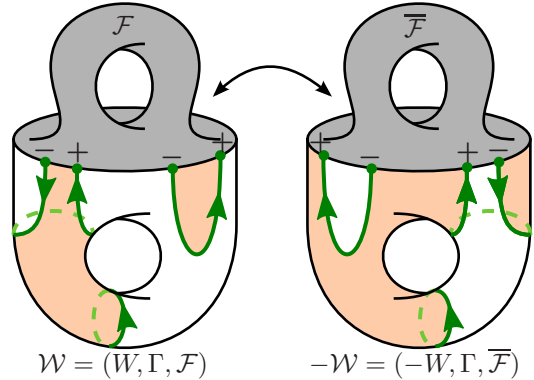}
	\caption{A partially sutured manifold $\W$ and its mirror $-\W$.}
	\label{fig:mirror_example}
\end{figure}

There are two partially sutured manifolds, which will play an important role.

\begin{defn}
	\label{def:twisting_slice}
	A \emph{positive} (respectively \emph{negative}) \emph{twisting slice along the sutured surface $\F=(F,\Lambda)$} is
	the partially sutured manifold $\TW_{\F,\pm}=(F\times[0,1],\Gamma,-\F\cup-\ol\F)$ where we identify $-\F$ with
	$F\times\{0\}$, and $-\ol\F$ with $F\times\{1\}$. The dividing set $\Gamma$ is obtained from $\Lambda\times[0,1]$
	by applying $\frac{1}{n}$--th of a positive (respectively negative) Dehn twist along each component of $\del
	F\times\{\frac{1}{2}\}$, containing $n$ points of $\Lambda$. (The twists might be different for different
	components.)
\end{defn}

Examples of twisting slices are shown
in \Figure~\ref{fig:twists}.

\begin{figure}
	\begin{subfigure}[b]{.48\linewidth}
		\centering
		\labellist
		\hair=2.5pt
		\pinlabel $\F$ [b] at 12 244
		\pinlabel $\ol\F$ [b] at 184 244
		\pinlabel $[0,1]$ [b] at 96 228
		\endlabellist
		\includegraphics[scale=.4]{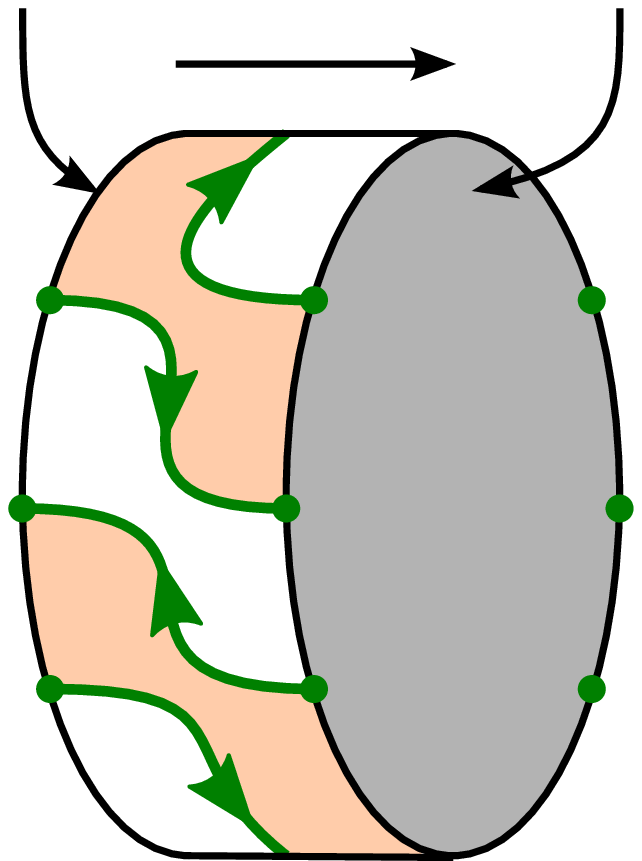}
		\caption{$\TW_{\F,+}$}
		\label{subfig:positive_twist}
	\end{subfigure}
	\begin{subfigure}[b]{.48\linewidth}
		\centering
		\labellist
		\hair=2.5pt
		\pinlabel $\F$ [b] at 12 244
		\pinlabel $\ol\F$ [b] at 184 244
		\pinlabel $[0,1]$ [b] at 96 228
		\endlabellist
		\includegraphics[scale=.4]{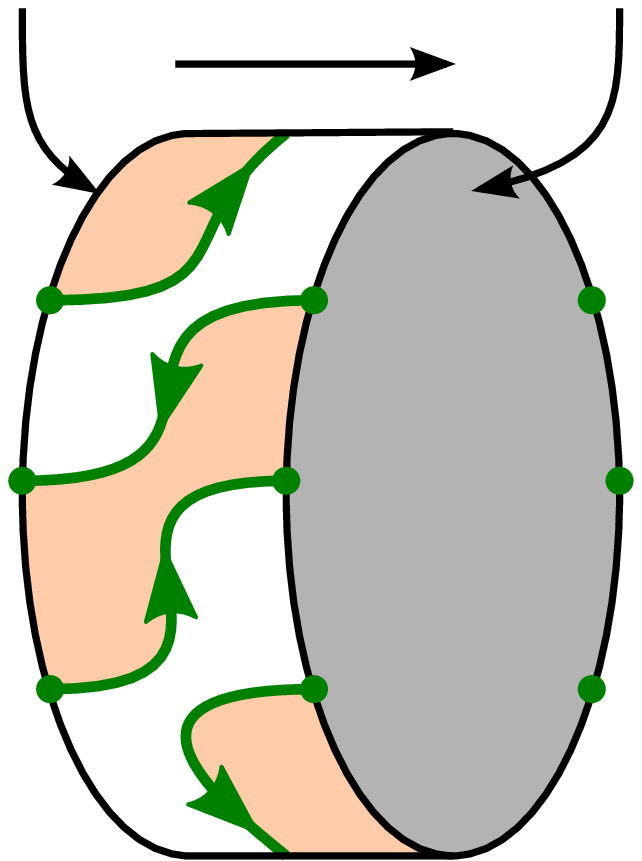}
		\caption{$\TW_{\F,-}$}
		\label{subfig:negative_twist}
	\end{subfigure}
	\caption{Positive and negative twisting slices. The dividing sets are $\Lambda\times[0,1]$, after a fractional Dehn twist has
	been applied. The $R_+$ regions have been shaded.}
	\label{fig:twists}
\end{figure}

\begin{defn}
	\label{def:def_join}
	Let $\Y_1$ and $\Y_2$ be sutured manifolds, and $\W=(W,\Gamma,-\F)$ be partially sutured. Suppose there are
	embeddings $\W\into\Y_1$ and $-\W\into\Y_2$. We will call the new sutured manifold
	\begin{equation*}
		\Y_1~\Cup_{\W}~\Y_2 = (\Y_1\setminus\W)~\cup_{\F}~\TW_{\F,+}~\cup_{-\ol\F}~(\Y_2\setminus-\W)
	\end{equation*}
	the \emph{join of $\Y_1$ and $\Y_2$ along $\W$}.
\end{defn}

Intuitively, this means that we cut out $\W$ and $-\W$ and concatenate the complements together. There is a mismatch of
$R_+$ with $R_-$ along the boundary, so we introduce a positive twist to fix it. An example of gluing was shown
in \Figure~\ref{fig:intro_join_example}.

Another important operation is gluing.

\begin{defn}
	\label{def:gluing}
	Suppose that $\Y_1=(Y_1,\Gamma_1,\F)$ and $\Y_2=(Y_2,\Gamma_2,\ol\F)$ are two partially sutured manifolds, and $\Gamma_0$ is a
	dividing set for $\F=(F,\Lambda)$. We define the \emph{gluing} of the sutured manifolds $(Y_1,\Gamma_1\cup_\Lambda\Gamma_0)$ and
	$(Y_2,\Gamma_2\cup_{\Lambda}\Gamma_0)$ along $(F,\Gamma_0)$ to be the concatenation
	\begin{equation*}
		\Y_1\cup_{-\F}\TW_{\F,+}\cup_{\ol\F}\Y_2,
	\end{equation*}
	and denote it by
	\begin{equation*}
		(Y_1,\Gamma_1\cup\Gamma_0)\cup_{(F,\Gamma_0)}(Y_2,\Gamma_2\cup\Gamma_0).
	\end{equation*}
\end{defn}

An example of gluing was shown in \Figure~\ref{fig:intro_gluing_example}.
It is easy to see that gluing is a special case of the join. Recall that the
concatenation $(Y,\Gamma',\F)\cup_{\F}\W_{\Gamma}$ is the sutured manifold $(Y,\Gamma'\cup\Gamma)$. Thus we can identify
gluing along $(F,\Gamma_0)$ with join along $\W_{\Gamma_0}$. 

Another useful object is the double of a partially sutured manifold.

\begin{defn}
	\label{def:double}
	Given a partially sutured manifold $\W=(W,\Gamma,\F)$, where $\F=(F,\Lambda)$, define the \emph{double} of $\W$ to be 
	the be sutured manifold obtained by concatenation as follows:
	\begin{equation*}
		\Db(\W)=-\W~\cup_{-\ol\F}~\TW_{-\ol\F,-}~\cup_{\F}~\W.
	\end{equation*}
\end{defn}

All the operations we have defined so far keep us in the realm of balanced sutured manifolds.

\begin{prop}
	\label{prop:balanced}
	If we join or glue two balanced sutured manifolds together, the result is balanced. The double of any partially
	sutured manifold $\W$ is balanced.
\end{prop}
\begin{proof}
	There are three key observations. The first one is that $\chi(R_+)-\chi(R_-)$ is additive under
	concatenation. The second is that when passing from $\W$ to its mirror $-\W$, the values of $\chi(R_+)$ and $\chi(R_-)$ are
	interchanged. Finally, for positive and negative twisting slices $\chi(R_+)=\chi(R_-)$.
\end{proof}

The operations of joining and gluing sutured manifolds have good formal properties described in the following proposition.

\begin{prop}
	\label{prop:join_properties} The join satisfies the following:
	\begin{enumerate}
		\item Commutativity: $\Y_1\Cup_{\W}\Y_2$ is canonically diffeomorphic to $\Y_2\Cup_{-\W}\Y_1$.
		\item Associativity: If there are embeddings $\W\into\Y_1$, $(-\W\sqcup\W')\into\Y_2$, and $-\W'\into\W_3$
			then there are canonical diffeomorphisms
			\begin{align*}
				(\Y_1~\Cup_{\W}~\Y_2)~\Cup_{\W'}~\Y_3 
				& \cong \Y_1~\Cup_{\W}~(\Y_2~\Cup_{\W'}~\Y_3)\\
				& \cong (\Y_1~\sqcup~\Y_3)~\Cup_{\W\cup-\W'}~\Y_2.
			\end{align*}
		\item Identity: $\Y~\Cup_{\W}~\D(\W)\cong\Y$.
	\end{enumerate}
	Gluing satisfies analogous properties.
\end{prop}
\begin{proof}
	These facts follow immediately from the definitions.
\end{proof}

\section{Bordered sutured Floer homology}
\label{sec:bordered}

We recall the definitions of bordered sutured manifolds and their invariants, as introduced in~\cite{Zar:BSFH}.

\subsection{Arc diagrams and bordered sutured manifolds}
\label{sec:arc_diagrams}

Parametrizations by arc diagrams, as described below are a slight generalization of those originally defined in~\cite{Zar:BSFH}. The latter
corresponded to parametrizations using only $\alpha$--arcs. While this is sufficient to define invariants for all
possible situations, it is
somewhat restrictive computationally. Indeed, to define the join map $\Psi$ we need to exploit some symmetries that are
not apparent unless we also allow parametrizations using $\beta$--arcs.

\begin{defn}
	\label{def:arc_diagram}
	An \emph{arc diagram of rank $k$} is a triple $\Z=(\ZZZ,\aaa,M)$ consisting of the following:
	\begin{itemize}
		\item A finite collection $\ZZZ$ of oriented arcs.
		\item A collection of points $\aaa=\{a_1,\ldots,a_{2k}\}\subset\ZZZ$.
		\item A 2--to--1 matching $M\co\aaa\to\{1,\ldots,k\}$ of the points into pairs. 
		\item A type: ``$\alpha$'' or ``$\beta$''.
	\end{itemize}
	
	We require that the 1--manifold obtained by performing surgery on all the 0--spheres $M^{-1}(i)$ in $\ZZZ$ has no closed components.
\end{defn}

We represent arc diagrams graphically by a graph $G(\Z)$, which consists of the arcs $\ZZZ$, oriented upwards, and an arc $e_i$ attached at
the pair $M^{-1}(i)\in\ZZZ$, for $i=1,\ldots,k$. Depending on whether the diagram is of $\alpha$ or $\beta$ type, we draw the arcs to the right or to the left,
respectively.

\begin{defn}
	\label{def:surface_from_diagram}
	The \emph{sutured surface $\F(\Z)=(F(\Z),\Lambda(\Z))$ associated to the $\alpha$--arc diagram $\Z$} is constructed in the following
	way. The underlying surface $F$ is produced from the product $\ZZZ\times[0,1]$ by attaching 1--handles along the 0--spheres
	$M^{-1}(i)\times\{0\}$, for $i=1,\ldots,k$. The sutures are $\Lambda=\del\ZZZ\times\{1/2\}$, with the positive region $S_+$ being
	``above'', i.e. containing $\ZZZ\times\{1\}$.
\end{defn}

The \emph{sutured surface associated to a $\beta$--arc diagram} is constructed in the same fashion, except that the
1--handles are attached ``on top'', i.e. at $\M^{-1}(i)\times\{1\}$. The positive region $S_+$ is still above.

Suppose $F$ is a surface with boundary, $G(\Z)$ is properly embedded in $F$, and $\Lambda=\del
G(\Z)\subset\del F$ are the vertices of valence $1$. If $F$ deformation retracts onto $G(\Z)$, we can identify
$(F,\Lambda)$ with $\F(\Z)$. In fact, the embedding uniquely determines such an identification, up to isotopies fixing the
boundary. We say that \emph{$\Z$ parametrizes $(F,\Lambda)$}.

As mentioned earlier, all arc diagrams considered in~\cite{Zar:BSFH} are of $\alpha$--type.

Let $\Z=(\ZZZ,\aaa,M)$ be an arc diagram. We will denote by $-\Z$ the diagram obtained by reversing the orientation of $\ZZZ$ (and
preserving the type). We will denote by $\Zbar$ the diagram obtained by switching the type---from $\alpha$ to $\beta$, or vice versa---and
preserving the triple $(\ZZZ,\aaa,M)$. There are now four related diagrams: $\Z$, $-\Z$, $\Zbar$, and $-\Zbar$. The notation is
intentionally similar to the one for the variations on a sutured surface. Indeed, they are related as follows:
\begin{align*}
	\F(-\Z)&=-\F(\Z),&\F(\ol{\Z})&=\ol{\F(\Z)}.
\end{align*}

To illustrate that, \Figure~\ref{fig:arc_diags} has four variations of an arc diagram of rank $3$.
\Figure~\ref{fig:parametrized_surfaces} shows the corresponding parametrizations of sutured surfaces, which are all tori
with one boundary component and four sutures. Notice the embedding of the graph in each case.

\begin{figure}
	\begin{subfigure}[t]{.24\linewidth}
		\centering
		\labellist
		\small\hair 3pt
		\pinlabel $e_1$ [l] at 36 48
		\pinlabel $e_2$ [l] at 36 88
		\pinlabel $e_3$ [l] at 36 112
		\endlabellist
		\includegraphics[scale=.7]{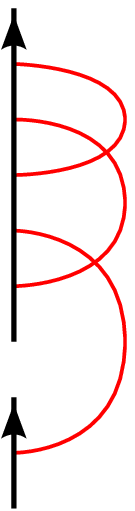}
		\caption{$\Z$ of $\alpha$-type}
		\label{subfig:arc_diags_1}
	\end{subfigure}
	\begin{subfigure}[t]{.24\linewidth}
		\centering
		\labellist
		\small\hair 3pt
		\pinlabel $e_1$ [l] at 36 96
		\pinlabel $e_2$ [l] at 36 56
		\pinlabel $e_3$ [l] at 36 32
		\endlabellist
		\includegraphics[scale=.7]{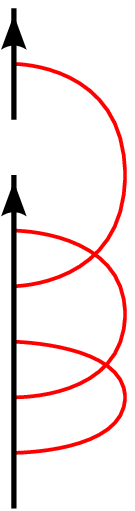}
		\caption{$-\Z$ of $\alpha$-type}
		\label{subfig:arc_diags_2}
	\end{subfigure}
	\begin{subfigure}[t]{.24\linewidth}
		\centering
		\labellist
		\small\hair 3pt
		\pinlabel $e_1$ [r] at 12 48
		\pinlabel $e_2$ [r] at 12 88
		\pinlabel $e_3$ [r] at 12 112
		\endlabellist
		\includegraphics[scale=.7]{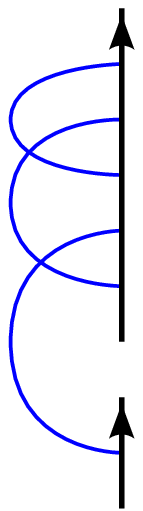}
		\caption{$\Zbar$ of $\beta$-type}
		\label{subfig:arc_diags_3}
	\end{subfigure}
	\begin{subfigure}[t]{.24\linewidth}
		\centering
		\labellist
		\small\hair 3pt
		\pinlabel $e_1$ [r] at 12 96
		\pinlabel $e_2$ [r] at 12 56
		\pinlabel $e_3$ [r] at 12 32
		\endlabellist
		\includegraphics[scale=.7]{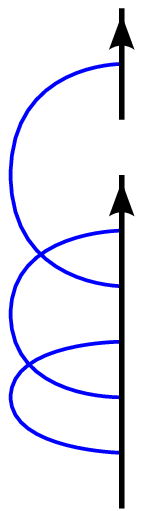}
		\caption{$-\Zbar$ of $\beta$-type}
		\label{subfig:arc_diags_4}
	\end{subfigure}
	\caption{Four variants of an arc diagram}
	\label{fig:arc_diags}
\end{figure}
\begin{figure}
	\begin{subfigure}[b]{.24\linewidth}
		\centering
		\labellist
		\small\hair=0pt
		\pinlabel $S_+$ [r] at 25 176
		\pinlabel $S_+$ [r] at 25 50
		\pinlabel $S_-$ [l] at 150 132
		\tiny
		\hair=0.5pt
		\pinlabel $+$ [t] at 55 19
		\pinlabel $+$ [t] at 55 110
		\pinlabel $-$ [b] at 55 83
		\pinlabel $-$ [b] at 55 243
		\endlabellist
		\includegraphics[scale=.45]{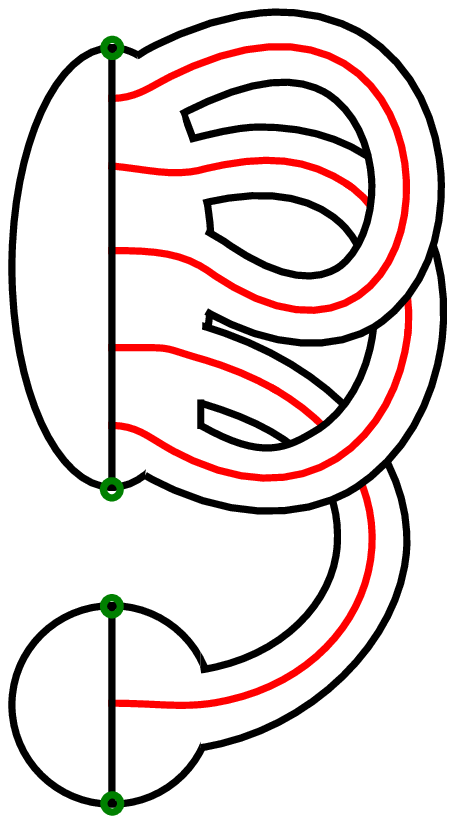}
		\caption{$\F(\Z)$}
		\label{subfig:paramatrized_surface_1}
	\end{subfigure}
	\begin{subfigure}[b]{.24\linewidth}
		\centering
		\labellist
		\small\hair=0pt
		\pinlabel $S_+$ [r] at 25 88
		\pinlabel $S_+$ [r] at 25 214
		\pinlabel $S_-$ [l] at 150 132
		\tiny
		\hair=0.5pt
		\pinlabel $+$ [t] at 55 20
		\pinlabel $+$ [t] at 55 181
		\pinlabel $-$ [b] at 55 154
		\pinlabel $-$ [b] at 55 245
		\endlabellist
		\includegraphics[scale=.45]{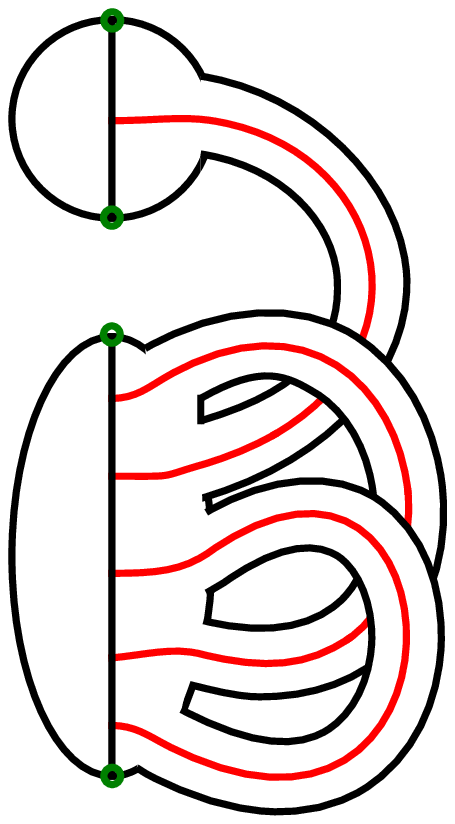}
		\caption{$\F(-\Z)$}
		\label{subfig:paramatrized_surface_2}
	\end{subfigure}
	\begin{subfigure}[b]{.24\linewidth}
		\centering
		\labellist
		\small\hair=0pt
		\pinlabel $S_-$ [l] at 155 176
		\pinlabel $S_-$ [l] at 155 50
		\pinlabel $S_+$ [r] at 30 132
		\tiny
		\hair=0.5pt
		\pinlabel $+$ [t] at 125 19
		\pinlabel $+$ [t] at 125 110
		\pinlabel $-$ [b] at 125 83
		\pinlabel $-$ [b] at 125 243
		\endlabellist
		\includegraphics[scale=.45]{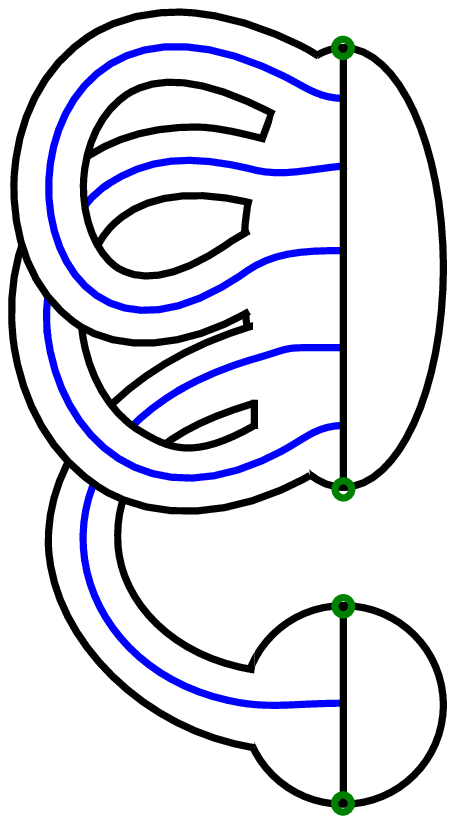}
		\caption{$\F(\ol\Z)$}
		\label{subfig:paramatrized_surface_3}
	\end{subfigure}
	\begin{subfigure}[b]{.24\linewidth}
		\centering
		\labellist
		\small\hair=0pt
		\pinlabel $S_-$ [l] at 155 88
		\pinlabel $S_-$ [l] at 155 214
		\pinlabel $S_+$ [r] at 30 132
		\tiny
		\hair=0.5pt
		\pinlabel $+$ [t] at 125 20
		\pinlabel $+$ [t] at 125 181
		\pinlabel $-$ [b] at 125 154
		\pinlabel $-$ [b] at 125 245
		\endlabellist
		\includegraphics[scale=.45]{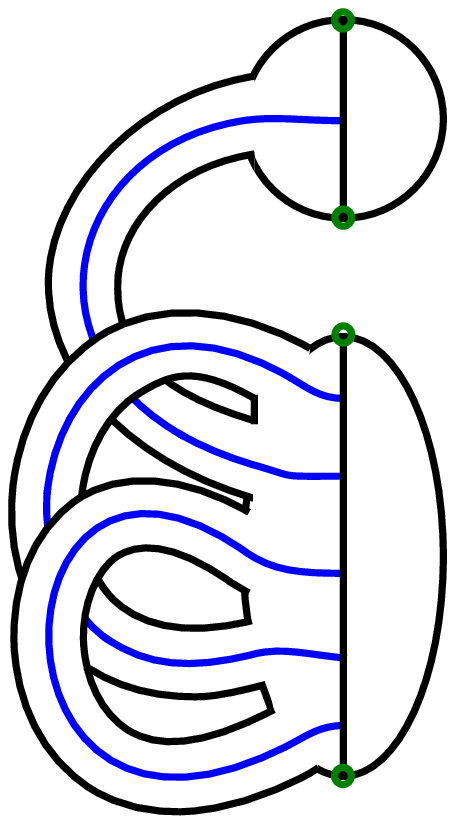}
		\caption{$\F(-\ol\Z)$}
		\label{subfig:paramatrized_surface_4}
	\end{subfigure}
	\caption{Parametrizations of surfaces by the arc diagrams in \Figure~\ref{fig:arc_diags}}
	\label{fig:parametrized_surfaces}
\end{figure}

\begin{defn}
	\label{def:bs_manifold}
	A \emph{bordered sutured manifold} $\Y=(Y,\Gamma,\Z)$ is a partially sutured manifold $(Y,\Gamma,\F)$, whose sutured surface $\F$
	has been parametrized by the arc diagram $\Z$.
\end{defn}

As with partially sutured manifolds, $\Y=(Y,\Gamma,\Z_1\sqcup\Z_2)$ can be thought of as a cobordism from
$\F(-\Z_1)$ to $\F(\Z_2)$.

\subsection{The bordered algebra}
\label{sec:algebra_def}

We will briefly recall the definition of the algebra $\A(\Z)$ associated to an $\alpha$--type arc
diagram $\Z$. Fix a diagram $\Z=(\ZZZ,\aaa,M)$ of rank $k$. First, we define a larger \emph{strands algebra}
$\A'(\ZZZ,\aaa)$,
which is independent of the matching $M$. Then we define $\A(\Z)$ as a subalgebra of $\A'(\ZZZ,\aaa)$.

\begin{defn}
	\label{def:strands_algebra}
	The \emph{strands algebra associated to $(\ZZZ,\aaa)$} is a $\ZZ/2$--algebra $\A'(\ZZZ,\aaa)$, which is generated (as a vector space)
	by diagrams in $[0,1]\times\ZZZ$ of the following type. Each diagram consists of several embedded oriented arcs or
	\emph{strands}, starting in $\{0\}\times\aaa$ and ending in $\{1\}\times\aaa$. All tangent vectors on the strands
	should project non-negatively on $\ZZZ$, i.e. they are ``upward-veering''. Only transverse intersections are
	allowed.

	The diagrams are subjects to two relations---any two diagrams related by a Reidemeister III move represent the same
	element in $\A'(\ZZZ,\aaa)$, and any diagram in which two strands intersect more than once represents zero.

	Multiplication is given by concatenation of diagrams in the $[0,1]$--direction, provided the endpoints of the
	strands agree. Otherwise the product is zero. The differential of a diagram is the sum of all diagrams obtained
	from it by taking the oriented resolution of a crossing.
\end{defn}

We refer to a strand connecting $(0,a)$ to $(1,a)$ for some $a\in\aaa$ as \emph{horizontal}.
Notice that the idempotent elements of $\A'(\ZZZ,\aaa)$ are precisely those which are sums of diagrams
with only horizontal strands. To recover the information carried by the matching $M$ we single out some of these
idempotents.

\begin{defn}
	\label{def:idempotent_ring}
	The \emph{ground ring $\I(\Z)$ associated to $\Z$} is a ground ring, in the sense of \Definition~\ref{def:ground_ring},
	of rank $2^k$ over $\ZZ/2$, with canonical basis $(\iota_I)_{I\subset\{1,\ldots,k\}}$. It is identified with a
	subring of the strands algebra $\A'(\ZZZ,\aaa)$, by setting $\iota_I=\sum_J D_J$. The sum is over all $J\subset\aaa$
	such that $M|_J\co J\to I$ is a bijection, and $D_J$ is the diagram with horizontal strands $[0,1]\times J$.
\end{defn}

For all $I\subset\{1,\ldots,k\}$, the generator $\iota_I$ is a sum of $2^{\#I}$ diagrams.

\begin{defn}
	\label{def:alpha_algebra}
	The \emph{bordered algebra $\A(\Z)$ associated to $\Z$} is the subalgebra of
	$\I(\Z)\cdot\A'(\ZZZ,\aaa)\cdot\I(\Z)$ consisting of all elements $\alpha$ subject to the following condition.
	Suppose $M(a)=M(b)$, and $D$ and $D'$ are two diagrams, where $D'$ is obtained from $D$ by replacing
	the horizontal arc $[0,1]\times\{a\}$ by the horizontal arc	$[0,1]\times\{b\}$. Then $\alpha$ contains $D$ as a
	summand iff it contains $D'$ as a summand.
\end{defn}

We use $\I(\Z)$ as the ground ring for $\A(\Z)$, in the sense of \Definition~\ref{def:ainf_algebra}. The condition in
\Definition~\ref{def:alpha_algebra} ensures that the
canonical basis elements of $\I(\Z)$ are indecomposable in $\A(\Z)$.

It is straightforward to check that \Definition~\ref{def:alpha_algebra} is equivalent to the definition of $\A(\Z)$
in~\cite{Zar:BSFH}.

Examples of several algebra elements are given in \Figure~\ref{fig:algebra_elements}. The dotted lines on the side
are given to remind us of the matching in the arc diagram $\Z$. All strands are oriented left to right, so we avoid
drawing them with arrows. The horizontal lines in \Figure~\ref{subfig:algebra_elements_a2} are dotted, as a shorthand
for the sum of two diagrams, with a single horizontal line each. For the elements in this example, we have
$a_1\cdot a_2=a_3$, and $\del a_1=a_4$.

\begin{figure}
	\begin{subfigure}[t]{.24\linewidth}
		\centering
		\includegraphics[scale=.8]{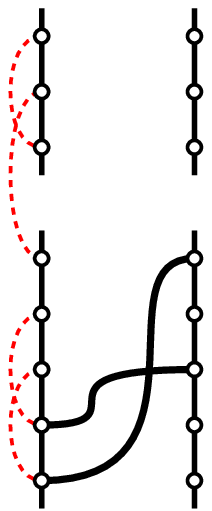}
		\caption{$a_1$}
		\label{subfig:algebra_elements_a1}
	\end{subfigure}
	\begin{subfigure}[t]{.24\linewidth}
		\centering
		\includegraphics[scale=.8]{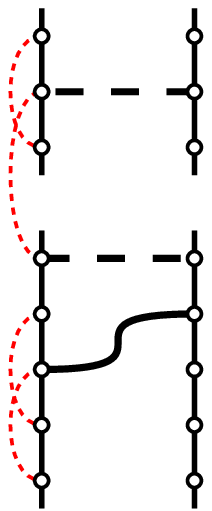}
		\caption{$a_2$}
		\label{subfig:algebra_elements_a2}
	\end{subfigure}
	\begin{subfigure}[t]{.24\linewidth}
		\centering
		\includegraphics[scale=.8]{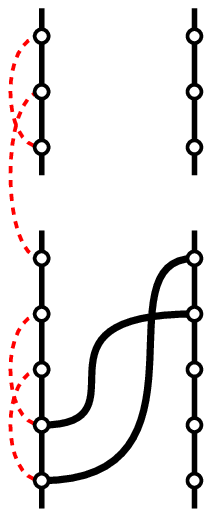}
		\caption{$a_3$}
		\label{subfig:algebra_elements_a3}
	\end{subfigure}
	\begin{subfigure}[t]{.24\linewidth}
		\centering
		\includegraphics[scale=.8]{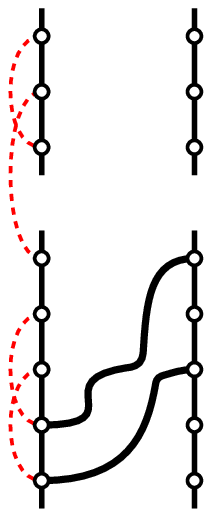}
		\caption{$a_4$}
		\label{subfig:algebra_elements_a4}
	\end{subfigure}
	\caption{Four generators of $\A(\Z)$.}
	\label{fig:algebra_elements}
\end{figure}

The situation for arc diagrams of $\beta$--type is completely analogous, with one important difference.

\begin{defn}
	\label{def:beta_algebra}
	The \emph{bordered algebra $\A(\Z)$ associated to a $\beta$--arc diagram $\Z$},
	is defined in the exact same way as in \Definitions~\ref{def:alpha_algebra},
	except that moving strands are downward veering, instead of upward.
\end{defn}

The relationship between the different types of algebras is summarized in the following proposition.

\begin{prop}
	\label{prop:algebra_relationships}
	Suppose $\Z$ is an arc diagram of either $\alpha$ or $\beta$--type. The algebras associated
	to $\Z$, $-\Z$, $\Zbar$, and $-\Zbar$ are related as follows:
	\begin{align*}
		\A(-\Z)&\cong\A(\ol\Z)\cong\A(\Z)^{\op},\\
		\A(-\ol\Z)&\cong\A(\Z).
	\end{align*}
\end{prop}

Here $A^{\op}$ denotes the opposite algebra of $A$. That is, an algebra with the same additive structure and
differential, but the order of multiplication reversed.

\begin{proof}
	This is easily seen by reflecting and rotating diagrams. To get from $\A(\Z)$ to $\A(-\Z)$ we have to rotate all
	diagrams by 180 degrees. This means that multiplication switches order, so we get the opposite algebra.

	To get from $\A(\Z)$ to $\A(\ol\Z)$ we have to reflect all diagrams along the vertical axis. This again means that
	multiplication switches order.
	
	An example of the correspondence is shown in \Figure~\ref{fig:rotated_elements}.
\end{proof}

\begin{figure}
	\begin{subfigure}[t]{.24\linewidth}
		\centering
		\includegraphics[scale=.8]{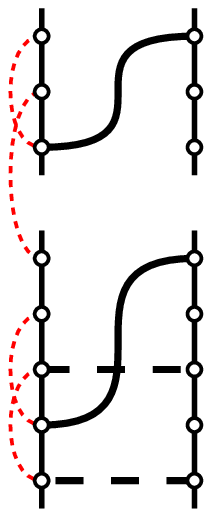}
		\caption{$\A(\Z)$}
		\label{subfig:rotated_Z}
	\end{subfigure}
	\begin{subfigure}[t]{.24\linewidth}
		\centering
		\includegraphics[scale=.8]{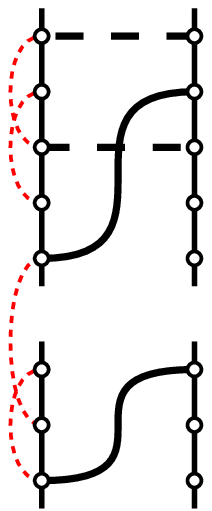}
		\caption{$\A(-\Z)$}
		\label{subfig:rotated_-Z}
	\end{subfigure}
	\begin{subfigure}[t]{.24\linewidth}
		\centering
		\includegraphics[scale=.8]{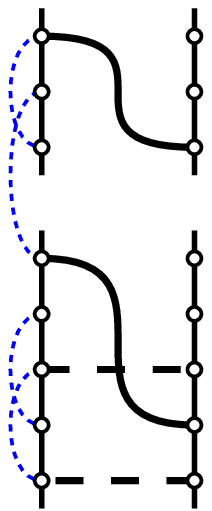}
		\caption{$\A(\ol\Z)$}
		\label{subfig:rotated_Zbar}
	\end{subfigure}
	\begin{subfigure}[t]{.24\linewidth}
		\centering
		\includegraphics[scale=.8]{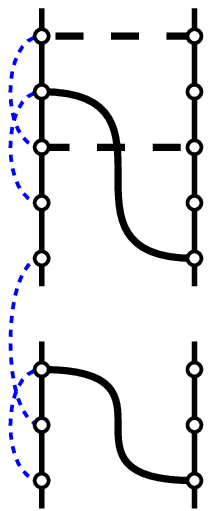}
		\caption{$\A(-\ol\Z)$}
		\label{subfig:rotated_-Zbar}
	\end{subfigure}
	\caption{Four elements in the algebras $\Z$, $-\Z$, $\ol\Z$, and $-\ol\Z$, which correspond to each other.}
	\label{fig:rotated_elements}
\end{figure}

\subsection{The bordered invariants}
\label{sec:bordered_invariants_definition}

We will give a brief sketch of the definitions of the bordered invariants from~\cite{Zar:BSFH}, which apply for 
the case of $\alpha$--arc diagrams. Then we discuss the necessary modifications when $\beta$--arcs are involved.

For now assume $\Z=(\ZZZ,\aaa,M)$ is an $\alpha$--arc diagram.

\begin{defn}
	\label{def:heegaard_diagram}
	A \emph{bordered sutured Heegaard diagram} $\HH=(\Sigma,\balpha,\bbeta,\Z)$ consists of the following:
	\begin{itemize}
		\item A compact surface $\Sigma$ with no closed components.
		\item A collection of circles $\balpha^c$ and a collection of arcs $\balpha^a$, which are pairwise disjoint and
			properly embedded in $\Sigma$. We set $\balpha=\balpha^a\cup\balpha^c$.
		\item A collection of disjoint circles $\bbeta$, properly embedded in $\Sigma$.
		\item An embedding $G(\Z)\into\Sigma$, such that $\ZZZ$ is sent into $\del\Sigma$,
			preserving orientation, while $\balpha^a$ is the image of the arcs $e_i$ in $G(\Z)$.
	\end{itemize}
	
	We require that $\pi_0(\del\Sigma\setminus\ZZZ)\to\pi_0(\Sigma\setminus(\balpha^c\cup\balpha^a))$ and
	$\pi_0(\del\Sigma\setminus\ZZZ)\to\pi_0(\Sigma\setminus\bbeta)$ be surjective.
\end{defn}

To such a diagram we can associate a bordered sutured manifold $(Y,\Gamma,\Z)$ as follows. We obtain $Y$ from
$\Sigma\times[0,1]$ by
gluing $2$--handles to $\bbeta\times\{1\}$ and $\balpha^c\times\{0\}$. The dividing set is $\Gamma=(\del\Sigma\setminus\ZZZ)\times\{1/2\}$,
and $F(\Z)$ is a neighborhood of $\ZZZ\times[0,1]\cup\balpha^a\times\{0\}$.

As proved in~\cite{Zar:BSFH}, for every bordered sutured manifold there is a unique Heegaard diagram, up to isotopy
and some moves.

The bordered invariants are certain homotopy-equivalence classes of $\Ainf$--modules (see \Appendix~\ref{sec:algebra}).
For a given Heegaard diagram $\HH$, we can form the set of \emph{generators}
$\G(\HH)$ consisting of collections of intersection points of $\balpha\cap\bbeta$.

The invariant $\BSA(\HH)_{\A(\Z)} $ is a right type--$A$ $\Ainf$--module over $\A(\Z)$, with $\ZZ/2$--basis $\G(\HH)$. The
ground ring $\I(\Z)$ acts as follows. The only idempotent in $\I(\Z)$ which acts nontrivially on $\xxx\in\G(\HH)$ is
$\iota_{I(\xxx)}$ where $I(\xxx)\subset\{1,\ldots,k\}$ records the $\alpha$--arcs which contain a point of $\xxx$.

The structure map $m$ of $\BSA(\HH)$ counts certain holomorphic curves in $\Int\Sigma\times[0,1]\times\RR$, with
boundary on $(\balpha\times\{1\}\times\RR)\cup(\bbeta\times\{0\}\times\RR$). Each such curve
has two types of asymptotics---ends at $(\balpha\cap\bbeta)\times[0,1]\times\pm\infty$, and ends at
$\del\Sigma\times\{0\}\times\{h\}$ where $h\in\RR$ is finite. The possible ends at $\del\Sigma$ are in 1-to-1
correspondence with elements of $\A(\Z)$.

The expression $\left<m(\xxx,a_1,\ldots,a_n),\yyy^\vee\right>$ counts curves as above, which have asymptotics
$\xxx\times[0,1]$ at $-\infty$, $\yyy\times[0,1]$ at $+\infty$, and $a_1, a_2,\ldots,a_n$ at some finite values
$h_1<h_2<\ldots<h_n$.

We write $\BSA(\Y)$ for the homotopy equivalence class of $\BSA(\HH)$. (Invariance was proven
in~\cite{Zar:BSFH}.)

The invariant $\lu{\A(-\Z)}\BSD(\HH)$ is a left type--$D$ $\Ainf$--module over $\A(-\Z)={\A(\Z)}^{\op}$,
with $\ZZ/2$--basis $\G(\HH)$. (See \Appendix~\ref{sec:algebras_and_modules} for type--$D$ modules, and the meaning of
upper and lower indices). The
ground ring $\I(-\Z)$ acts as follows. The only idempotent in $\I(-\Z)$ which acts nontrivially on $\xxx\in\G(\HH)$ is
$\iota_{I^c(\xxx)}$ where $I^c(\xxx)\subset\{1,\ldots,k\}$ records the $\alpha$--arcs which \emph{do not} contain a point of $\xxx$.

The structure map $\delta$ of $\BSD(\HH)$ counts a subset of the same holomorphic curves as for $\BSA(\HH)$.
Their interpretation is somewhat different, though. Equivalently,
$\lu{\A(\Z)^{\op}}\BSD(\HH)=\BSA(\HH)_{\A(\Z)}\sqtens\lu{\A(\Z),{\A(\Z)}^{\op}}\II$, where $\II$ is a certain bimodule
defined in~\cite{LOT:bimodules}.

Again, we write $\BSD(\Y)$ for the homotopy equivalence class of $\BSD(\HH)$. (Invariance was proven
in~\cite{Zar:BSFH}.)

We can also construct invariants $\li{_{{\A(\Z)}^{\op}}}\BSA(\Y)$ and $\BSD(\Y)^{\A(\Z)}$ purely algebraically from the
usual $\BSA$ and $\BSD$. Indeed, as discussed in \Appendix~\ref{sec:duals}, any right $A$--module is a left--$A^{\op}$ module and vice versa.

If $\Y$ is bordered by $\F(\Z_1)\sqcup\F(\Z_2)$, we can similarly define several bimodules invariants for $\Y$:
\begin{align*}
	&\li{_{\A(\Z_1)^{\op}}}\BSAA(\Y)_{\A(\Z_2)} &
	&\lu{\A(\Z_1)^{\op}}\BSDA(\Y)_{\A(\Z_2)} \\
	&\li{_{\A(\Z_1)^{\op}}}\BSAD(\Y)^{\A(\Z_2)} &
	&\lu{\A(\Z_1)^{\op}}\BSDD(\Y)^{\A(\Z_2)}
\end{align*}

For the invariants of $\beta$--diagrams little changes. Suppose $\Z$ is a $\beta$--type arc diagram. Heegaard diagrams
will now involve $\beta$--arcs as the images of $e_i\subset G(\Z)$, instead of $\alpha$--arcs. We still count holomorphic curves in
$\Int\Sigma\times[0,1]\times\RR$. However, since there are $\beta$--curves hitting $\del\Sigma$ instead of $\alpha$,
the asymptotic ends at $\del\Sigma\times\{1\}\times\{h\}$ are replaced by ends at $\del\Sigma\times\{0\}\times\{h\}$, which
again correspond to elements of $\A(\Z)$. The rest of the definition is essentially unchanged.

The last case is when $\Y$ is bordered by $\F(\Z_1)\sqcup\F(\Z_2)$, where $\Z_1$ is a diagram of $\alpha$--type and
$\Z_2$ is of $\beta$--type. We can extend the definition of $\BSAA(\Y)$ as before. There are now four types of asymptotic ends:
\begin{itemize}
	\item The ones at $\pm\infty$ which correspond to generators $\xxx,\yyy\in\G(\HH)$.
	\item $\del\Sigma\times\{1\}\times\{h\}$ (or \emph{$\alpha$--ends}) which correspond to $\A(\Z_1)$.
	\item $\del\Sigma\times\{0\}\times\{h\}$ (or \emph{$\beta$--ends}) which correspond to $\A(\Z_2)$.
\end{itemize}
Each holomorphic curve will have some number $k\geq 0$ of $\alpha$--ends, and some number $l\geq 0$ of $\beta$--ends.
Such a curve contributes to the structure map $m_{k|1|l}$ which takes $k$ elements of $\A(\Z_1)$ and $l$ elements of $\A(\Z_2)$.

To summarize we have the following theorem.

\begin{thm}
	\label{thm:bimodules_all}
	Let $\Y$ be a bordered sutured manifold, bordered by $-\F(\Z_1)\sqcup\F(\Z_2)$, where $\Z_1$ and $\Z_2$ can be any
	combination of $\alpha$ and $\beta$ types. Then there are bimodules, well defined up to homotopy equivalence:
	\begin{align*}
		&\li{_{\A(\Z_1)}}\BSAA(\Y)_{\A(\Z_2)} &
		&\lu{\A(\Z_1)}\BSDA(\Y)_{\A(\Z_2)} \\
		&\li{_{\A(\Z_1)}}\BSAD(\Y)^{\A(\Z_2)} &
		&\lu{\A(\Z_1)}\BSDD(\Y)^{\A(\Z_2)}
	\end{align*}

	If $\Y_1$ and $\Y_2$ are two such manifolds, bordered by $-\F(\Z_1)\sqcup\F(\Z_2)$ and $-\F(\Z_2)\sqcup\F(\Z_3)$,
	respectively, then there are homotopy equivalences
	\begin{align*}
		\BSAA(\Y_1\cup\Y_2)&\simeq\BSAA(\Y_1)\sqtens_{\A(\Z_2)}\BSDA(\Y_2),\\
		\BSDA(\Y_1\cup\Y_2)&\simeq\BSDD(\Y_1)\sqtens_{\A(\Z_2)}\BSAA(\Y_2),
	\end{align*}
	etc. Any combination of bimodules for $\Y_1$ and $\Y_2$ can be used, where one is type--$A$ for
	$\A(\Z_2)$, and the other is type--$D$ for $\A(\Z_2)$.
\end{thm}

The latter statement is referred to as the \emph{pairing theorem}. The proof of \Proposition~\ref{thm:bimodules_all} is
a straightforward adaptation of the corresponding proofs when dealing with only type--$\alpha$ diagrams. An analogous construction
involving both $\alpha$ and $\beta$ arcs in the purely bordered setting is given in~\cite{LOT:morphism_spaces}.

\subsection{Mirrors and twisting slices}
\label{sec:BSA_computations}

In this section we give two computations of bordered invariants. One of them relates the invariants for a bordered
sutured manifold $\W$ and its mirror $-\W$. The other gives the invariants for a positive and negative twisting slice.

Recall that if $\W=(W,\Gamma,\F(\Z))$, its mirror is $-\W=(-W,\Gamma,\ol{\F(\Z)})=(-W,\Gamma,\F(\ol\Z))$.

\begin{prop}
	\label{prop:mirror_is_dual}
	Let $\W$ and $-\W$ be as above. Let $M_{\A(\Z)}$ be a representative for the homotopy equivalence class
	$\BSA(\W)_{\A(\Z)}$. Then its dual $\li{_{\A(\Z)}}M^\vee$ is a representative for $\li{_{\A(\Z)}}\BSA(-\W)$.
	Similarly, there are homotopy equivalences
	\begin{align*}
		\left(\BSD(\W)^{\A(\Z)}\right)^\vee&\simeq\lu{\A(\Z)}\BSD(-\W),\\
		\left(\li{_{{\A(\Z)}^{\op}}}\BSA(\W)\right)^\vee&\simeq\BSA(-\W)_{{\A(\Z)}^{\op}},\\
		\left(\lu{{\A(\Z)}^{\op}}\BSD(\W)\right)^\vee&\simeq\BSD(-\W)^{{\A(\Z)}^{\op}}.
	\end{align*}

	A similar statement holds for bimodules---if $\W$ is bordered by $\F(\Z_1)\sqcup\F(\Z_2)$, then the corresponding
	bimodule invariants of $\W$ and $-\W$ are duals of each other.
\end{prop}
\begin{proof}
	We prove one case. All others follow by analogy. Let $\HH=(\Sigma,\balpha,\bbeta,\Z)$ be a Heegaard
	diagram for $\W$. Let $\HH'=(\Sigma,\bbeta,\balpha,\ol\Z)$ be the diagram obtained by switching all $\alpha$ and
	$\beta$ curves. (Note that if $\Z$ was an $\alpha$--type diagram, this turns it into the $\beta$--type diagram
	$\ol\Z$, and vice versa.)

	The bordered sutured manifold described by $\HH'$ is precisely $-\W$. Indeed, it is obtained from the same manifold
	$\Sigma\times[0,1]$ by attaching all $2$--handles on the opposite side, and taking the sutured surface $\F$ also on
	the opposite side. This is equivalent to reversing the orientation of $W$, while keeping the orientations of
	$\Gamma\subset\del\Sigma$ and $\ZZZ\subset\del\Sigma$ the same. (Compare to~\cite{HKM:EH}, where the
	$EH$--invariant for contact structures on $(Y,\Gamma)$ is defined in $\SFH(-Y,+\Gamma)$.)

	The generators $\G(\HH)$ and $\G(\HH')$ of the two diagrams are the same.
	There is also a 1--to--1 correspondence between the holomorphic curves $u$ in the definition of $\BSA(\HH)_{\A(\Z)}$ and
	the curves $u'$ in the definition of $\BSA(\HH')_{\A(\ol\Z)}$. This is given by reflecting both the $[0,1]$--factor
	and the $\RR$--factor in the domain $\Int\Sigma\times[0,1]\times\RR$. The $\pm\infty$ asymptotic ends are reversed. The
	$\alpha$--ends of $u$ are sent to the $\beta$--ends of $u'$, and vice versa, while their heights $h$ on the
	$\RR$--scale are reversed. When turning $\alpha$--ends to $\beta$--ends, the corresponding elements of $\A(\Z)$ are
	reflected (as in the correspondence $\A(\ol\Z)\cong{\A(\Z)}^{\op}$ from \Proposition~\ref{prop:algebra_relationships}).

	This implies the following relation between the structure maps $m$ of $\BSA(\HH)$ and $m'$ of $\BSA(\HH')$:
	\begin{equation*}
		\left<m(\xxx,a_1,\ldots,a_n),\yyy^\vee\right>=\left<m'(\yyy',a_n^{\op},\ldots,a_1^{op}),\xxx'^\vee\right>.
	\end{equation*}

	Turning $\BSA(\HH')$ into a left module over $(\A(\Z)^{\op})^{\op}=\A(\Z)$, we get the relation
	\begin{equation*}
		\left<m(\xxx,a_1,\ldots,a_n),\yyy^\vee\right>=\left<m'(a_1,\ldots,a_n,\yyy'),\xxx'^\vee\right>.
	\end{equation*}

	This is precisely the statement that $\BSA(\HH)_{\A(\Z)}$ and $\li{_{\A(\Z)}}\BSA(\HH')$ are duals, with $\G(\HH)$
	and $\G(\HH')$ as dual bases.
\end{proof}

A similar statement for purely bordered invariants is proven in~\cite{LOT:morphism_spaces}.

\begin{prop}
	\label{prop:twisting_slice_invariants}
	Let $\Z$ be any arc diagram, and let $A=\A(\Z)$. The twisting slices $\TW_{\F(\Z),\pm}$ are bordered
	by $-\F(\Z)\sqcup-\ol{\F(\Z)}$. They have bimodule invariants
	\begin{align*}
		\li{_A}\BSAA(\TW_{\F(\Z),-})_A&\simeq \li{_A}A_A, &
		\li{_A}\BSAA(\TW_{\F(\Z),+})_A&\simeq \li{_A}{A^\vee}_A.
	\end{align*}
\end{prop}
\begin{proof}
	Since $\TW_{\F(\Z),\pm}$ are mirrors of each other, by \Proposition~\ref{prop:mirror_is_dual}, it is enough to prove the first equivalence. The key ingredient is a very
	convenient \emph{nice diagram} $\HH$ for $\TW_{\F(\Z),-}$. This diagram was discovered by the author, and
	independently by Auroux in~\cite{Aur:bordered}, where it appears in a rather different setting.
	
	Recall from~\cite{Zar:BSFH} that a nice diagram is a diagram,
	$(\Sigma,\balpha,\bbeta,\Z)$ where each region of $\Sigma\setminus(\balpha\cup\bbeta)$ is either a boundary region,
	a rectangle, or a bigon. The definition trivially extends to the current more general setting. Nice diagrams can still
	be used to combinatorially compute bordered sutured invariants.

	The diagram is obtained as follows. For concreteness assume that $\Z$ is of $\alpha$--type. To construct the
	Heegaard surface $\Sigma$, start with several
	squares $[0,1]\times[0,1]$, one for each component $Z\in\ZZZ$. There are three identifications of $\ZZZ$ with sides
	of the squares:
	\begin{itemize}
		\item $\varphi$ sending $\ZZZ$ to the ``left sides'' $\{0\}\times[0,1]$, oriented from $0$ to $1$.
		\item $\varphi'$ sending $\ZZZ$ to the ``right sides'' $\{1\}\times[0,1]$, oriented from $1$ to $0$.
		\item $\psi$ sending $\ZZZ$ to the ``top sides'' $[0,1]\times\{1\}$, oriented from $1$ to $0$.
	\end{itemize}

	For each matched pair $\{a,b\}=M^{-1}(i)\subset\aaa\subset\ZZZ$, attach a 1--handle at $\psi(\{a,b\})$. Add an
	$\alpha$--arc $\alpha^a_i$ from $\varphi(a)$ to $\varphi(b)$, and a $\beta$--arc
	$\beta^a_i$ from $\varphi'(a)$ to $\varphi'(b)$, both running through the handle corresponding to ${a,b}$. To see that this gives the correct manifold, notice
	that there are no $\alpha$ or $\beta$--circles, so the manifold is topologically $\Sigma\times[0,1]$. The pattern of
	attachment of the 1--handles shows that $\Sigma=F(\Z)$. It is easy to check that $\Gamma$ and the arcs are in the
	correct positions.

	This construction is demonstrated in \Figure~\ref{fig:identity}. The figure corresponds to the arc diagram $\Z$ from
	\Figure~\ref{subfig:arc_diags_3}.

	Calculations with the same diagram in~\cite{Aur:bordered} and~\cite{LOT:morphism_spaces} show that the bimodule
	$\BSAA(\HH)$ is indeed the algebra $A$ as a bimodule over itself. While the statements in those cases are not about
	bordered sutured Floer homology, the argument is purely combinatorial and caries over completely.

	We give a brief summary of this argument. Intersection points in $\balpha\cap\bbeta$ are of two types:
	\begin{itemize}
		\item $x_i\in\alpha_i^a\cap\beta_i^a$ inside the 1--handle corresponding to
			$M^{-1}(i)$, for $i\in\{1,\ldots,k\}$. The point $x_i$ corresponds to the two horizontal strands
			$[0,1]\times M^{-1}(i)$ in $\A(\Z)$.
		\item $y_{ab}\in\alpha_{M(a)}^a\cap\beta_{M(b)}^a$, inside the square regions of $\HH$. The point $y_{ab}$
			corresponds to a strand $(0,a)\to(1,b)$ (or $a\to b$ for short) in $\A(\Z)$.
	\end{itemize}
	The allowed combinations of intersection points correspond to the allowed diagrams in $\A(\Z)$, so
	$\BSA(\HH)\cong\A(\Z)$ as a $\ZZ/2$--vector space.
	
	Since $\HH$ is a nice diagram the differential counts embedded rectangles
	in $\HH$, with sides on $\balpha$ and $\bbeta$. The rectangle with corners $(y_{ad},y_{bc},y_{ac},y_{ad})$
	corresponds to resolving the crossing between the strands $a\to d$ and $b\to c$ (getting $a\to c$ and $b\to d$).

	The left action $m_{1|1|0}$ of $A$ counts rectangles hitting the $-\Z$--part of the boundary. The rectangle with
	corners $(\varphi(a),y_{ac},y_{bc},\varphi(b))$ corresponds to concatenating the strands $a\to b$ and $b\to c$ (getting
	$a\to c$). The right action is similar, with rectangles hitting the $-\ol\Z$--part of the boundary.

	Some examples of domains in $\HH$ contributing to $m_{0|1|0}$, $m_{1|1|0}$, and $m_{0|1|1}$ are shown in
	\Figure~\ref{fig:identity_domains}. They are for the diagram $\HH$ from \Figure~\ref{fig:identity}.
\end{proof}

\begin{figure}
	\labellist
	\small\hair 2pt
	\pinlabel $-\Z$ [r] at 30 60
	\pinlabel $-\ol\Z$ [l] at 270 60
	\endlabellist
	\includegraphics[scale=.5]{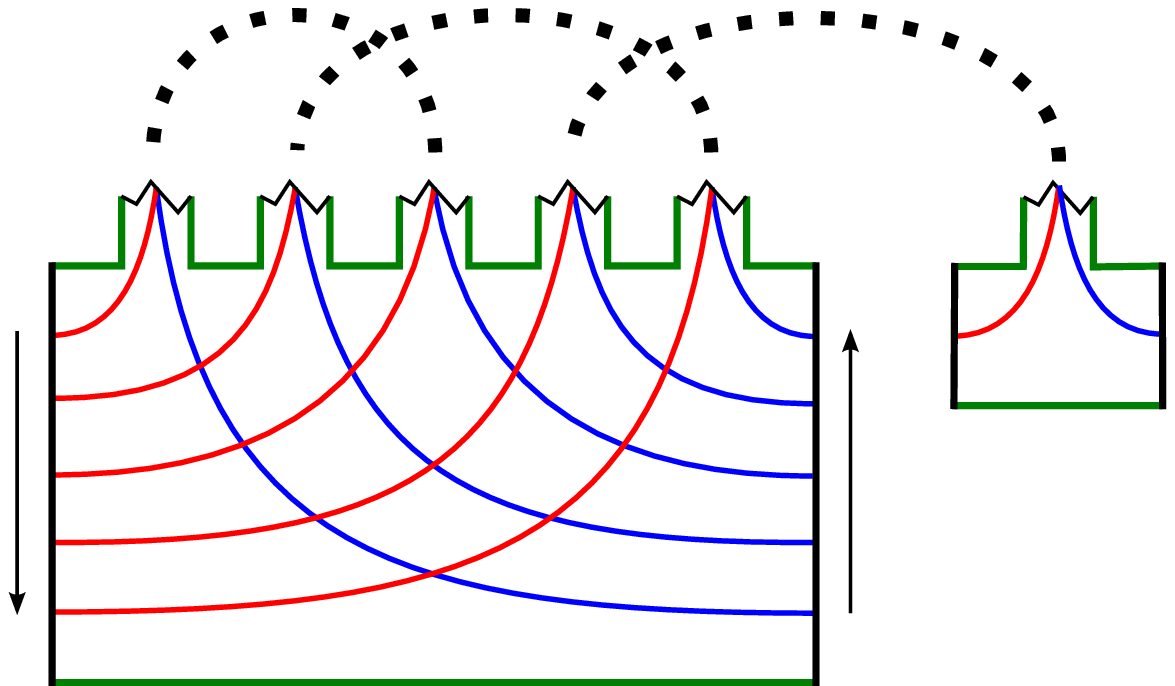}
	\caption{Heegaard diagram for a negative twisted slice $\TW_{\F,-}$.}
	\label{fig:identity}
\end{figure}

\begin{figure}
	\begin{subfigure}[t]{.32\linewidth}
		\labellist
		\small\hair=2pt
		\pinlabel $\del$ [b] at 144 70
		\endlabellist
		\centering
		\includegraphics[scale=.35]{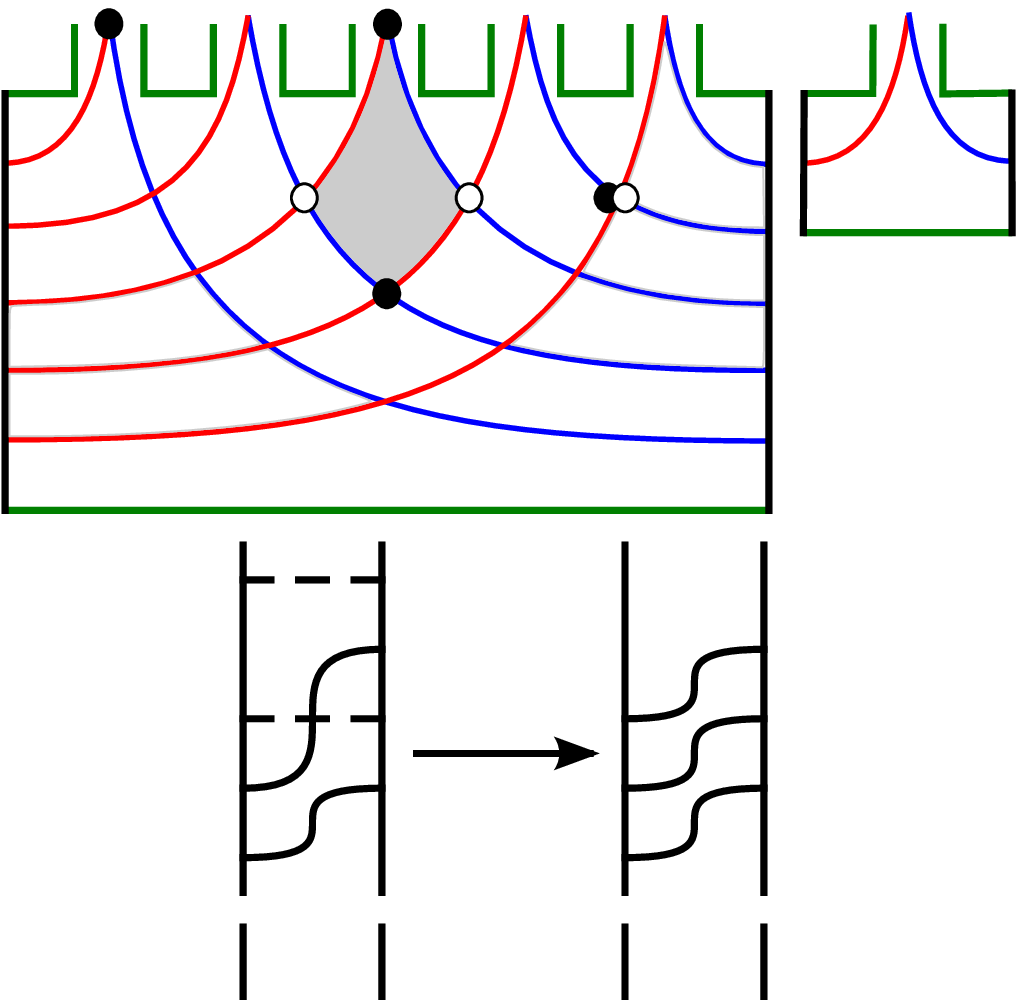}
		\caption{Differential.}
		\label{subfig:identity_domain_1}
	\end{subfigure}
	\begin{subfigure}[t]{.32\linewidth}
		\labellist
		\pinlabel $\cdot$ at 90 70
		\pinlabel $=$ at 185 70
		\endlabellist
		\centering
		\includegraphics[scale=.35]{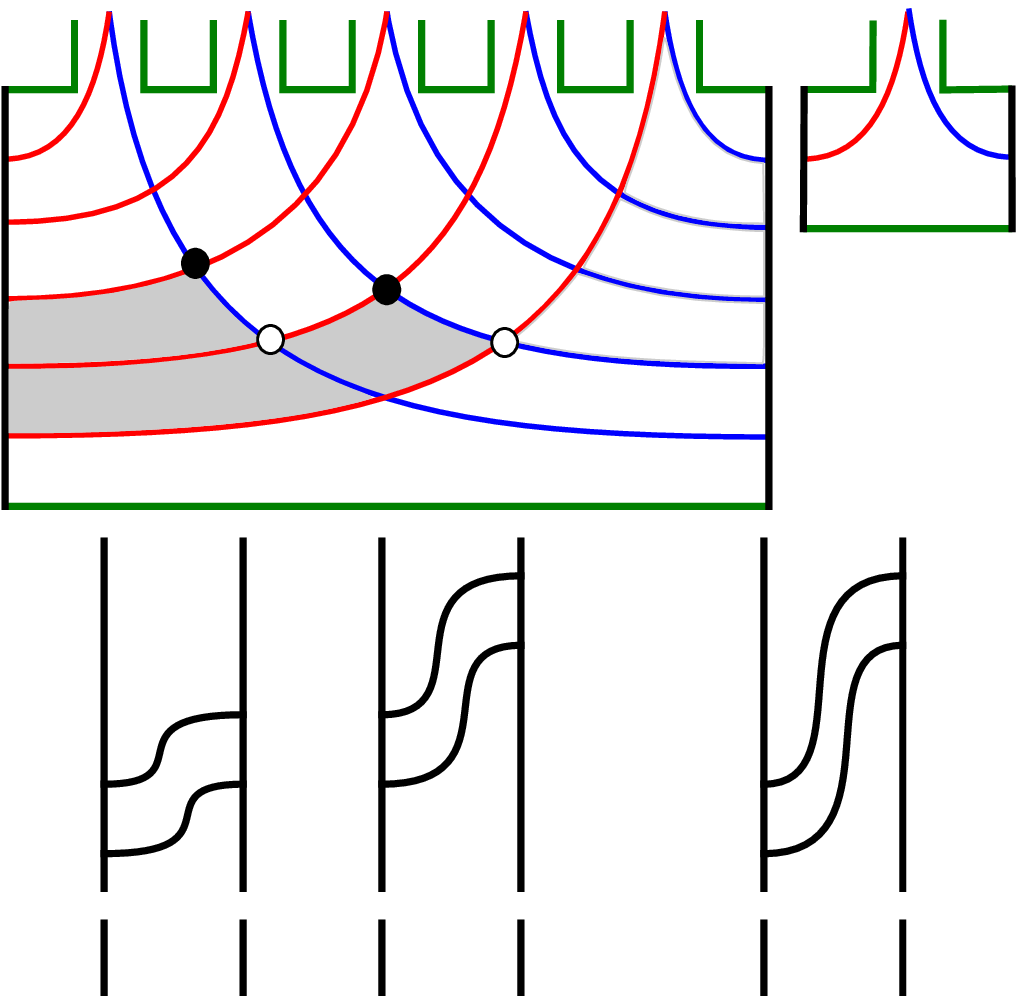}
		\caption{Left action.}
		\label{subfig:identity_domain_2}
	\end{subfigure}
	\begin{subfigure}[t]{.32\linewidth}
		\labellist
		\pinlabel $\cdot$ at 90 70
		\pinlabel $=$ at 185 70
		\endlabellist
		\centering
		\includegraphics[scale=.35]{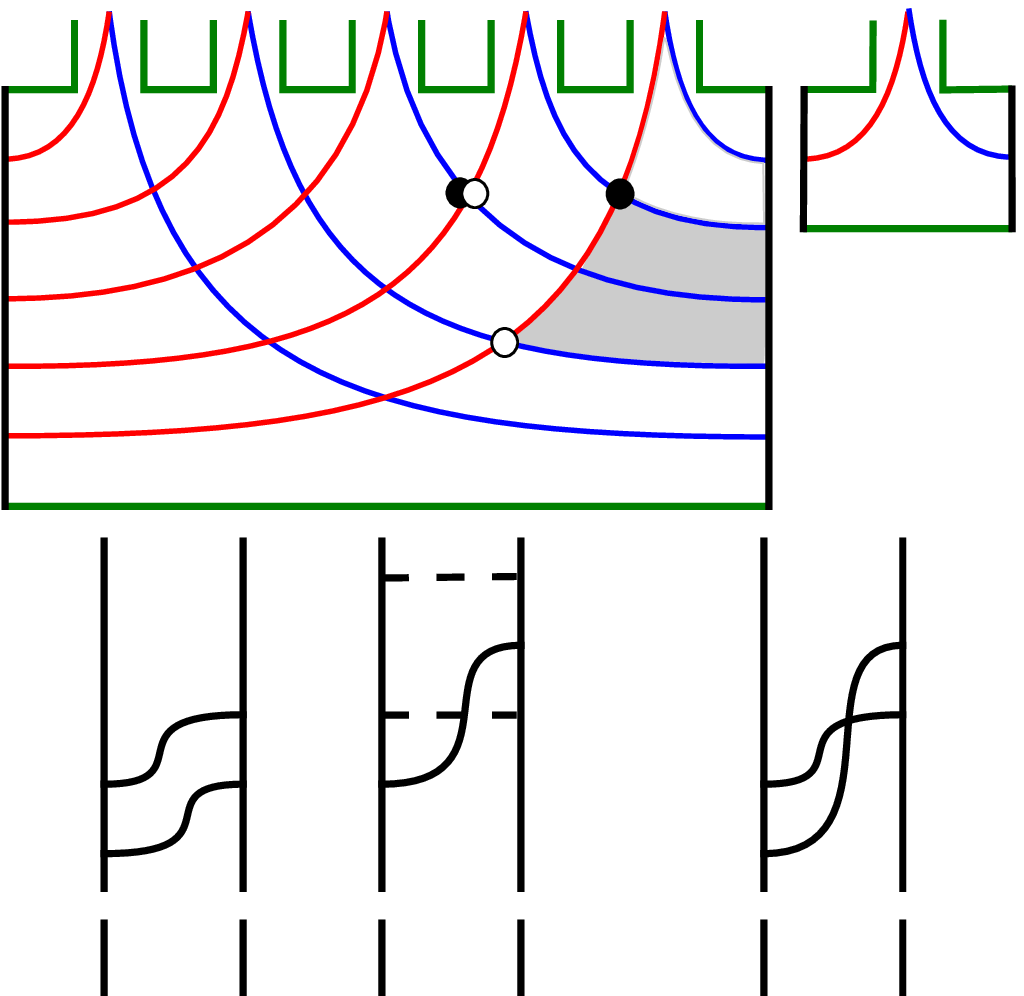}
		\caption{Right action.}
		\label{subfig:identity_domain_3}
	\end{subfigure}
	\caption{Examples of domains counted in the diagram for $\TW_{\F,-}$. In each case the domain goes from the black
	dots to the white dots. Below them we show the corresponding operations on the algebra.}
	\label{fig:identity_domains}
\end{figure}

\section{The join map}
\label{sec:join}

In this section we will define the join and gluing maps, and prove some basic properties.
Recall that the gluing operation is defined as a special case of the join operation. The gluing map
is similarly a special case of the join map. Thus for the most part we will only talk about the
general case, i.e. the join map.

\subsection{The algebraic map}
\label{sec:join_def_algebraic}

We will first define an abstract algebraic map, on the level of $\Ainf$--modules.

Let $A$ be a differential graded algebra, and $\li{_A}M$ be a left $\Ainf$--module over it. As discussed in \Appendix~\ref{sec:duals},
the dual
${M^\vee}_A$ is a right $\Ainf$--module over $A$. Thus $\li{_A}(M\otimes M^\vee)_A$ is an
$\Ainf$--bimodule. On the other hand, since $A$ is a bimodule over itself, so is its dual $\li{_A}{A^\vee}_A$.
We define a map $M\otimes M^\vee\to A^\vee$ which is an $\Ainf$--analog of the natural pairing of a module and its dual.

\begin{defn}
	\label{def:nabla_map}
	The \emph{algebraic join map} $\nabla_M\co \li{_A}(M\otimes M^\vee)_A \to \li{_A}{A^\vee}_A$---or
	just $\nabla$ when unambiguous---is an
	$\Ainf$--bimodule morphism, defined as follows. It is the unique morphism satisfying
	\begin{multline}
		\label{eq:nabla_def}
		\left<\nabla_{i|1|j}(a_1,\ldots,a_i,~p,~q^\vee,~,a'_1,\ldots,a'_j),a''\right>\\
		=\left<m_{i+j+1|1}(a'_1,\ldots,a'_j,a'',a_1,\ldots,a_i,~p),q^\vee\right>,
	\end{multline}
	for any $i,j\geq 0$, $p\in M$, $q^\vee\in M^\vee$, and $a^*_*\in A$.
\end{defn}

\Equation~(\ref{eq:nabla_def}) is best represented diagrammatically, as in \Figure~\ref{fig:nabla_def}. Note that
$\nabla_M$ is a bounded morphism if and only if $M$ is a bounded module.

\begin{figure}
	\centering
	\labellist
	\pinlabel $=$ at 128 78
	\pinlabel $=$ at 264 78
	\small
	\pinlabel $\nabla_M$ at 60 66
	\pinlabel $\nabla_M$ at 196 66
	\pinlabel \rotatebox[origin=c]{90}{$m_M$} at 348 92
	\hair 1pt
	\pinlabel $A$ [b] at 20 120
	\pinlabel $M$ [b] at 48 120
	\pinlabel $M^\vee$ [b] at 76 120
	\pinlabel $A$ [b] at 100 120
	\pinlabel $A^\vee$ [t] at 64 16
	\pinlabel $A$ [b] at 156 120
	\pinlabel $M$ [b] at 184 120
	\pinlabel $M$ [b] at 208 120
	\pinlabel $A$ [b] at 236 120
	\pinlabel $A$ [t] at 196 16
	\pinlabel $A$ [b] at 292 120
	\pinlabel $M$ [b] at 324 120
	\pinlabel $M$ [b] at 372 120
	\pinlabel $A$ [b] at 404 120
	\pinlabel $A$ [t] at 348 16
	\endlabellist
	\includegraphics[scale=.5]{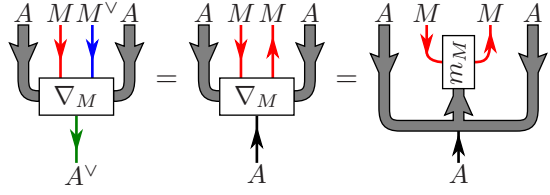}
	\caption{Definition of the join map $\nabla$.}
	\label{fig:nabla_def}
\end{figure}

As discussed in \Appendix~\ref{sec:morphisms_and_premorphisms}, morphisms of $\Ainf$--modules form chain complexes,
where cycles are \emph{homomorphisms}. Only homomorphisms descend to maps on homology.

\begin{prop}
	\label{prop:nabla_well_defined}
	For any $\li{_A}M$, the join map $\nabla_M$ is a homomorphism.
\end{prop}
\begin{proof}
	It is a straightforward but 
	tedious computation to see that $\del \nabla_M=0$ is equivalent to the structure equation for $m_M$.

	A more enlightening way to see this is to notice that by turning the diagram in \Figure~\ref{fig:nabla_def} partly sideways,
	we get a diagram for the homotopy equivalence $h_M\co A\dtens M\to M$, shown in \Figure~\ref{fig:h_M}. Taking the differential $\del\nabla_M$
	and turning the resulting diagrams sideways, we get precisely $\del h_M$. We know that $h_M$ is a homomorphism and,
	so $\del h_M=0$.
	
	The equivalences are presented in \Figure~\ref{fig:del_nabla_proof}.
\end{proof}

\begin{figure}
			\labellist
			\pinlabel $=$ at 116 76
			\small
			\pinlabel $h_M$ at 60 76
			\pinlabel $m_M$ at 224 76
			\hair=1.5pt
			\pinlabel $A$ [b] at 12 136
			\pinlabel $A$ [b] at 36 136
			\pinlabel $A$ [b] at 60 136
			\pinlabel $M$ [b] at 84 136
			\pinlabel $M$ [t] at 60 16
			\pinlabel $A$ [b] at 136 136
			\pinlabel $A$ [b] at 156 136
			\pinlabel $A$ [b] at 176 136
			\pinlabel $M$ [b] at 224 136
			\pinlabel $M$ [t] at 224 16
			\endlabellist
			\includegraphics[scale=.55]{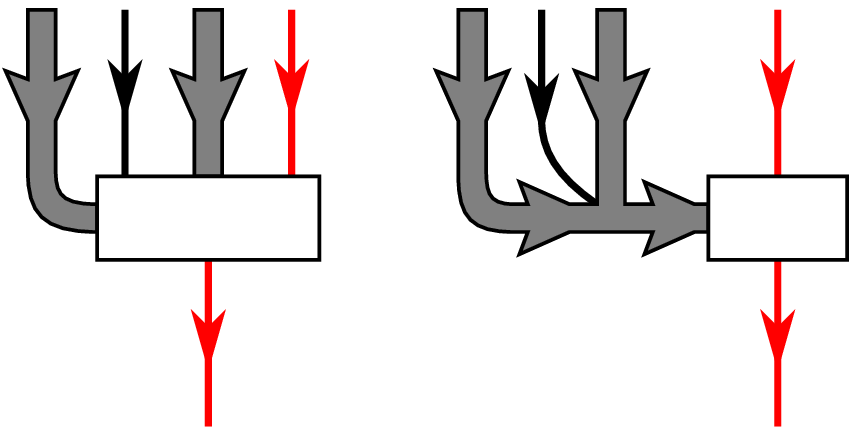}
			\caption{The homotopy equivalence $h_M\co A\dtens M\to M$.}
			\label{fig:h_M}
\end{figure}

\begin{figure}
	\begin{subfigure}[h]{\linewidth}
		\centering
		\labellist
		\pinlabel $+$ at 124 78
		\pinlabel $+$ at 252 78
		\pinlabel $+$ at 412 78
		\pinlabel $+$ at 564 78
		\small
		\pinlabel $\ol\mu_A$ at 20 96
		\pinlabel $\ol\mu_A$ at 228 96
		\pinlabel \rotatebox[origin=c]{180}{$\ol\mu_A$} at 332 60
		\pinlabel $\nabla_M$ at 60 50
		\pinlabel $\nabla_M$ at 188 50
		\pinlabel $\nabla_M$ at 332 104
		\pinlabel $\nabla_M$ at 500 50
		\pinlabel $\nabla_M$ at 628 50
		\pinlabel $m_M$ at 484 100
		\pinlabel \rotatebox[origin=c]{180}{$m_M$} at 644 100
		\endlabellist
		\includegraphics[scale=.5]{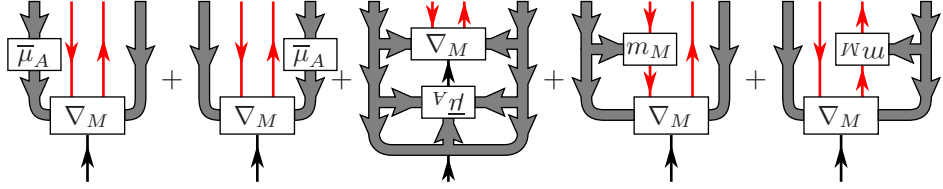}
		\caption{The differential $\del\nabla_M$ which needs to vanish to show that $\nabla_M$ is
		an $\Ainf$--bimodule homomorphism.}
		\label{subfig:del_nabla}
	\end{subfigure}
	\begin{subfigure}[h]{\linewidth}
		\labellist
		\pinlabel $+$ at 108 78
		\pinlabel $+$ at 238 78
		\pinlabel $+$ at 368 78
		\pinlabel $+$ at 488 78
		\small
		\pinlabel $\ol\mu_A$ at 20 96
		\pinlabel $\ol\mu_A$ at 296 76
		\pinlabel $\ol\mu_A$ at 440 100
		\pinlabel $h_M$ at 68 50
		\pinlabel $h_M$ at 192 86
		\pinlabel $h_M$ at 328 32
		\pinlabel $h_M$ at 440 50
		\pinlabel $h_M$ at 560 50
		\pinlabel $m_M$ at 192 38
		\pinlabel $m_M$ at 612 100
		\endlabellist
		\centering
		\includegraphics[scale=.5]{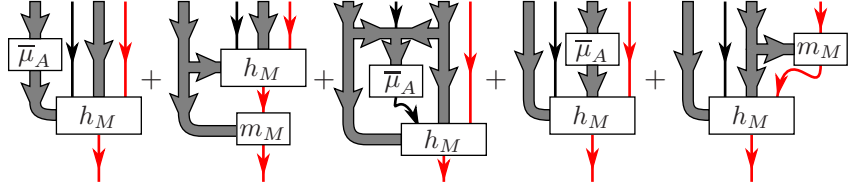}
		\caption{The differential $\del h_M$ of the homotopy equivalence $h_M$.}
		\label{subfig:del_h_equivalence}
	\end{subfigure}
	\caption{Proof that $\nabla$ is a homomorphism, by rotating diagrams.}
	\label{fig:del_nabla_proof}
\end{figure}

We will prove two naturallity statements about $\nabla$ that together imply that $\nabla$ descends to a well defined map on the
derived category. The first shows that $\nabla$ is natural with respect to isomorphisms in the derived category of
the DG-algebra $A$, i.e. homotopy equivalences of modules.
The second shows that $\nabla$  is natural with respect to equivalences of derived categories. (Recall from~\cite{Zar:BSFH} that different algebras
corresponding to the same sutured surface are derived-equivalent.)

\begin{prop}
	\label{prop:nabla_invariance_module}
	Suppose $\li{_A}M$ and $\li{_A}N$ are two $\Ainf$--modules over $A$, such that there are inverse homotopy equivalences
	$\varphi\co M\to N$ and $\psi\co N\to M$. Then there is an $\Ainf$--homotopy
	equivalence of $A,A$--bimodules
	\begin{equation*}
		\varphi\otimes\psi^\vee\co M\otimes M^\vee\to N\otimes N^\vee,
	\end{equation*}
	and the following diagram commutes up to $\Ainf$--homotopy:

	\centerline{\xymatrix@C=1in{
	M\otimes M^\vee \ar[d]_{\varphi\otimes\psi^\vee} \ar[dr]^{\nabla_M} & \\
	N\otimes N^\vee \ar[r]^{\nabla_N} & A^\vee. \\
	}}
\end{prop}

\begin{prop}
	\label{prop:nabla_invariance_algebra}
	Suppose $A$ and $B$ are differential graded algebras, and $\li{_B}X^A$ and $\li{_A}Y^B$ are two
	type--$DA$ bimodules, which are quasi-inverses. That is, there are $\Ainf$--homotopy
	equivalences
	\begin{align*}
		\li{_A}(Y\sqtens X)^A &\simeq\li{_A}\II^A,
		&\li{_B}(X\sqtens Y)^B &\simeq\li{_B}\II^B.
	\end{align*}
	Moreover, suppose $H_*(B^\vee)$ and $H_*(X\sqtens A^\vee\sqtens X^\vee)$ have the same rank (over $\ZZ/2$).

	Then there is a $B,B$--bimodule homotopy equivalence
	\begin{equation*}
		\varphi_X\co X\sqtens A^\vee \sqtens X^\vee \to B^\vee.
	\end{equation*}
	Moreover, for any $\Ainf$--module $\li{_A}M$, such that $X\sqtens M$ is well defined, the following diagram commutes up to $\Ainf$--homotopy:
	
	\centerline{\xymatrix@C=1in{
		X\sqtens M\otimes M^\vee \sqtens X^\vee \ar[d]_{\id_X\sqtens\nabla_M\sqtens\id_{X^\vee}}
		\ar[dr]^{\nabla_{X\sqtens M}} & \\
		X\sqtens A^\vee\sqtens X^\vee \ar[r]^{\varphi_X} & B^\vee.
	}}
\end{prop}

Notice the condition that $X\sqtens M$ be well defined. This can be satisfied for example if $M$ is a bounded module,
or if $X$ is reletively bounded in $A$ with respect to $B$.
Before proving \Propositions~\ref{prop:nabla_invariance_module} and~\ref{prop:nabla_invariance_algebra} in \Section~\ref{sec:proof_join_algebraic}, we will use them to
define the join $\Psi$.

\subsection{The geometric map}
\label{sec:join_def_geometric}

Suppose that $\Y_1$ and $\Y_2$ are two sutured manifolds, and $\W=(W,\Gamma,-\F)$
is a partially sutured manifold, with embeddings
$\W\into\Y_1$ and $-\W\into\Y_2$. Let $\Z$ be any arc diagram parametrizing the surface $\F$.
Recall that $-\W=(-W,\Gamma,-\ol\F)$. Also recall the twisting slice $\TW_{\F,+}$, from \Definition~\ref{def:twisting_slice}.
The join $\Y_1\Cup_{\W}\Y_2$ of $\Y_1$ and $\Y_2$ along $\W$ was defined as
\begin{equation*}
	\Y_1~\Cup_{\W}~\Y_2 = (\Y_1\setminus\W)~\cup_\F~\TW_{\F,+}~\cup_{-\ol\F}~(\Y_2\setminus-\W).
\end{equation*}

Let $A=\A(\Z)$ be the algebra associated to $\Z$. Let $\li{_A}M$, $U^A$, and $\lu{A}V$ be representatives for the
bordered sutured modules
$\li{_A}\BSA(\W)$, $\BSD(\Y_1\setminus \W)^A$, and $\lu{A}\BSD(\Y_2\setminus -\W)$, respectively such that $U\sqtens M$
and $M^\vee\sqtens V$ are well-defined. (Recall that the modules are only defined up to homotopy equivalence, and that
the $\sqtens$ product is only defined under some boundedness conditions.)
We proved in
\Proposition~\ref{prop:mirror_is_dual} that ${M^\vee}_A$ is a representative for $\BSA(-\W)_A$, and in
\Proposition~\ref{prop:twisting_slice_invariants} that $\li{_A}{A^\vee}_A$ is a representative for $\BSAA(\TW_{\F,+})$.

From the K\"unneth formula for $\SFH$ of a disjoint union, and from \Theorem~\ref{thm:bimodules_all}, we have the following homotopy equivalences of
chain complexes.
\begin{multline*}
	\SFC(\Y_1\cup\Y_2)\cong\SFC(\Y_1)\otimes\SFC(\Y_2)\\
	\simeq\left(\BSD(\Y_1\setminus\W)\sqtens_{A}\BSA(\W)\right)\otimes\left(\BSA(-\W)\sqtens_{A}\BSD(\Y_2\setminus-\W)\right)\\
	\simeq U^A~\sqtens~\li{_A}(M\otimes M^\vee)_A~\sqtens~\lu{A}V.
\end{multline*}
\begin{multline*}
	\SFC(\Y_1\Cup_{\W}\Y_2)\\
	\simeq\BSD(\Y_1\setminus\W)\sqtens_{A}\BSAA(\TW_{\F,+})\sqtens_{A}\BSD(\Y_2\setminus-\W)\\
	\simeq U^A~\sqtens~\li{_A}{A^\vee}_A~\sqtens~\lu{A}V.
\end{multline*}

\begin{defn}
	\label{def:join_map_def}
	Let $\Y_1$, $\Y_2$ and $\W$ be as described above. Define the \emph{geometric join map}
	\begin{equation*}
		\Psi_M\co\SFC(\Y_1)\otimes\SFC(\Y_2)\to\SFC(\Y_1\Cup_{\W}\Y_2)
	\end{equation*}
	by the formula
	\begin{equation}
		\label{eq:join_map_geom_def}
		\Psi_M=\id_{U}\sqtens\nabla_M\sqtens\id_{V}
		\co U\sqtens M \otimes M^\vee \sqtens V \to U \sqtens A^\vee \sqtens V.
	\end{equation}
\end{defn}

Note that such an induced map is not generally well defined (it might involve an infinite sum). In this case, however,
we have made some boundedness assumptions. Since $U\sqtens M$ and $M^\vee\sqtens V$ are defined, either $M$ must be bounded,
or both of $U$ and $V$ must be bounded. In the former case, $\nabla_M$ is also bounded. Either of these situations
guarantees that the sum defining $\Psi_M$ is finite. 

\begin{thm}
	\label{thm:join_map_invariance}
	The map $\Psi_M$ from \Definition~\ref{def:join_map_def} is, up to homotopy, independent on the choice of
	parametrization $\Z$, and on the choices of representatives $M$, $U$, and $V$.
\end{thm}
\begin{proof}
	First, we will give a more precise version of the statement.
	Let $\Z'$ be any other parametrization of $\F$, with $B=\A(-\Z')$, and let 
	$\li{_B}M'$, $U'^B$ and $\lu{B}V'$, be representatives for the respective bordered sutured modules.
	Then there are homotopy equivalences $\varphi$ and $\psi$ making the following diagram commute up to $\Ainf$--homotopy:

	\centerline{\xymatrix@C=1in{
	U\sqtens M\otimes M^\vee\sqtens V \ar[r]^\varphi \ar[d]_{\Psi_M}
	& U'\sqtens M'\otimes M'^\vee\sqtens V' \ar[d]^{\Psi_{M'}} \\
	U\sqtens A^\vee \sqtens V \ar[r]^\psi
	& U'\sqtens B^\vee \sqtens V'.
	}}

	The proof can be broken up into several steps. The first step is
	independence
	from the choice of $U$ and $V$, given a fixed choice for $A$ and $M$. This follows directly from the fact
	$\id\sqtens\cdot$ and $\cdot\sqtens\id$ are DG-functors.

	The second step is to show independence from the choice of $M$, for fixed $A$, $U$, and $V$. This follows from
	\Proposition~\ref{prop:nabla_invariance_module}. Indeed, suppose $\varphi\co M\to M'$ is a homotopy equivalence
	with homotopy inverse $\psi\co M'\to M$. Then $\psi^\vee\co M^\vee \to {M'}^\vee$ is also a homotopy equivalence
	inducing the homotopy equivalence
	\begin{equation*}
		\id_U\sqtens\varphi\otimes\psi^\vee\sqtens\id_V
		\co U\sqtens M\otimes M^\vee\sqtens V
		\to U\sqtens M'\otimes {M'}^\vee\sqtens V.
	\end{equation*}
	By \Proposition~\ref{prop:nabla_invariance_module}, $\nabla_M\simeq\nabla_{M'}\circ(\varphi\otimes\psi^\vee)$, which implies
	\begin{equation*}
		\id_U\sqtens\nabla_M\sqtens\id_V\simeq(\id_U\sqtens\nabla_{M'}\sqtens\id_V)
		\circ(\id_U\sqtens\varphi\otimes\psi^\vee\sqtens\id_V).
	\end{equation*}

	The final step is to show independence from the choice of algebra $A$. We will cut $\Y_1$ and $\Y_2$ into several
	pieces, so we can evaluate the two different versions of $\Psi$ from the same geometric picture.

	Let $-\F'$ and $-\F''$ be two parallel copies of $-\F$ in $\W$, which cut out $\W'=(W',\Gamma',-\F')$ and
	$\W''=(W'',\Gamma'',-\F'')$, where $\W''\subset\W'\subset\W$. Let $\PP=\W'\setminus\W''$ and $\Q=\W\setminus\W'$
	(see \Figure~\ref{fig:x-cutting-W-3d}). Both $\PP$ and $\Q$ are topologically $F\times[0,1]$.

	\begin{figure}
		\centering
		\labellist
		\small\hair=1.5pt
		\pinlabel $-\F$ [t] at 12 20
		\pinlabel $-\F'$ [t] at 92 20
		\pinlabel $-\F''$ [t] at 172 20
		\pinlabel $\W$ [b] at 210 296
		\pinlabel $\W'$ [b] at 250 264
		\pinlabel $\W''$ [b] at 290 232
		\pinlabel $\PP$ [b] at 140 232
		\pinlabel $\Q$ [b] at 60 232
		\endlabellist
		\includegraphics[scale=.4]{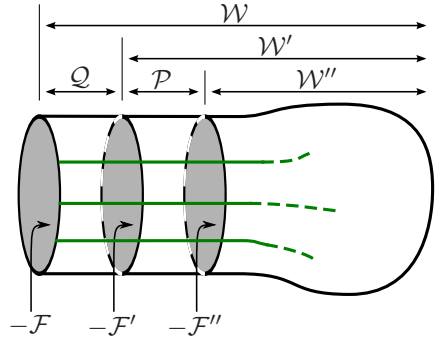}
		\caption{The various pieces produced by slicing $\W$ at two surfaces parallel to $\F$.}
		\label{fig:x-cutting-W-3d}
	\end{figure}

	Parametrize $\F$ and $\F''$ by $\Z$, and $\F'$ by $\Z'$, where $\A(\Z)=A$, and $\A(\Z')=B$.  
	Let $\li{_B}X^A$ and $\li{_A}Y^B$ be representatives for $\li{_B}\BSAD(\PP)^A$ and $\li{_A}\BSAD(\Q)^B$,
	respectively. Note that $\Q\cup_{\F'}\PP$ is a product bordered sutured manifold, and thus has trivial
	invariant $\li{_A}\BSAD(\Q\cup\PP)^A\simeq\li{_A}\II^A$. By the pairing theorem, this implies
	$Y\sqtens X\simeq\li{_A}\II^A$. Similarly, by stacking $\PP$ and $\Q$ in the opposite order we get
	$X\sqtens Y\simeq\li{_B}\II^B$.

	There are embeddings $\W',\W''\into\Y_1$ and $-\W',-\W''\into\Y_2$ and two distinct ways to cut and glue them together,
	getting $\Y_1\Cup_{\W'}\Y_2\cong\Y_1\Cup_{\W''}\Y_2$. This is illustrated schematically in
	\Figure~\ref{fig:x-cutting-and-regluing}.

	\begin{figure}
		\begin{subfigure}[t]{\linewidth}
			\centering
			\labellist
			\small\hair=1.5pt
			\pinlabel $\Y_1\setminus\W'$ at 38 58
			\pinlabel $\PP$ at 92 58
			\pinlabel $\W''$ at 134 58
			\pinlabel $-\W''$ at 226 58
			\pinlabel $-\PP$ at 268 58
			\pinlabel $\Y_2\setminus-\W'$ at 344 58
			\pinlabel $\F'$ [t] at 76 20
			\pinlabel $\F''$ [t] at 108 20
			\pinlabel $\ol\F''$ [t] at 252 20
			\pinlabel $\ol\F'$ [t] at 284 20
			\endlabellist
			\includegraphics[scale=.6]{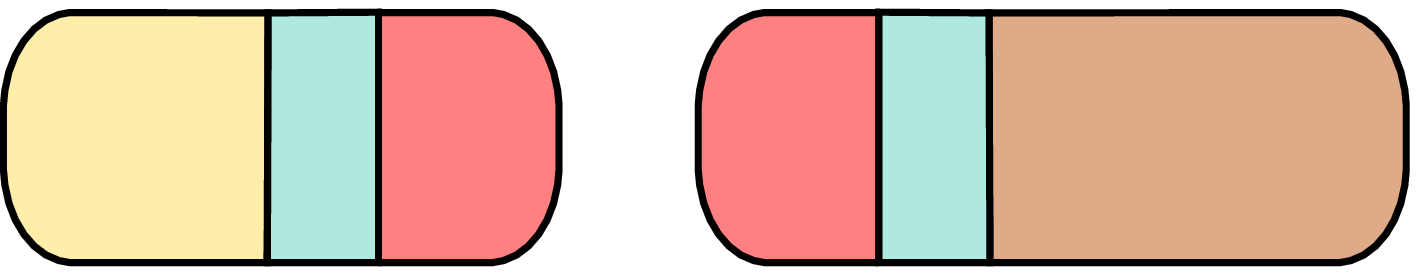}
			\caption{Cutting $\Y_1$ and $\Y_2$ in two different places.}
			\label{subfig:x-cutting-W-2d}
		\end{subfigure}
		\begin{subfigure}[t]{\linewidth}
			\centering
			\labellist
			\small\hair=1.5pt
			\pinlabel $\Y_1\setminus\W'$ at 38 58
			\pinlabel $\PP$ at 92 58
			\pinlabel $\TW_{\F'',+}$ at 144 58
			\pinlabel $-\PP$ at 192 58
			\pinlabel $\Y_2\setminus-\W'$ at 268 58
			\pinlabel $\F'$ [t] at 76 20
			\pinlabel $\F''$ [t] at 108 20
			\pinlabel $\ol\F''$ [t] at 176 20
			\pinlabel $\ol\F'$ [t] at 208 20
			\endlabellist
			\includegraphics[scale=.6]{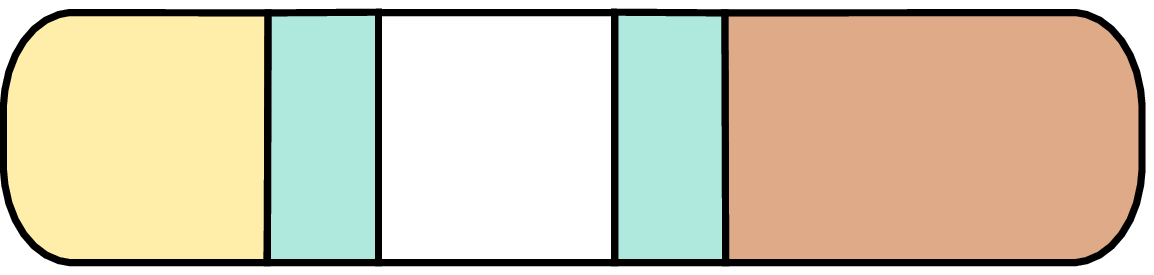}
			\caption{The join by $\W''$.}
			\label{subfig:x-cutting-W-regluing-1}
		\end{subfigure}
		\begin{subfigure}[t]{\linewidth}
			\centering
			\labellist
			\small\hair=1.5pt
			\pinlabel $\Y_1\setminus\W'$ at 38 58
			\pinlabel $\TW_{\F',+}$ at 142 58
			\pinlabel $\Y_2\setminus-\W'$ at 268 58
			\pinlabel $\F'$ [t] at 76 20
			\pinlabel $\F''$ [t] at 108 20
			\pinlabel $\ol\F''$ [t] at 176 20
			\pinlabel $\ol\F'$ [t] at 208 20
			\endlabellist
			\includegraphics[scale=.6]{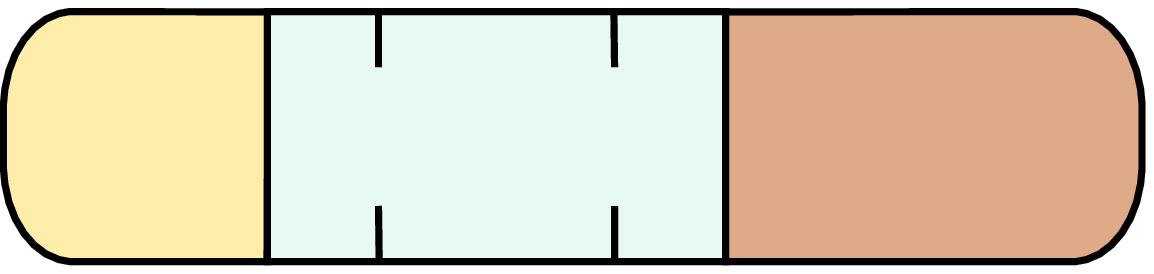}
			\caption{The join by $\W'$.}
			\label{subfig:x-cutting-W-regluing-2}
		\end{subfigure}
		\caption{Two ways of cutting and pasting to get the join of $\Y_1$ and $\Y_2$.}
		\label{fig:x-cutting-and-regluing}
	\end{figure}

	Let $\li{_A}M$ be a representative for $\li{_A}\BSA(\W'')$. By the pairing theorem, $\li{_B}(X\sqtens M)$ is
	a representative for $\li{_B}\BSA(\W')$. Notice that $\TW_{\F',+}\cong\PP\cup\TW_{\F'',+}\cup-\PP$ and
	$\li{_B}{B^\vee}_B$ and $\li{_B}(X\sqtens A^\vee \sqtens X^\vee)_B$ are both representatives for its
	$\BSAA$ invariant. In particular, they have the same homology.
	Finally, let $U^B$ and $\lu{B}V$ be representatives for $\BSD(\Y_1\setminus\W')^B$ and
	$\lu{B}\BSD(\Y_2\setminus-\W')$, respectively.
	
	The two join maps $\Psi_M$ and $\Psi_{X\sqtens M}$ are described by the following equations.
	\begin{multline*}
		\Psi_M=\id_{U\sqtens X}\sqtens\nabla_{M}\sqtens\id_{X^\vee\sqtens V}\co\\
		(U\sqtens X)\sqtens M \otimes M^\vee\sqtens (X^\vee \sqtens V)
		\to (U\sqtens X) \sqtens A^\vee \sqtens (X^\vee \sqtens V),
	\end{multline*}
	\begin{multline*}
		\Psi_{X\sqtens M}=\id_U\sqtens\nabla_{X\sqtens M}\sqtens\id_V\co\\
		U\sqtens (X\sqtens M) \otimes (M^\vee\sqtens X^\vee) \sqtens V
		\to U\sqtens B^\vee \sqtens V.
	\end{multline*}

	We can apply \Proposition~\ref{prop:nabla_invariance_algebra}. The boundedness condition can be satisfied by
	requiring that $X$ and $Y$ are bounded modules. There is a homotopy equivalence
	$\varphi_X\co X\sqtens A^\vee\sqtens X^\vee\to B$, and a homotopy
	$\nabla_{X\sqtens M}\sim\varphi_X\circ(\id_X\sqtens\nabla_{M}\sqtens\id_{X^\vee})$. These induce a homotopy
	\begin{multline*}
		(\id_U\sqtens\varphi_X\sqtens\id_V)\circ\Psi_M
		=\id_U\sqtens(\varphi_X\circ(\id_X\sqtens\nabla_M\sqtens\id_{X^\vee}))\sqtens\id_V\\
		\sim\id_U\sqtens\nabla_{X\sqtens M}\sqtens\id_V
		=\Psi_{X\sqtens M}.
	\end{multline*}
	
	This finishes the last step. Combining all three gives complete invariance. Thus we can refer to $\Psi_\W$
	from now on.
\end{proof}

\subsection{Proof of algebraic invariance}
\label{sec:proof_join_algebraic}

In this section we prove \Propositions~\ref{prop:nabla_invariance_module}
and~\ref{prop:nabla_invariance_algebra}.

\begin{proof}[Proof of \Proposition~\ref{prop:nabla_invariance_module}]
	The proof will be mostly diagrammatic. There are two modules $\li{_A}M$ and $\li{_A}N$, and two inverse homotopy
	equivalences,
	$\varphi\co M\to N$ and $\psi\co N\to M$. The dualizing functor $\li{_A}\Mod\to\Mod_{A}$ is a DG-functor. Thus it is easy to see that
	\begin{equation*}
	\varphi\otimes\psi^\vee=(\varphi\otimes\id_{N^\vee})\circ(\id_M\otimes\psi^\vee)
	\end{equation*}
	is also a homotopy equivalence. Let $H\co M\to M$ be the homotopy between $\id_M$ and $\psi\circ\varphi$.
	
	We have to show that the homomorphism
	\begin{equation}
		\label{eq:nabla_variation_module}
		\nabla_M+\nabla_N\circ(\varphi\otimes\psi^\vee)
	\end{equation}
	is null-homotopic (see \Figure~\ref{subfig:nabla_variation_original}).
	Again, it helps if we turn the diagram sideways,
	where bar resolutions come into play. Let $h_M\co A\dtens M\to M$ and $h_N\co A\dtens N\to N$ be the
	natural homotopy equivalences.
	
	Turning the first term in \Equation~(\ref{eq:nabla_variation_module}) sideways, we get $h_M$. Turning the second term
	sideways we get $\psi\circ h_N\circ(\id_A\dtens\varphi)$. Thus we need to show that
	\begin{equation}
		\label{eq:nabla_variation_module_sideways}
		h_M + \psi\circ h_N\circ(\id_A\dtens\varphi)
	\end{equation}
	is null-homotopic (see \Figure~\ref{subfig:nabla_variation_sideways}).

	There is a canonical homotopy $h_{\varphi}\co A \dtens M\to N$  between
	$\varphi\circ h_M$ and $h_N\circ(\id_A\dtens\varphi)$, given by
	\begin{multline*}
		h_\varphi(a_1,\ldots,a_i,~(a',~a''_1,\ldots,a''_j,~m))
		=\varphi(a_1,\ldots,a_i,a',a''_1,\ldots,a''_j,~m).
	\end{multline*}

	Thus we can build the null-homotopy $\psi\circ h_{\varphi} + H\circ h_M$ 
	(see \Figure~\ref{subfig:nabla_var_homotopy_sideways}). Indeed,
	\begin{align*}
	\del(\psi\circ h_{\varphi})&=\psi\circ\varphi\circ h_M + \psi\circ h_N \circ(\id_A\dtens\varphi),\\
	\del(H\circ h_M)&=\id_M\circ h_M+\psi\circ\varphi\circ h_M.
	\end{align*}

	Alternatively, we can express the null-homotopy of the expression~(\ref{eq:nabla_variation_module}) directly
	as in \Figure~\ref{subfig:nabla_var_homotopy_original}.
\end{proof}

\begin{figure}
	\begin{subfigure}[h]{.48\linewidth}
		\centering
		\labellist
		\pinlabel $+$ at 148 82
		\small
		\pinlabel \rotatebox[origin=c]{90}{$m_M$} at 68 100
		\pinlabel \rotatebox[origin=c]{90}{$m_N$} at 236 80
		\pinlabel $\varphi$ at 212 124
		\pinlabel \rotatebox[origin=c]{180}{$\psi$} at 260 124
		\endlabellist
		\includegraphics[scale=.55]{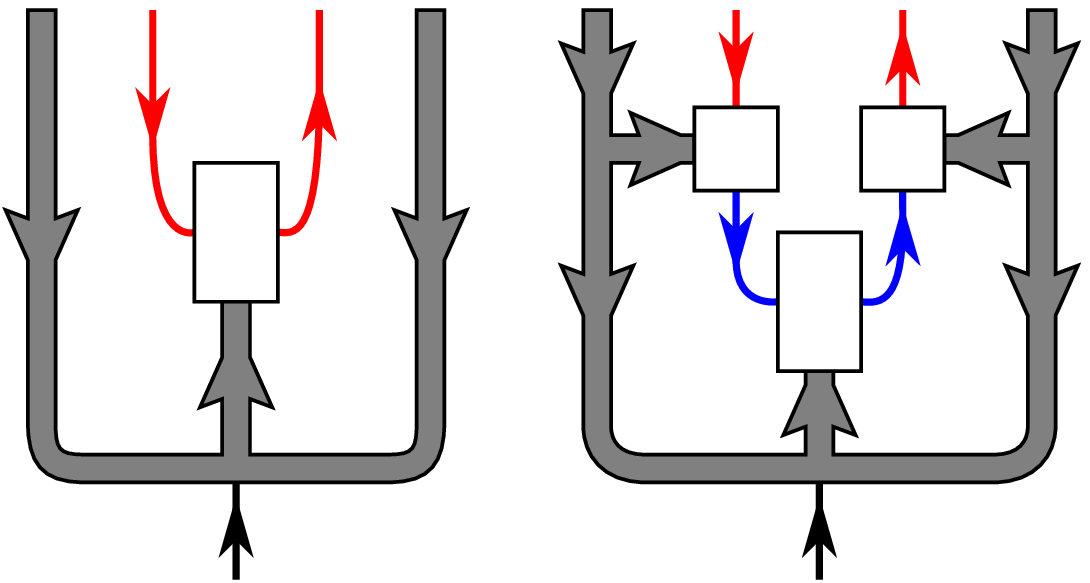}
		\caption{Representation of \Equation~(\ref{eq:nabla_variation_module}).}
		\label{subfig:nabla_variation_original}
	\end{subfigure}
	\begin{subfigure}[h]{.48\linewidth}
		\centering
		\labellist
		\pinlabel $+$ at 106 82
		\small
		\pinlabel $h_M$ at 60 80
		\pinlabel $h_N$ at 184 80
		\pinlabel $\varphi$ at 228 128
		\pinlabel $\psi$ at 184 36
		\endlabellist
		\includegraphics[scale=.55]{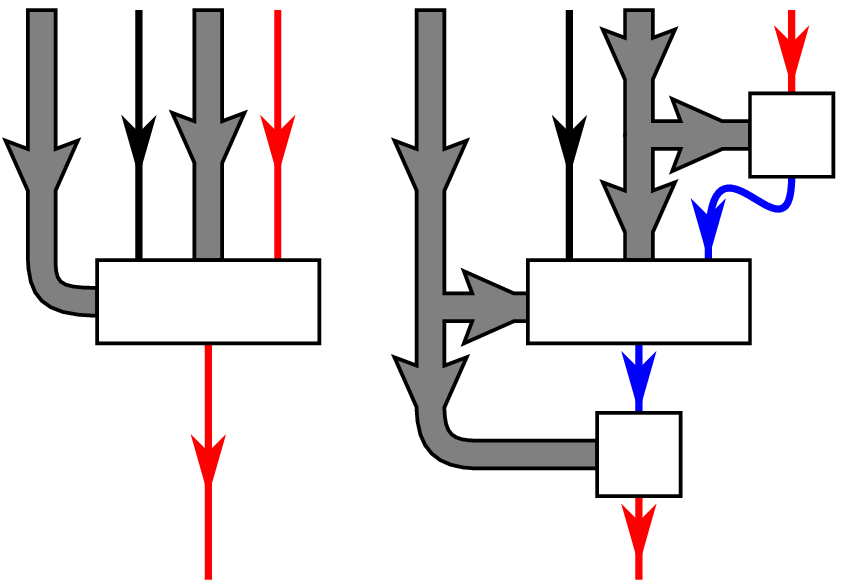}
		\caption{Representation of \Equation~(\ref{eq:nabla_variation_module_sideways}).}
		\label{subfig:nabla_variation_sideways}
	\end{subfigure}
	\begin{subfigure}[h]{.35\linewidth}
		\centering
		\labellist
		\pinlabel $+$ at 112 68
		\small
		\pinlabel $\varphi$ at 92 88
		\pinlabel $\psi$ at 92 40
		\pinlabel $h_M$ at 196 88
		\pinlabel $H$ at 196 40
		\endlabellist
		\includegraphics[scale=.55]{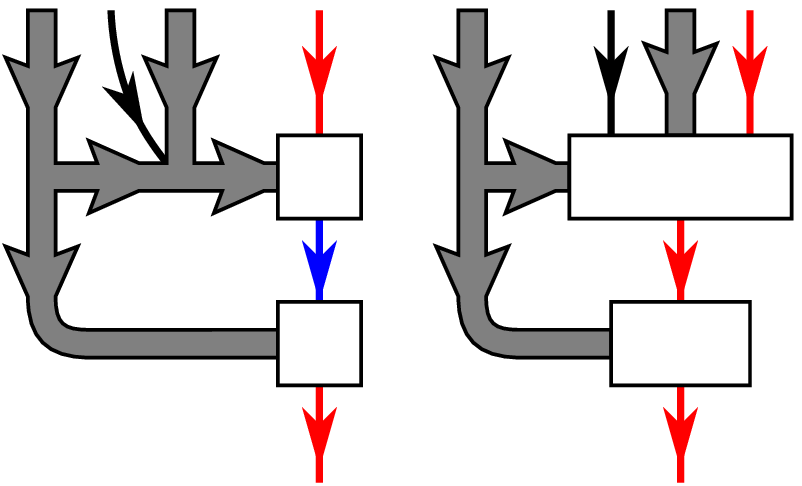}
		\caption{Null-homotopy of~(\ref{eq:nabla_variation_module_sideways}).}
		\label{subfig:nabla_var_homotopy_sideways}
	\end{subfigure}
	\begin{subfigure}[h]{.62\linewidth}
		\centering
		\labellist
		\pinlabel $+$ at 176 68
		\small
		\pinlabel \rotatebox[origin=c]{90}{$\varphi$} at 56 96
		\pinlabel \rotatebox[origin=c]{90}{$\psi$} at 108 96
		\pinlabel \rotatebox[origin=c]{90}{$m_M$} at 244 96
		\pinlabel \rotatebox[origin=c]{90}{$H$} at 296 96
		\endlabellist
		\includegraphics[scale=.55]{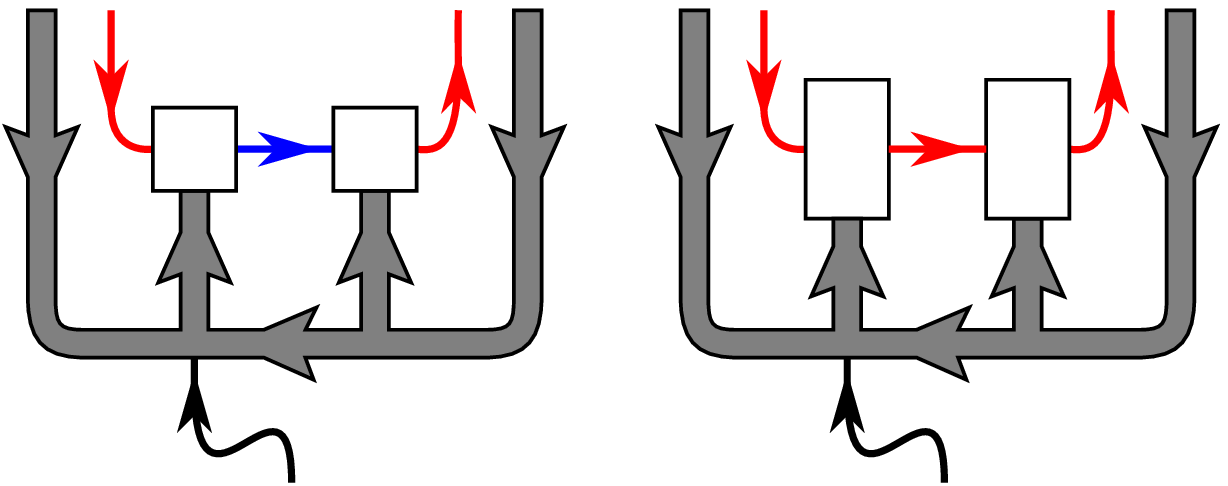}
		\caption{Null-homotopy of~(\ref{eq:nabla_variation_module}).}
		\label{subfig:nabla_var_homotopy_original}
	\end{subfigure}
	\caption{Diagrams from the proof of \Proposition~\ref{prop:nabla_invariance_module}.}
	\label{fig:join_variation_module_proof}
\end{figure}

\begin{proof}[Proof of \Proposition~\ref{prop:nabla_invariance_algebra}]
	Recall the statement of \Proposition~\ref{prop:nabla_invariance_algebra}. We are given two differential graded
	algebras $A$ and $B$, and three modules---$\li{_B}X^A$, $\li{_A}Y^B$, and $\li{_A}M$. We assume that there are
	homotopy equivalences $X\sqtens Y\simeq\li{_B}\II^B$ and $Y\sqtens X\simeq\li{_A}\II^A$, and
	that $X\sqtens A^\vee\sqtens X^\vee$ and $B^\vee$ have homologies of the same rank.
	
	We have to construct a homotopy equivalence $\varphi_X\co X\sqtens A^\vee\sqtens X^\vee\to B^\vee$, and a homotopy
	$\nabla_{X\sqtens M}\simeq\varphi_X\circ(\id_X\sqtens\nabla_M\sqtens\id_{X^\vee})$.

	We start by constructing the morphism $\varphi$. We can define it by the following equation:
	\begin{multline}
		\label{eq:phi_XAX_to_B}
		\left< (\varphi_X)_{i|1|j}(b_1,\ldots,b_i,~(x,a^\vee,x'^\vee),~b'_1,\ldots,b'_j),b''\right>\\
		=\left< \delta_{i+j+1|1|1}(b'_1,\ldots,b'_j,b'',b_1,\ldots,b_i,~x), (x',a)^\vee \right>.
	\end{multline}
	
	Again, it is useful to ``turn it sideways''. We can reinterpret $\varphi_X$ as a morphism of type--$AD$ modules
	$B\dtens X\to X$. In fact, it is precisely the canonical homotopy equivalence $h_X$ between the two. Diagrams for
	$\varphi_X$ and $h_X$ are shown in \Figure~\ref{fig:phi_XAX_to_B_both_ways}. 
	Since the $h_X$ is a homomorphism, it follows that $\varphi_X$ is one as well.
	\begin{figure}
		\begin{subfigure}[t]{.32\linewidth}
			\centering
			\labellist
			\small
			\pinlabel \rotatebox[origin=c]{90}{$\delta_X$} at 68 92
			\hair=1.5pt
			\pinlabel $B$ [b] at 12 136
			\pinlabel $X$ [b] at 40 136
			\pinlabel $A$ [b] at 68 136
			\pinlabel $X$ [b] at 96 136
			\pinlabel $B$ [b] at 124 136
			\pinlabel $B$ [t] at 68 16
			\endlabellist
			\includegraphics[scale=.55]{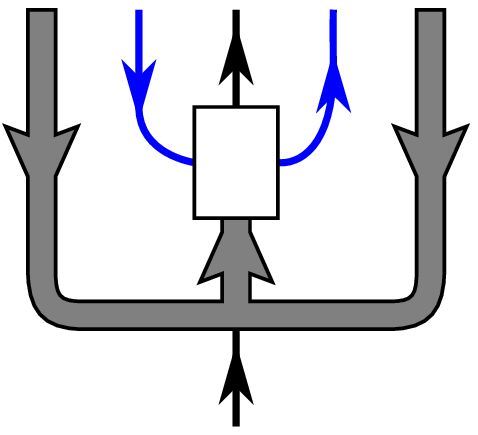}
			\caption{Definition of $\varphi$.}
			\label{subfig:phi_XAX_to_B}
		\end{subfigure}
		\begin{subfigure}[t]{.65\linewidth}
			\centering
			\labellist
			\pinlabel $=$ at 116 76
			\small
			\pinlabel $h_X$ at 60 76
			\pinlabel $\delta_X$ at 224 76
			\hair=1.5pt
			\pinlabel $B$ [b] at 12 136
			\pinlabel $B$ [b] at 36 136
			\pinlabel $B$ [b] at 60 136
			\pinlabel $X$ [b] at 84 136
			\pinlabel $X$ [t] at 60 16
			\pinlabel $A$ [t] at 104 16
			\pinlabel $B$ [b] at 136 136
			\pinlabel $B$ [b] at 156 136
			\pinlabel $B$ [b] at 176 136
			\pinlabel $X$ [b] at 224 136
			\pinlabel $X$ [t] at 224 16
			\pinlabel $A$ [t] at 256 16
			\endlabellist
			\includegraphics[scale=.55]{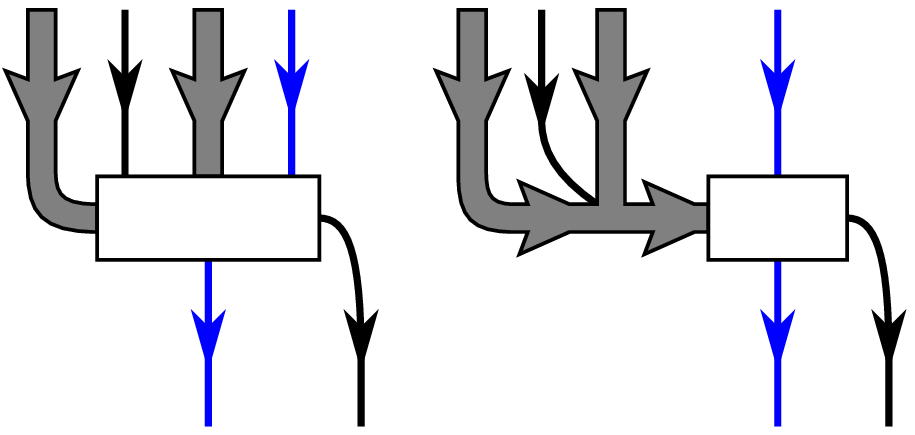}
			\caption{View as a map $B\dtens X\to X$.}
			\label{subfig:phi_XAX_to_B_sideways}
		\end{subfigure}
		\caption{Two views of the homotopy equivalence $\varphi$ from \Equation~(\ref{eq:phi_XAX_to_B}).}
		\label{fig:phi_XAX_to_B_both_ways}
	\end{figure}

	Next we show that $\nabla_{X\sqtens M}$ is homotopic to $\varphi_X\circ(\id_X\sqtens\nabla_M\sqtens\id_{X^\vee})$.
	They are in fact equal. This is best seen in \Figure~\ref{fig:nabla_XM}. We use the fact that $\ol\delta_X$ and
	$\delta_X$ commute with merges and splits.

	\begin{figure}
		\labellist
		\pinlabel $=$ at 232 118
		\pinlabel $=$ at 400 118
		\small
		\pinlabel $\ol\delta_X$ at 60 140
		\pinlabel \rotatebox[origin=c]{180}{$\ol\delta_X$} at 164 140
		\pinlabel \rotatebox[origin=c]{90}{$m_M$} at 112 184
		\pinlabel \rotatebox[origin=c]{90}{$\delta_X$} at 112 92
		\pinlabel \rotatebox[origin=c]{90}{$m_M$} at 316 164
		\pinlabel \rotatebox[origin=c]{90}{$\ol\delta_X$} at 316 92
		\pinlabel \rotatebox[origin=c]{90}{$m_{X\sqtens M}$} at 484 128
		\hair=1.5pt
		\pinlabel $B$ [b] at 16 220
		\pinlabel $X$ [b] at 60 220
		\pinlabel $M$ [b] at 88 220
		\pinlabel $M$ [b] at 136 220
		\pinlabel $X$ [b] at 164 220
		\pinlabel $B$ [b] at 208 220
		\pinlabel $B$ [t] at 112 16
		\pinlabel $B$ [b] at 256 220
		\pinlabel $X$ [b] at 276 220
		\pinlabel $M$ [b] at 292 220
		\pinlabel $M$ [b] at 340 220
		\pinlabel $X$ [b] at 356 220
		\pinlabel $B$ [b] at 376 220
		\pinlabel $B$ [t] at 316 16
		\pinlabel $B$ [b] at 424 220
		\pinlabel $X$ [b] at 444 220
		\pinlabel $M$ [b] at 460 220
		\pinlabel $M$ [b] at 508 220
		\pinlabel $X$ [b] at 524 220
		\pinlabel $B$ [b] at 544 220
		\pinlabel $B$ [t] at 484 16
		\endlabellist
		\includegraphics[scale=.6]{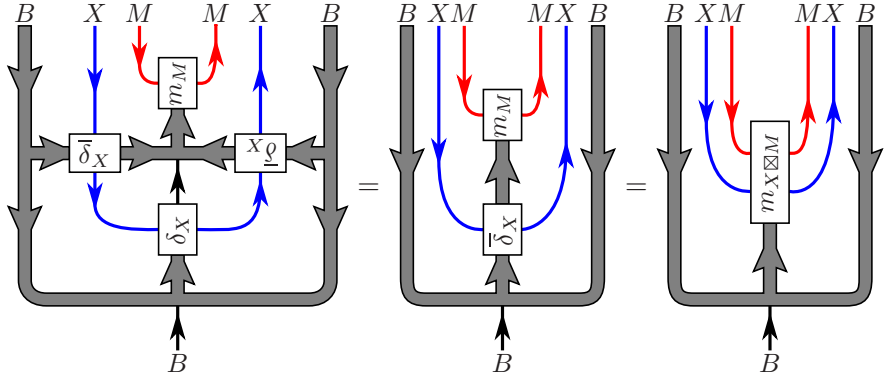}
		\caption{Equality of the direct and induced $\nabla$ maps for $X\sqtens M$.}
		\label{fig:nabla_XM}
	\end{figure}
	
	Finally, we need to show that $\varphi_X$ is a homotopy equivalence. We will do that by constructing a
	right homotopy	inverse for it.
	Combined with the fact that the homologies of the two sides have equal rank, this is enough to ascertain that
	it is indeed a homotopy equivalence.

	Recall that $X\sqtens Y\simeq\II$. Thus there exist morphisms of type--$AD$ $B,B$--bimodules
	$f\co\II\to X\sqtens Y$, and $g\co X\sqtens Y\to\II$, and a null-homotopy $H\co\II\to\II$ of
	$\id_\II-g\circ f$. Note that $g^\vee\co\II^\vee\to Y^\vee \sqtens X^\vee$ is a map of type--$DA$--modules,
	and $(\li{_B}\II^B)^\vee=\lu{B}\II_B$.

	Let $\varphi_Y\co Y\sqtens B^\vee\sqtens Y^\vee\to A$ be defined analogous to $\varphi_X$. Construct the homomorphism
	\begin{multline*}
		\psi=(\id_X\sqtens\varphi_Y\sqtens\id_{X^\vee})\circ(f\sqtens\id_{B^\vee}\sqtens\id_{Y^\vee}\sqtens\id_{X^\vee})
		\circ(\id_{\II}\sqtens\id_B\sqtens g^\vee)\co\\
		\II\sqtens B^\vee\sqtens\II \to X\sqtens A^\vee \sqtens X^\vee.
	\end{multline*}

	We need to show that $\varphi_X\circ\psi$ is homotopic to $\id_{B^\vee}$, or more precisely to the canonical isomorphism
	$\iota\co\II\sqtens B^\vee\sqtens\II\to B^\vee$.
	A graphical representation of $\varphi_X\circ\psi$ is shown in \Figure~\ref{subfig:phi_phi_unsimplified}.
	It simplifies significantly, due to the fact that $B$ is a DG-algebra, and $\mu_B$ only has two nonzero terms.
	The simplified version of $\varphi_X\circ\psi$ is shown in \Figure~\ref{subfig:phi_phi_simplified}.
	As usual, it helps to turn the diagram sideways. We can view it as a homomorphism $B\dtens\II\to\II$ of
	type--$AD$ $B,B$--bimodules. As can be seen from \Figure~\ref{subfig:phi_phi_sideways}, we get
	the composition
	\begin{equation}
		\label{eq:phi_phi_sideways}
		g\circ(h_X\sqtens\id_Y)\circ(\id_B\dtens f)=
		g\circ h_{X\sqtens Y}\circ(\id_B\dtens f)\co B\dtens\II\to\II.
	\end{equation}
	
	\begin{figure}
		\begin{subfigure}[h]{\linewidth}
			\centering
			\labellist
			\small
			\pinlabel $\ol\delta_X$ at 80 232
			\pinlabel $\ol\delta_X$ at 80 144
			\pinlabel \rotatebox[origin=c]{90}{$\delta_X$} at 204 96
			\pinlabel \rotatebox[origin=c]{180}{$\ol\delta_X$} at 328 144
			\pinlabel \rotatebox[origin=c]{180}{$\ol\delta_X$} at 328 284
			\pinlabel \rotatebox[origin=c]{180}{$\ol\delta_X$} at 328 336
			\pinlabel $\ol\delta_Y$ at 136 232
			\pinlabel \rotatebox[origin=c]{90}{$\delta_Y$} at 204 192
			\pinlabel \rotatebox[origin=c]{180}{$\ol\delta_Y$} at 272 284
			\pinlabel \rotatebox[origin=c]{180}{$\ol\delta_Y$} at 272 336
			\pinlabel \rotatebox[origin=c]{180}{$\mu_B$} at 204 328
			\pinlabel \rotatebox[origin=c]{180}{$\mu_B$} at 204 432
			\pinlabel $f$ at 108 284
			\pinlabel \rotatebox[origin=c]{180}{$g$} at 300 388
			\hair=1.5pt
			\pinlabel $B$ [b] at 12 472
			\pinlabel $\li{_B}\II^B$ [b] at 108 472
			\pinlabel $B$ [b] at 204 472
			\pinlabel $\lu{B}\II_B$ [b] at 300 472
			\pinlabel $B$ [b] at 396 472
			\hair=1.5pt
			\pinlabel $B$ [t] at 204 16
			\endlabellist
			\includegraphics[scale=.6]{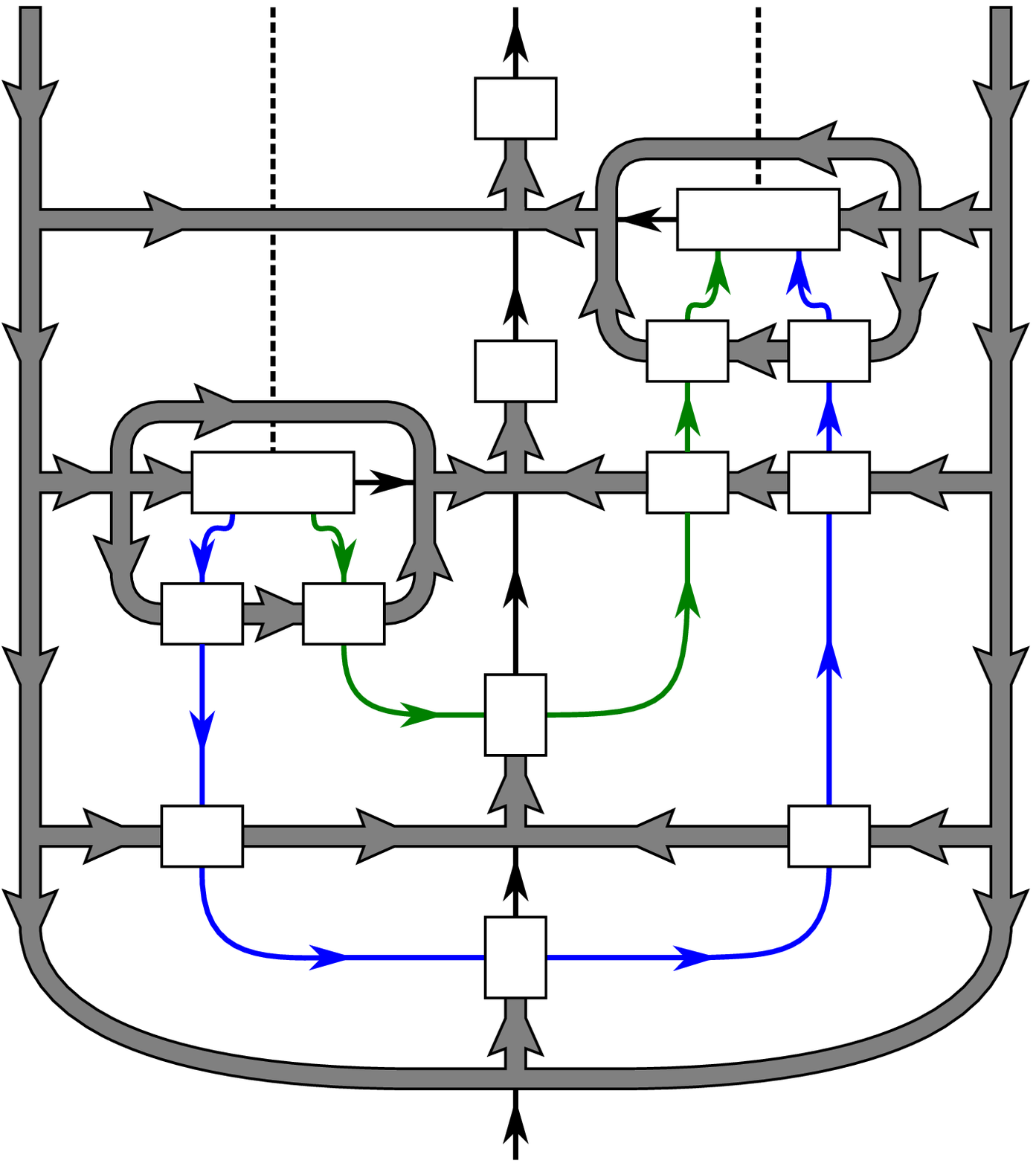}
			\caption{Before simplification.}
			\label{subfig:phi_phi_unsimplified}
		\end{subfigure}
		\begin{subfigure}[h]{.48\linewidth}
			\centering
			\labellist
			\small
			\pinlabel $\ol\delta_X$ at 56 132
			\pinlabel \rotatebox[origin=c]{90}{$\delta_X$} at 120 88
			\pinlabel \rotatebox[origin=c]{180}{$\ol\delta_X$} at 184 132
			\pinlabel \rotatebox[origin=c]{90}{$\delta_Y$} at 120 176
			\pinlabel \rotatebox[origin=c]{180}{$\mu_B$} at 120 236
			\pinlabel \rotatebox[origin=c]{180}{$\mu_B$} at 120 292
			\pinlabel $f$ at 72 204
			\pinlabel \rotatebox[origin=c]{180}{$g$} at 168 264
			\hair=1.5pt
			\pinlabel $B$ [b] at 12 328
			\pinlabel $\li{_B}\II^B$ [b] at 72 328
			\pinlabel $B$ [b] at 120 328
			\pinlabel $\lu{B}\II_B$ [b] at 168 328
			\pinlabel $B$ [b] at 228 328
			\hair=1.5pt
			\pinlabel $B$ [t] at 120 16
			\endlabellist
			\includegraphics[scale=.6]{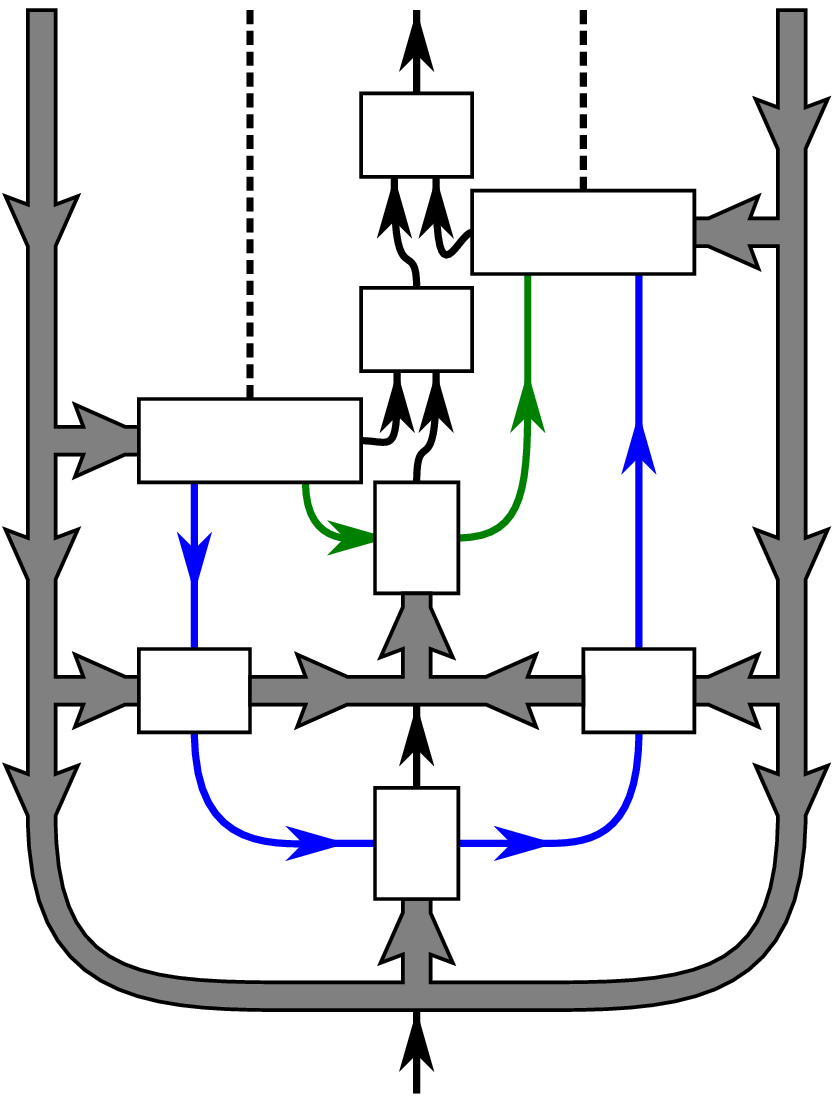}
			\caption{After simplification.}
			\label{subfig:phi_phi_simplified}
		\end{subfigure}
		\begin{subfigure}[h]{.48\linewidth}
			\centering
			\labellist
			\small
			\pinlabel $\ol\delta_X$ at 96 228
			\pinlabel $\delta_X$ at 96 176
			\pinlabel $\ol\delta_X$ at 96 124
			\pinlabel $\delta_Y$ at 184 176
			\pinlabel $\mu_B$ at 232 128
			\pinlabel $\mu_B$ at 232 56
			\pinlabel $f$ at 140 280
			\pinlabel $g$ at 140 72
			\hair=1.5pt
			\pinlabel $B$ [b] at 12 328
			\pinlabel $B$ [b] at 32 328
			\pinlabel $B$ [b] at 52 328
			\pinlabel $\li{_B}\II^B$ [b] at 140 328
			\hair=1.5pt
			\pinlabel $\li{_B}\II^B$ [t] at 140 20
			\pinlabel $B$ [t] at 232 20
			\endlabellist
			\includegraphics[scale=.6]{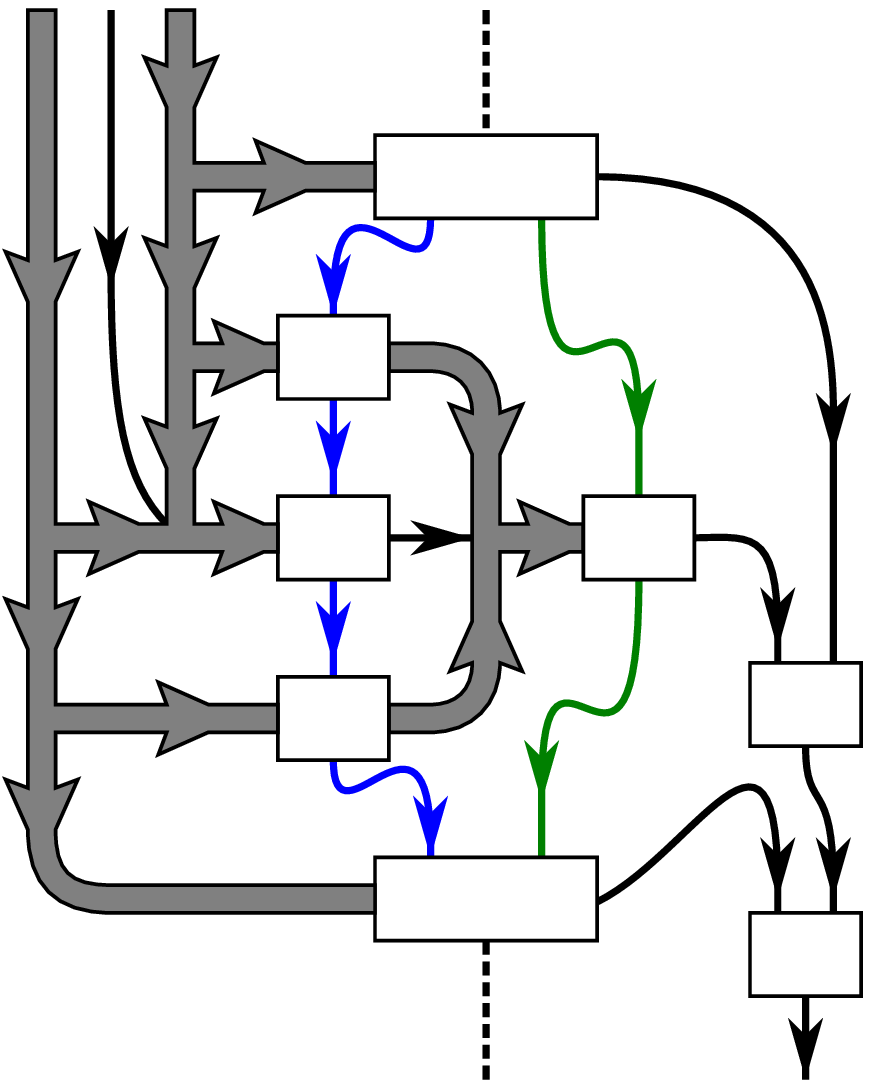}
			\caption{Written sideways.}
			\label{subfig:phi_phi_sideways}
		\end{subfigure}
		\caption{Three views of $\varphi_X\circ\psi\co\II\sqtens B^\vee\sqtens\II\to B^\vee$.}
		\label{fig:phi_phi}
	\end{figure}

	On the other hand, the homomorphism $\iota\co\II\sqtens B^\vee\sqtens\II\to B^\vee$, if written sideways, becomes
	the homotopy equivalence $h_{\II}\co B\dtens\II\to\II$. See \Figure~\ref{fig:iota_h_equivalence} for the
	calculation. In the second step we use some new notation. The caps on the thick strands denote a map $\BBar B\to K$
	to the ground ring, which is the identity on $B^{\otimes 0}$, and zero on $B^{\otimes i}$ for any $i>0$. The dots
	on the $\II$ strands denote the canonical isomorphism of $\II\sqtens B^\vee\sqtens\II$ and $B^\vee$ as
	modules over the ground ring.

	\begin{figure}
		\labellist
		\pinlabel $=$ at 132 76
		\pinlabel $=$ at 276 76
		\pinlabel $\longleftrightarrow$ at 456 76
		\small
		\pinlabel $\iota$ at 60 76
		\pinlabel \rotatebox[origin=c]{90}{$\delta_\II$} at 356 92
		\pinlabel $h_\II$ at 548 76
		\hair=1.5pt
		\pinlabel $B$ [b] at 12 136
		\pinlabel $\II$ [b] at 40 136
		\pinlabel $B$ [b] at 60 136
		\pinlabel $\II$ [b] at 80 136
		\pinlabel $B$ [b] at 108 136
		\pinlabel $B$ [t] at 60 16
		\pinlabel $B$ [b] at 156 136
		\pinlabel $\II$ [b] at 184 136
		\pinlabel $B$ [b] at 204 136
		\pinlabel $\II$ [b] at 224 136
		\pinlabel $B$ [b] at 252 136
		\pinlabel $B$ [t] at 204 16
		\pinlabel $B$ [b] at 300 136
		\pinlabel $\II$ [b] at 328 136
		\pinlabel $B$ [b] at 356 136
		\pinlabel $\II$ [b] at 384 136
		\pinlabel $B$ [b] at 411 136
		\pinlabel $B$ [t] at 356 16
		\pinlabel $B$ [b] at 500 136
		\pinlabel $B$ [b] at 524 136
		\pinlabel $B$ [b] at 548 136
		\pinlabel $\II$ [b] at 572 136
		\pinlabel $\II$ [t] at 548 16
		\pinlabel $B$ [t] at 592 16
		\endlabellist
		\includegraphics[scale=.59]{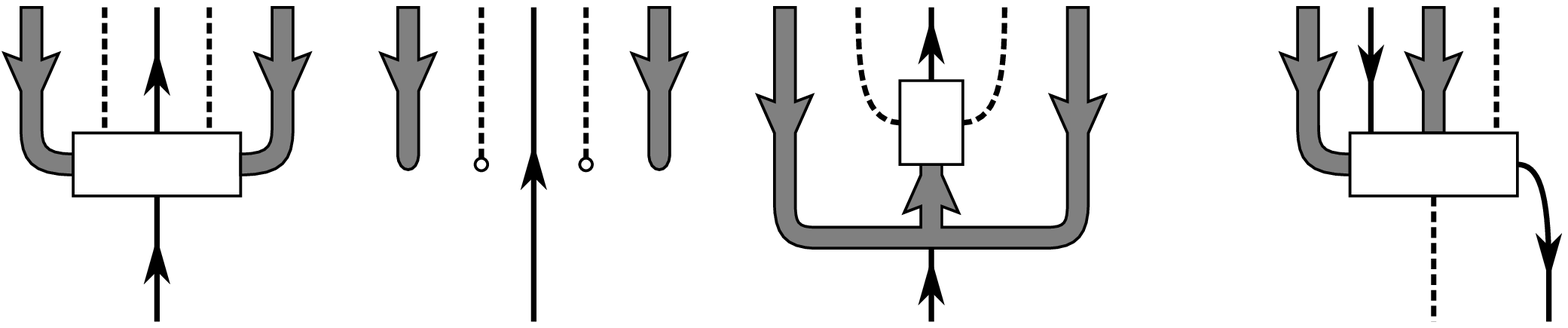}
		\caption{The equivalence of the morphism $\iota$ and $h_\II$.}
		\label{fig:iota_h_equivalence}
	\end{figure}
	
	Finding a null-homotopy for $\iota+\varphi_X\circ\psi$ is equivalent to finding a null-homotopy $B\dtens\II\to\II$ of
	$h_{\II}+g\circ h_{X\sqtens Y}\circ(\id_B\dtens f)$. There is a null-homotopy
	$\zeta_f\co B\dtens\II\to B\dtens\sqtens X\sqtens Y$ of $f\circ h_\II+h_{X\sqtens Y}\circ(\id_B\dtens f)$.
	Recall that $H$ was a null-homotopy of $\id_\II+g\circ f$.
	Thus we have
	\begin{align*}
		\del(H\circ h_\II + g\circ\zeta_f)&=(\id_\II\circ h_\II+g\circ f\circ h_\II)\\
		&+ (g\circ f\circ h_{\II}+g\circ h_{X\sqtens Y}\circ(\id_B\dtens F)\\
		&= h_{\II} + g\circ h_{X\sqtens Y}\circ(\id_B\dtens F),
	\end{align*}
	giving us the required null-homotopy.

	To finish the proof, notice that if $\varphi_X\circ\psi$ is homotopic to $\id_B$, then it is a quasi-isomorphism,
	i.e. a homomorphism whose scalar component is a quasi-isomorphism of chain
	complexes. Moreover, when working with $\ZZ/2$--coefficients, as we do, quasi-isomorphisms of $\Ainf$--modules
	and bimodules coincide with homotopy equivalences.

	In particular we have that $(\varphi_X\circ\psi)_{0|1|0}=(\varphi_X)_{0|1|0}\circ\psi_{0|1|0}$ induces an isomorphism on
	homology (in this case the identity map on homology). In particular $\psi$ induces an injection, while $\varphi_X$
	induces a surjection. Combined with the initial assumption that $B^\vee$ and $X\sqtens A^\vee \sqtens X^\vee$ have
	homologies of equal rank, this implies that $(\varphi_X)_{0|1|0}$ and $\psi_{0|1|0}$ induce isomorphisms on homology.
	That is, $\varphi_X$ and $\psi$ are quasi-isomorphisms, and so homotopy equivalences. This concludes the proof
	of \Proposition~\ref{prop:nabla_invariance_algebra}, and with it, of \Theorem~\ref{thm:join_map_invariance}.
\end{proof}

\section{Properties of the join map}
\label{sec:join_properties}

In this section we give some formulas for the join and gluing maps, and prove their formal properties.

\subsection{Explicit formulas}
\label{sec:join_alternatives}

We have abstractly defined the join map $\Psi_{\W}$ in terms of $\nabla_{\BSA(\W)}$ but so far have not given any
explicit formula for it. Here we give the general formula, as well as some special cases which are somewhat simpler.

If we want to compute $\Psi_{\W}$ for the join $\Y_1\Cup_{\W}\Y_2$, we need to pick
a parametrization by an arc diagram $\Z$, with associated algebra $A$, and
 representatives $U$ for $\BSD(\Y_1)^A$, $V$ for $\lu{A}\BSD(\Y_2)$, and $M$ for $\li{_A}\BSA(\W)$.
Then we know $\SFC(\Y_1)=U\sqtens M$, $\SFC(\Y_2)=M^\vee\sqtens V$, and $\SFC(\Y_1\Cup_{\W}\Y_2)=U\sqtens A^\vee\sqtens V$.
As given in \Definition~$\ref{def:join_map_def}$, the join map $\Psi_{\W}$ is
\begin{equation*}
	\Psi_{\W}=\id_U\sqtens\nabla_M\sqtens\id_V\co U\sqtens M\otimes M^\vee\sqtens V\to U\sqtens A^\vee \sqtens V.
\end{equation*}
In graphic form this can be seen in \Figure~\ref{subfig:join_full_general}.

This general form is not good for computations, especially if we try to write it algebraically.
However $\Psi_\W$ has a much simpler form when $M$ is a \emph{DG-type} module.

\begin{defn}
	\label{def:nice_module}
	An $\Ainf$--module $M_A$ is \emph{of DG-type} if it is a DG-module, i.e., if its structure maps $m_{1|i}$ vanish for $i\geq 2$. A
	bimodule $\li{_A}M_B$ is \emph{of DG-type} if $m_{i|1|j}$ vanish, unless $(i,j)$ is one of $(0,0)$, $(1,0)$ or
	$(0,1)$ (i.e. it is a DG-module over $A\otimes B$).

	A type--$DA$ bimodule $\lu{A}M_B$ is \emph{of DG-type} if $\delta_{1|1|j}$ vanish for all $j\geq 2$. A type--$DD$ bimodule
	$\lu{A}M^B$ is \emph{of DG-type} if $\delta_{1|1|1}(x)$ is always in $A\otimes X\otimes 1+1\otimes X\otimes B$
	(i.e. it is separated). All type $D$--modules $M^A$ are DG-type.
\end{defn}

The $\sqtens$--product of any combination of DG-type modules is also DG-type.
All modules $\BSA$, $\BSD$, $\BSAA$, etc., computed from a nice diagram are of DG-type.

\begin{prop}
	\label{prop:join_for_nice}
	Let the manifolds $\Y_1$, $\Y_2$, and $\W$, and the modules $U$, $V$, and $M$ be as in the above discussion.
	If $M$ is DG-type, the formula for the join map $\Psi_\W$ simplifies to:
	\begin{equation}
		\label{eq:join_nice_formula}
		\Psi_\W(u \sqtens m \otimes n^\vee \sqtens v)=
		\sum_a \left<m_M(a,m),n^\vee\right> \cdot u\sqtens{a}^\vee \sqtens v,
	\end{equation}
	where the sum is over a $\ZZ/2$--basis for $A$. A graphical representation is given
	in \Figure~\ref{subfig:join_full_nice}.
\end{prop}

Finally, an even simpler case is that of elementary modules. We will see later that elementary modules play
an important role for gluing, and for the relationship between the bordered and sutured theories.

\begin{defn}
	\label{def:elementary_module}
	A type--$A$ module $\li{_A}M$ (or similarly $M_A$) is called \emph{elementary} if the following conditions hold:
	\begin{enumerate}
		\item $M$ is generated by a single element $m$ over $\ZZ/2$.
		\item All structural operations on $M$ vanish (except for multiplication by an idempotent, which
			might be identity).
	\end{enumerate}

	A type--$D$ module $\lu{A}M$ (or $M^A$), is called \emph{elementary} if the following conditions hold:
	\begin{enumerate}
		\item $M$ is generated by a single element $m$ over $\ZZ/2$.
		\item $\delta(m)=0$.
	\end{enumerate}
\end{defn}

Notice that for an elementary module $M=\{0,m\}$ we can decompose $m$ as a sum
$m=\iota_1 m+\cdots+\iota_k m$, where $(\iota_i)$ is the canonical basis of the ground ring. Thus we must have $\iota_i m=m$ for some
$i$, and $\iota_j m=0$ for all $i\neq j$. Therefore, elementary (left) modules over $A$ are in a $1$--to--$1$ correspondence
with the canonical basis for its ground ring.

We only use elementary type--$A$ modules in this section but we will need both types later.

\begin{rmk}
	For the algebras we discuss, the elementary type--$A$ modules are precisely the simple modules. The elementary
	type--$D$
	modules are the those $\lu{A}M$ for which $A\sqtens M\in\li{_A}\Mod$ is an elementary projective module.
\end{rmk}

\begin{prop}
	\label{prop:join_for_elementary}
	If $\li{_A}M=\{m,0\}$ is an elementary module corresponding to the basis idempotent $\iota_M$, then the join map $\Psi_\W$ reduces to
	\begin{equation}
		\label{eq:join_elementary_formula}
		\Psi_{\W}(u\sqtens m\otimes m^\vee\sqtens v)=
		u\sqtens {\iota_M}^\vee \sqtens v.
	\end{equation}

	Graphically, this is given in \Figure~\ref{subfig:join_full_elementary}.

	Moreover, in this case, $\SFC(\Y_1)=U\sqtens M\cong U\cdot \iota_M\subset U$
	and $\SFC(\Y_2)=M\sqtens V\cong \iota_M\cdot V\subset V$ as chain complexes.
\end{prop}

\begin{figure}
	\begin{subfigure}[t]{.32\linewidth}
		\centering
		\labellist
		\small
		\pinlabel $\ol\delta_U$ at 16 56
		\pinlabel $\ol\delta_V$ at 120 56
		\pinlabel \rotatebox[origin=c]{90}{$m_M$} at 68 104
		\hair=1.5pt
		\pinlabel $U$ [b] at 16 136
		\pinlabel $U$ [t] at 16 16
		\pinlabel $V$ [b] at 120 136
		\pinlabel $V$ [t] at 120 16
		\pinlabel $M$ [b] at 44 136
		\pinlabel $M$ [b] at 92 136
		\pinlabel $A$ [t] at 68 16
		\endlabellist
		\includegraphics[scale=.6]{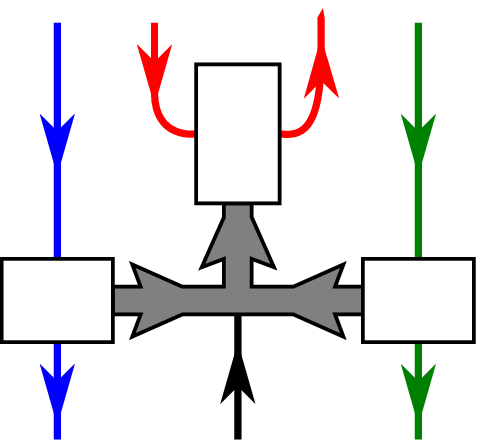}
		\caption{The general case.}
		\label{subfig:join_full_general}
	\end{subfigure}
	\begin{subfigure}[t]{.32\linewidth}
		\centering
		\labellist
		\small
		\pinlabel \rotatebox[origin=c]{90}{$m_M$} at 68 104
		\hair=1.5pt
		\pinlabel $U$ [b] at 16 136
		\pinlabel $U$ [t] at 16 16
		\pinlabel $V$ [b] at 120 136
		\pinlabel $V$ [t] at 120 16
		\pinlabel $M$ [b] at 44 136
		\pinlabel $M$ [b] at 92 136
		\pinlabel $A$ [t] at 68 16
		\endlabellist
		\includegraphics[scale=.6]{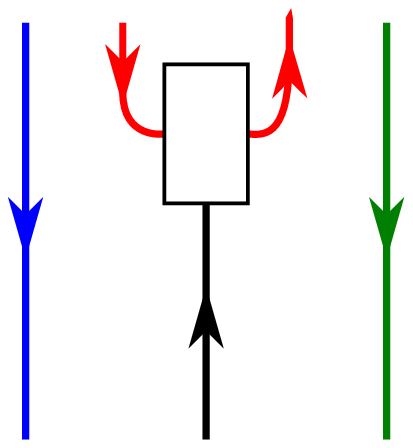}
		\caption{$M$ of DG-type.}
		\label{subfig:join_full_nice}
	\end{subfigure}
	\begin{subfigure}[t]{.32\linewidth}
		\centering
		\labellist
		\small
		\pinlabel $\iota_M^\vee$ at 68 76
		\hair=1.5pt
		\pinlabel $U$ [b] at 16 136
		\pinlabel $U$ [t] at 16 16
		\pinlabel $V$ [b] at 120 136
		\pinlabel $V$ [t] at 120 16
		\pinlabel $M$ [b] at 44 136
		\pinlabel $M$ [b] at 92 136
		\pinlabel $A$ [t] at 68 16
		\endlabellist
		\includegraphics[scale=.6]{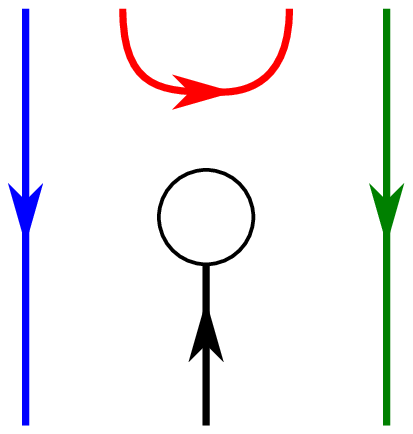}
		\caption{$M$ elementary.}
		\label{subfig:join_full_elementary}
	\end{subfigure}
	\caption{Full expression for join map in three cases.}
	\label{fig:join_full_all_cases}
\end{figure}

\Proposition~\ref{prop:join_for_nice} and \Proposition~\ref{prop:join_for_elementary} follow directly from the
definitions of DG-type and elementary modules.

\subsection{Formal properties}
\label{sec:join_basic_properties}

In this section we will show that the join map has the formal properties stated in
\Theorem~\ref{thm:intro_join_and_gluing}. A more precise statement of the properties is given below.

\begin{thm}
	\label{thm:join_nice_properties}
	The following properties hold:
	\begin{enumerate}
		\item \label{subthm:join_symmetric} Let $\Y_1$ and $\Y_2$ be sutured and $\W$ be partially
			sutured, with embeddings $\W\into\Y_1$ and $-\W\into\Y_2$. There are natural identifications
			of the disjoint unions $\Y_1\sqcup\Y_2$ and $\Y_2\sqcup\Y_1$,
			and of of the joins $\Y_1\Cup_{\W}\Y_2$ and $\Y_2\Cup_{-\W}\Y_1$.
			Under this identification, there is a homotopy
			\begin{equation*}
				\Psi_{\W}\simeq\Psi_{-\W}.
			\end{equation*}
		\item \label{subthm:join_associative} Let $\Y_1$, $\Y_2$, and $\Y_3$ be sutured, and $\W_1$ and $\W_2$
			be partially sutured, such that there are embeddings $\W_1\into\Y_1$, $(-\W_1\sqcup\W_2)\into\Y_2$,
			and $-\W_2\into\Y_3$. The following diagram commutes up to homotopy:
			{\centering
			\xymatrix@C=1in{
			\SFC(\Y_1\sqcup\Y_2\sqcup\Y_3) \ar[r]^{\Psi_{\W_1}} \ar[d]_{\Psi_{\W_2}} \ar[rd]^{\Psi_{\W_1\sqcup-\W_2}} & \SFC(\Y_1\Cup\Y_2\sqcup\Y_3) \ar[d]^{\Psi_{\W_2}} \\
				\SFC(\Y_1\sqcup\Y_2\Cup\Y_3) \ar[r]_{\Psi_{\W_1}} & \SFC(\Y_1\Cup\Y_2\Cup\Y_3)
				}}
		\item \label{subthm:join_identity} Let $\W$ be partially sutured. There is
			a canonical element $[\Delta_{\W}]$ in the sutured Floer homology $\SFH(\D(\W))$ of the double of
			$\W$. If $\Delta$ is any representative for $[\Delta_{\W}]$, and there is an embedding $\W\into\Y$, then
			\begin{equation}
				\label{eq:join_identity}
				\Psi_{\W}(\cdot,\Delta)\simeq\id_{\SFC(\Y)}\co\SFC(\Y)\to\SFC(\Y).
			\end{equation}
	\end{enumerate}
\end{thm}

\begin{proof}
	We will prove the three parts in order.

	For part~(\ref{subthm:join_symmetric}), take representatives $U^A$ for $\BSD(\Y_1\setminus\W)$, $\lu{A}V$ for $\BSD(\Y_2\setminus-\W)$, and
	$\li{_A}M$ for
	$\BSA(\W)$.
	The main observation here is that we can turn left modules into right modules and vice versa, by reflecting
	all diagrams along the vertical axis (see \Appendix~\ref{sec:duals}). If we reflect the entire diagram for
	$\Psi_\M$, domain and target chain complexes are turned into isomorphic ones and we
	get a new map that is equivalent.

	The domain $U^A\sqtens\li{_A}M\otimes{M^\vee}_A\sqtens\lu{A}V$ becomes
	$V^{A^{\op}}\sqtens\li{_{A^{\op}}}{M^\vee}\otimes M_{A^{\op}}\sqtens\lu{A^{\op}}U$, and the target
	$U^A\sqtens\li{_A}{A^\vee}_A\sqtens\lu{A}V$ becomes
	$V^{A^{\op}}\sqtens\li{_{A^{\op}}}{(A^\vee)^{\op}}_{A^{\op}}\sqtens\lu{A^{\op}}U$.

	Notice that $V^{A^{\op}}$ is $\BSD(\Y_2\setminus-\W)$, $\li{_{A^{\op}}}U$ is $\BSD(\Y_1\setminus\W)$, and $\li{_{A^{\op}}}{M^\vee}$ is
	$\BSA(-\W)$. In addition $(A^\vee)^{\op}=(A^{\op})^\vee$. Since the map $\nabla_M$ is completely symmetric,
	when we reflect it, we get $\nabla_{M^\vee}$. Everything else is preserved, so reflecting $\Psi_{\W}$ gives precisely
	$\Psi_{-\W}$. This finishes  part~(\ref{subthm:join_symmetric}).

	For part~(\ref{subthm:join_associative}),
	the equivalence is best seen by working with convenient representatives. Pick the following modules as
	representatives: $U^A$ for $\BSD(\Y_1\setminus\W_1)$, $\lu{A}X^B$ for $\BSDD(\Y_2\setminus(-\W_1\cup\W_2))$,
	$\lu{B}V$ for $\BSD(\Y_1)$, $\li{_A}M$ for $\BSA(\W_1)$ and $\li{_B}N$ for $\BSD(\W_2)$. We can always choose $M$,
	$N$, and $X$ to be of DG-type in the sense of \Definition~\ref{def:nice_module}. Since $X$ is of DG-type, taking the
	$\sqtens$--product with it is associative. (This is only true up to homotopy in general). Since $M$ and $N$ are
	DG-type, we can apply \Proposition~\ref{prop:join_for_nice} to get formulas for $\Psi_{\W_1}$ and $\Psi_{\W_2}$. The two
	possible compositions are shown in \Figures~\ref{subfig:psi_1_then_2} and~\ref{subfig:psi_2_then_1}.

	To compute
	$\Psi_{\W_1\cup-\W_2}$, notice that $(U\otimes V)^{A,B^{\op}}$ represents $\BSDD(
	(\Y_1\cup\Y_3)\setminus(\W_1\cup-\W_3))$, $\lu{A,B^{\op}}X$ represents
	$\BSDD(\Y_2\setminus(-\W_1\cup\W_2))$, and $\li{_{A,B^{\op}}}(M\otimes N^\vee)$ is a DG-type module representing
	$\BSAA(\W_1\cup-\W_2)$. To compute the join map, we need to convert them to single modules. For type--$DD$ modules,
	this is trivial (any $A,B$--bimodule is automatically an $A\otimes B$--module and vice versa). For type--$AA$
	modules, this could be complicated in general. Luckily, it is easy for DG-type modules. Indeed, if $P_{A,B}$ is
	DG-type,
	the corresponding $A\otimes B$--module $P_{A\otimes B}$ is also DG-type, with algebra action
	\begin{equation*}
		m_{1|1}(\cdot,a\otimes b)=m_{1|1|0}(\cdot,a)\circ m_{1|0|1}(\cdot,b)=m_{1|0|1}(\cdot,b)\circ m_{1|1|0}(\cdot,a).
	\end{equation*}
	In the definition of bimodule invariants in~\cite{Zar:BSFH}, the procedure
	used to get $\BSAA$ from $\BSA$, and $\BSDD$ from $\BSD$ is exactly the reverse of this construction.

	Thus, we can see that $(U\otimes V)^{A\otimes B^{\op}}$ represents
	$\BSD((\Y_1\cup\Y_3)\setminus(\W_1\cup-\W_3))$, $\lu{A\otimes B^{\op}}X$ represents
	$\BSD(\Y_2\setminus(-\W_1\cup\W_2))$, and $\li{_{A\otimes B^{\op}}}(M\otimes N^\vee)$ represents
	$\BSA(\W_1\cup-\W_2)$. It is also easy to check that
	\begin{equation*}
		\li{_A}{A^\vee}_A\otimes\li{_{B^{\op}}}{(B^{\op})^\vee}_{B^{\op}}\cong\li{_{A\otimes B^{\op}}}{(A\otimes
		B^{\op})^\vee}_{A\otimes B^{\op}}.
	\end{equation*}

	We can see a diagram for $\Psi_{\W_1\cup-\W_2}$ in \Figure~\ref{subfig:psi_12_together}.
	By examining the diagrams, we see that the three maps are the same, which finishes 
	part~(\ref{subthm:join_associative}).

	\begin{figure}
		\begin{subfigure}[b]{.32\linewidth}
			\centering
			\labellist
			\small
			\pinlabel \rotatebox[origin=c]{90}{$m_M$} at 52 136
			\pinlabel \rotatebox[origin=c]{90}{$m_N$} at 140 80
			\hair=1.5pt
			\pinlabel $U$ [b] at 8 172
			\pinlabel $M$ [b] at 28 172
			\pinlabel $M$ [b] at 76 172
			\pinlabel $X$ [b] at 96 172
			\pinlabel $N$ [b] at 116 172
			\pinlabel $N$ [b] at 164 172
			\pinlabel $V$ [b] at 184 172
			\pinlabel $U$ [t] at 8 20
			\pinlabel $A$ [t] at 52 20
			\pinlabel $X$ [t] at 96 20
			\pinlabel $B$ [t] at 140 20
			\pinlabel $V$ [t] at 184 20
			\endlabellist
			\includegraphics[scale=.55]{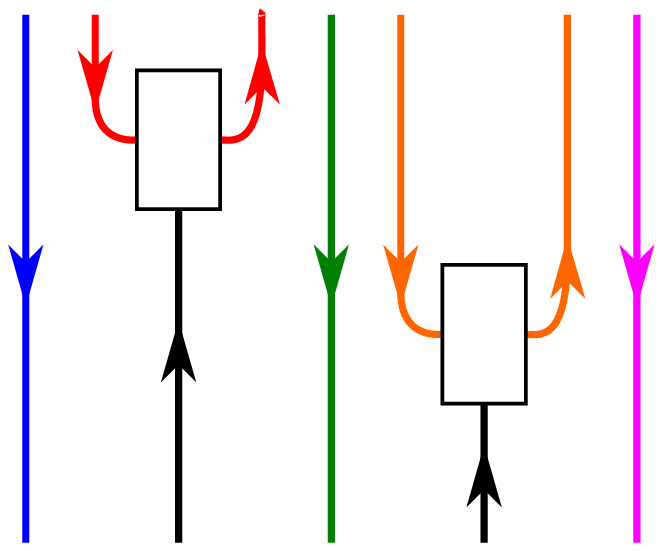}
			\caption{$\Psi_{\W_2}\circ\Psi_{\W_1}$}
			\label{subfig:psi_1_then_2}
		\end{subfigure}
		\begin{subfigure}[b]{.32\linewidth}
			\centering
			\labellist
			\small
			\pinlabel \rotatebox[origin=c]{90}{$m_M$} at 52 80
			\pinlabel \rotatebox[origin=c]{90}{$m_N$} at 140 136
			\hair=1.5pt
			\pinlabel $U$ [b] at 8 172
			\pinlabel $M$ [b] at 28 172
			\pinlabel $M$ [b] at 76 172
			\pinlabel $X$ [b] at 96 172
			\pinlabel $N$ [b] at 116 172
			\pinlabel $N$ [b] at 164 172
			\pinlabel $V$ [b] at 184 172
			\pinlabel $U$ [t] at 8 20
			\pinlabel $A$ [t] at 52 20
			\pinlabel $X$ [t] at 96 20
			\pinlabel $B$ [t] at 140 20
			\pinlabel $V$ [t] at 184 20
			\endlabellist
			\includegraphics[scale=.55]{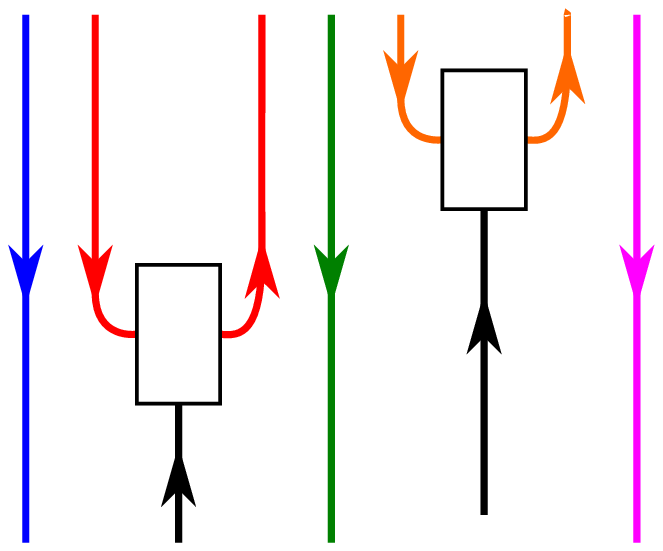}
			\caption{$\Psi_{\W_1}\circ\Psi_{\W_2}$}
			\label{subfig:psi_2_then_1}
		\end{subfigure}
		\begin{subfigure}[b]{.32\linewidth}
			\centering
			\labellist
			\small
			\pinlabel \rotatebox[origin=c]{90}{$m_M$} at 92 80
			\pinlabel \reflectbox{\rotatebox[origin=c]{90}{$m_N$}} at 124 108
			\hair=1.5pt
			\pinlabel $U\!\!\otimes\!\!V$ [b] at 22 172
			\pinlabel $M\!\!\otimes\!\!N^\vee$ [b] at 74 172
			\pinlabel $M^\vee\!\!\!\otimes\!\!N$ [b] at 154 172
			\pinlabel $X$ [b] at 200 172
			\pinlabel {$U\!\!\otimes\!\!V$} [t] at 22 20
			\pinlabel {$A\!\otimes\!B^{\op}$} [t] at 108 20
			\pinlabel $X$ [t] at 200 20
			\endlabellist
			\includegraphics[scale=.55]{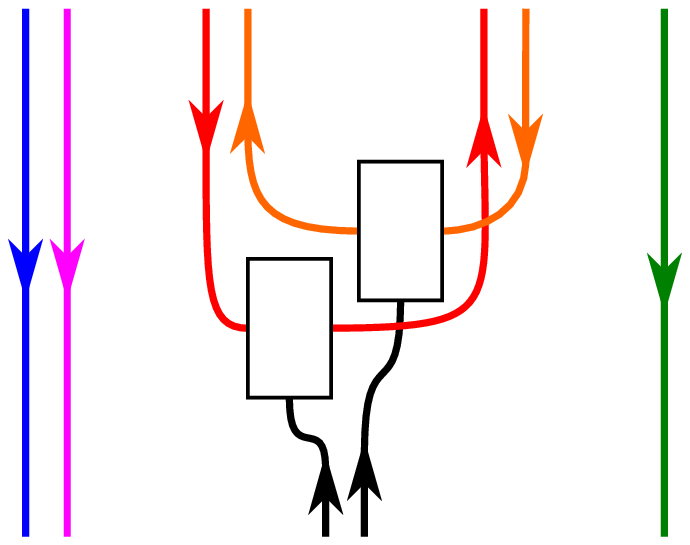}
			\caption{$\Psi_{\W_1\cup-\W_2}$}
			\label{subfig:psi_12_together}
		\end{subfigure}
		\caption{Three ways to join $\Y_1$, $\Y_2$, and $\Y_3$.}
		\label{fig:three_joins}
	\end{figure}

	Part~(\ref{subthm:join_identity}) requires some more work, so we will split it in several steps. We will define $\Delta_M$ for a fixed
	representative $M$ of $\BSD(\W)$.  We will prove that $[\Delta_M]$ does no depend on the choice of $M$.
	Finally, we will use a computational lemma to show that \Equation~(\ref{eq:join_identity}) holds for $\Delta_M$.

	First we will introduce some notation. Given an $\Ainf$--module $\li{_A}M$ over $A=\A(\Z)$, define the \emph{double}
	of $M$ to be
	\begin{equation}
		\label{eq:double_mod_def}
		\D(M)={M^\vee}\sqtens(\lu{A}\II^A\sqtens A\sqtens\lu{A}\II^A)\sqtens M.
	\end{equation}

	Note that if $M=\BSA(\W)$, then $\D(M)=\BSA(-\W)\sqtens\BSDD(\TW_{\F,-})\sqtens\BSA(\W)\simeq\SFC(\D(\W))$.
	Next we define the \emph{diagonal element} $\Delta_M\in\D(M)$ as follows. Pick a basis $(m_1,\ldots,m_k)$ of $M$
	over $\ZZ/2$. Define
	\begin{equation}
		\label{eq:delta_def}
		\Delta_M=\sum_{i=1}^k m_i\sqtens (*\sqtens 1 \sqtens *)\sqtens {m_i}^\vee.
	\end{equation}

	It is easy to check that this definition does not depend on the choice of basis. Indeed there is a really simple
	diagrammatic representation of $\Delta M$, given in \Figure~\ref{fig:delta_def}. We think of it as a linear map from
	$\ZZ/2$ to $\D(M)$. It is also easy to check that $\del\Delta_M=0$. Indeed, writing out the definition of
	$\del\Delta_M$, there are are only two nonzero terms which cancel.

	\begin{figure}
		\labellist
		\pinlabel $=$ at 124 64
		\small
		\pinlabel $\Delta_M$ at 58 104
		\pinlabel $1$ at 190 60
		\hair=1.5pt
		\pinlabel $M$ [t] at 12 20
		\pinlabel $\lu{A}\II^A$ [t] at 34 20
		\pinlabel $A$ [t] at 58 20
		\pinlabel $\lu{A}\II^A$ [t] at 82 20
		\pinlabel $M$ [t] at 104 20
		\pinlabel $M$ [t] at 144 20
		\pinlabel $\lu{A}\II^A$ [t] at 166 20
		\pinlabel $A$ [t] at 190 20
		\pinlabel $\lu{A}\II^A$ [t] at 214 20
		\pinlabel $M$ [t] at 236 20
		\endlabellist
		\includegraphics[scale=.7]{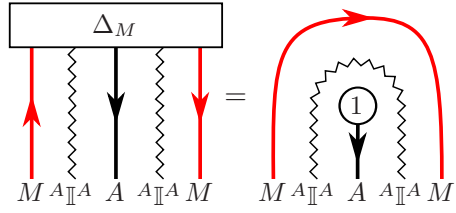}
		\caption{The diagonal element $\Delta_M$.}
		\label{fig:delta_def}
	\end{figure}

	The proof that $[\Delta_M]$ does not depend on the choices of $A$ and $M$ is very similar to the proof
	of \Theorem~\ref{thm:join_map_invariance}, so will omit it. (It involves showing independence from $M$, as well as
	from $A$ via a quasi-invertible bimodule $\lu{A}X_B$.)

	\begin{lem}
		\label{lem:av_a_homotopy_equivalence}
		Let $A$ be a differential graded algebra, coming from an arc diagram $\Z$.
		There is a homotopy equivalence
		\begin{equation*}
			c_A\co\lu{A}\II^A~\sqtens~ A^\vee~\sqtens~\lu{A}\II^A~\sqtens~\li{_A}A_A\to\lu{A}\II_A,
		\end{equation*}
		given by
		\begin{equation*}
			(c_A)_{1|1|0}\left( *\sqtens a^\vee \sqtens*\sqtens b \right)=
			\begin{cases}
				b \otimes * &\textrm{if $a$ is an idempotent,}\\
				0 &\textrm{otherwise.}
			\end{cases}
		\end{equation*}
	\end{lem}

	Here we use $*$ to denote the unique element with compatible idempotents in the two versions of $\II$.
	(Both versions have generators in 1--to--1 correspondence with the basis idempotents.)

	\begin{rmk}
		As we mentioned earlier, one has to be careful when working with type--$DD$ modules. While $\sqtens$ and $\dtens$ are usually associative
		by themselves, and with each other, this might fail  when a $DD$--module is involved, in which case we only
		have associativity up to homotopy equivalence. However, this could be mitigated in two situations. If the
		$DD$--module is DG-type (which fails for $\lu{A}\II^A$), or if the type--$A$ modules on both sides are DG-type, then
		true associativity still holds. This is true for $A$ and $A^\vee$, so the statement of the lemma makes sense.
	\end{rmk}

	\begin{proof}[Proof of Lemma~\ref{lem:av_a_homotopy_equivalence}]
		Note that we can easily see that there is \emph{some}
		homotopy equivalence $(\II\sqtens A^\vee\sqtens\II)\sqtens A\simeq\II$, since the left-hand side is
		\begin{equation*}
			\BSDD(\TW_{\F,+})\sqtens\BSAA(\TW_{-\ol\F,-})\simeq\BSDA(\TW_{\F,+}\cup\TW_{-\ol\F,-}),
		\end{equation*}
		while the right side is
		$\BSDA(\F\times[0,1])$, and those bordered sutured manifolds are the same.
		The difficulty is in finding the precise homotopy equivalence, which we need for computations, in order to
		``cancel'' $A^\vee$ and $A$.

		First, we need to show that $c_A$ is a homomorphism. This is best done graphically.
		The definition of $c_A$ is represented in \Figure~\ref{fig:av_a_h_def}. The notation we use there is that $\lu{A}\II^A$
		is a jagged line, without a direction, since $\II$ is its own dual. $\lu{A}\II_A$ is represented by a
		dashed line. As before the line can start or end with a dot, signifying the canonical isomorphism
		given by $\cdot\sqtens *$.

		We need to show that $\del c_A=0$. Note that by definition $c_A$ only has a $1|1|0$--term. On the other hand
		$\delta$ on $\II\sqtens A^\vee\sqtens\II\sqtens A$ has only $1|1|0$-- and $1|1|1$--terms, while
		$\delta$ on $\II$ has only a $1|1|1$--term.

		Thus only four terms from the definition of $\del c_A$ survive. These are shown in \Figure~\ref{fig:av_a_dh}.	
		Expanding the definition of $\delta$ on 
		$\II\sqtens A^\vee\sqtens\II\sqtens A$ in terms of the operations of $\II$, $A$, and $A^\vee$, we get seven terms.
		We can see them in \Figure~\ref{fig:av_a_dh_details}. The terms in
		\Figures~\ref{subfig:av_a_dh_details_1}---\ref{subfig:av_a_dh_details_4} correspond to
		\Figure~\ref{subfig:av_a_dh_1}, while those in \Figures~\ref{subfig:av_a_dh_details_5}---\ref{subfig:av_a_dh_details_7}
		correspond to \Figures~\ref{subfig:av_a_dh_2}---\ref{subfig:av_a_dh_4}, respectively.
		Six of the terms cancel in pairs, while the one in \ref{subfig:av_a_dh_details_2} equals $0$.

		\begin{figure}
			\labellist
			\pinlabel $=$ at 146 92
			\small
			\pinlabel $c_A$ at 66 92
			\pinlabel $1$ at 206 120
			\hair=1.5pt
			\pinlabel $\lu{A}\II^A$ [b] at 34 164
			\pinlabel $A$ [b] at 56 164
			\pinlabel $\lu{A}\II^A$ [b] at 78 164
			\pinlabel $A$ [b] at 100 164
			\pinlabel $\lu{A}\II^A$ [b] at 182 164
			\pinlabel $A$ [b] at 206 164
			\pinlabel $\lu{A}\II^A$ [b] at 230 164
			\pinlabel $A$ [b] at 256 164
			\pinlabel $A$ [t] at 8 20
			\pinlabel $\lu{A}\II_A$ [t] at 66 20
			\pinlabel $A$ [t] at 152 20
			\pinlabel $\lu{A}\II_A$ [t] at 206 20
		\endlabellist
		\includegraphics[scale=.6]{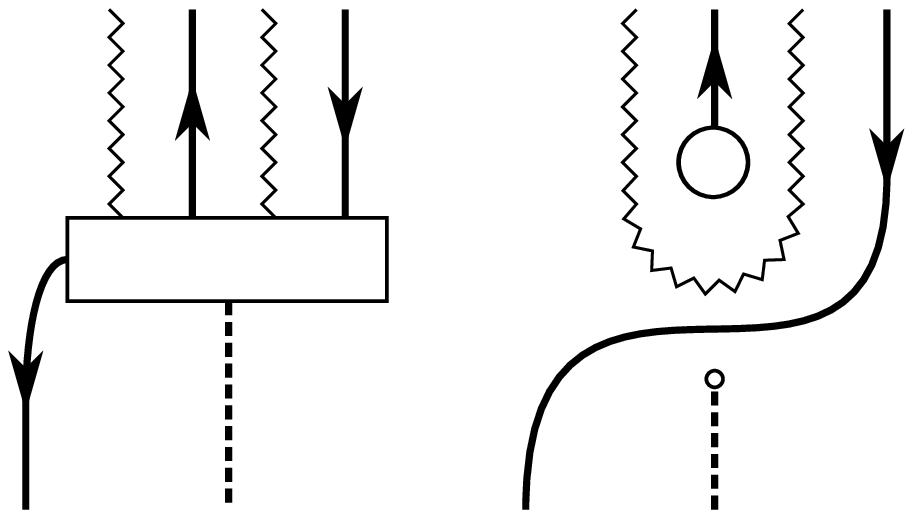}
		\caption{The cancellation homotopy equivalence $c_A:\II\sqtens A^\vee\sqtens\II\sqtens A\to\II$.}
		\label{fig:av_a_h_def}
	\end{figure}
	\begin{figure}
		\begin{subfigure}[b]{.24\linewidth}
			\centering
			\labellist
			\small
			\pinlabel $\mu_2$ at 16 40
			\pinlabel $\delta_{1|1|0}$ at 78 120
			\pinlabel $c_A$ at 78 68
		\endlabellist
		\includegraphics[scale=.5]{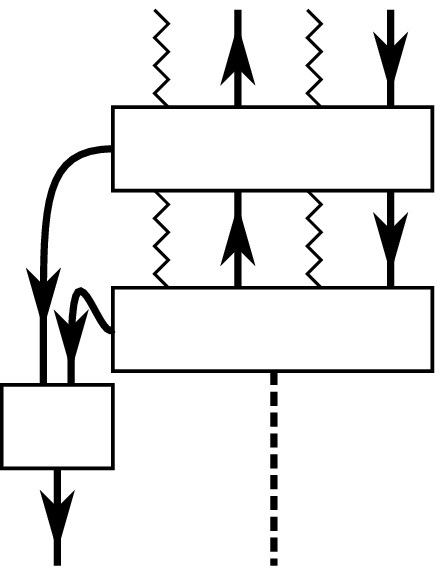}
		\caption{}
		\label{subfig:av_a_dh_1}
	\end{subfigure}
	\begin{subfigure}[b]{.24\linewidth}
		\centering
		\labellist
		\small
		\pinlabel $\mu_1$ at 16 40
		\pinlabel $c_A$ at 78 80
	\endlabellist
	\includegraphics[scale=.5]{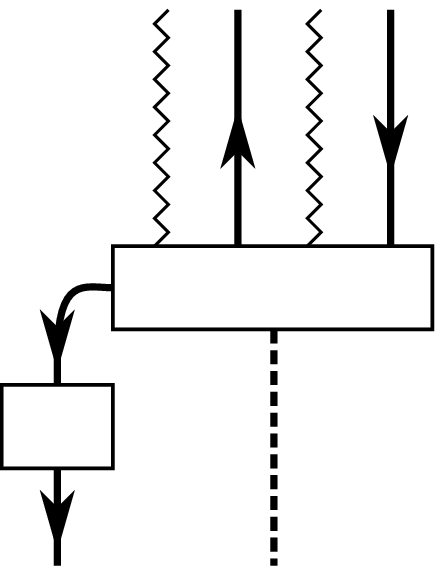}
	\caption{}
	\label{subfig:av_a_dh_2}
\end{subfigure}
\begin{subfigure}[b]{.24\linewidth}
	\centering
	\labellist
	\small
	\pinlabel $\mu_2$ at 16 40
	\pinlabel $c_A$ at 78 120
	\pinlabel $\delta_\II$ at 78 68
\endlabellist
\includegraphics[scale=.5]{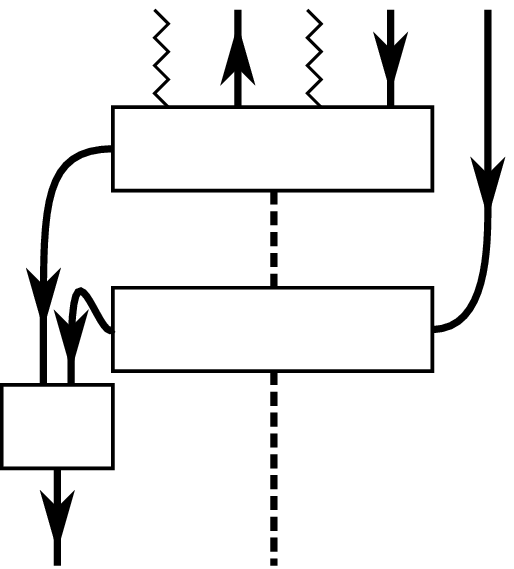}
\caption{}
\label{subfig:av_a_dh_3}
		\end{subfigure}
		\begin{subfigure}[b]{.24\linewidth}
			\centering
			\labellist
			\small
			\pinlabel $\mu_2$ at 16 40
			\pinlabel $\delta_{1|1|1}$ at 78 120
			\pinlabel $c_A$ at 78 68
		\endlabellist
		\includegraphics[scale=.5]{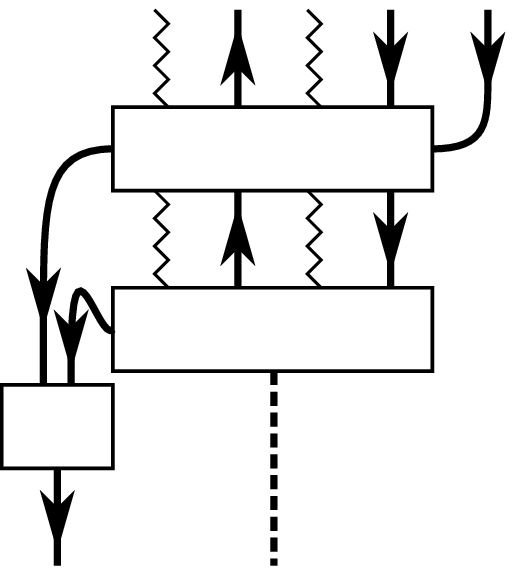}
		\caption{}
		\label{subfig:av_a_dh_4}
	\end{subfigure}
	\caption{Nontrivial terms of $\del c_A$.}
	\label{fig:av_a_dh}
\end{figure}
\begin{figure}
	\begin{subfigure}[b]{.24\linewidth}
		\centering
		\labellist
		\small
		\pinlabel $\mu_2$ at 14 36
		\pinlabel $\delta_\II$ at 42 72
		\pinlabel \rotatebox[origin=c]{180}{$\mu_2$} at 68 108
		\pinlabel $1$ at 69 60
	\endlabellist
	\includegraphics[scale=.6]{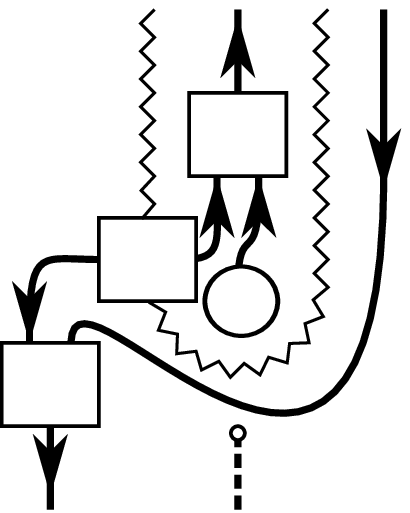}
	\caption{}
	\label{subfig:av_a_dh_details_1}
\end{subfigure}
\begin{subfigure}[b]{.24\linewidth}
	\centering
	\labellist
	\small
	\pinlabel $\mu_2$ at 16 44
	\pinlabel \rotatebox[origin=c]{180}{$\mu_1$} at 70 108
	\pinlabel $1$ at 10 112																									
	\pinlabel $1$ at 70 66
\endlabellist
\includegraphics[scale=.6]{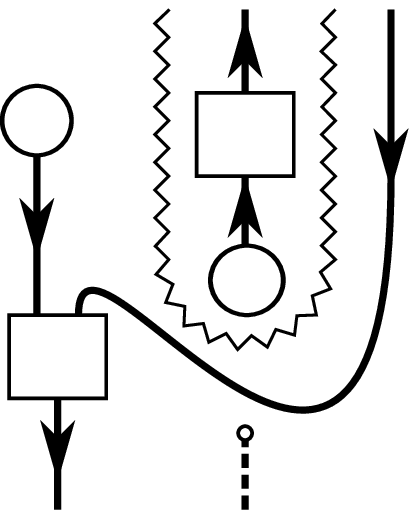}
\caption{}
\label{subfig:av_a_dh_details_2}
		\end{subfigure}
		\begin{subfigure}[b]{.24\linewidth}
			\centering
			\labellist
			\small
			\pinlabel $\mu_2$ at 16 32
			\pinlabel $\mu_2$ at 122 40
			\pinlabel $\delta_\II$ at 94 72
			\pinlabel \rotatebox[origin=c]{180}{$\mu_2$} at 68 108
			\pinlabel $1$ at 10 100
			\pinlabel $1$ at 68 60
		\endlabellist
		\includegraphics[scale=.6]{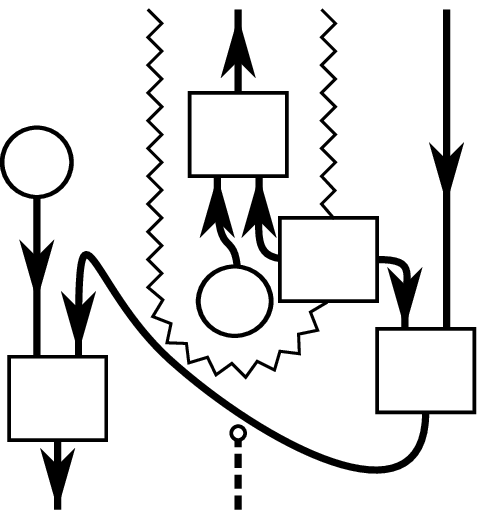}
		\caption{}
		\label{subfig:av_a_dh_details_3}
	\end{subfigure}
	\begin{subfigure}[b]{.24\linewidth}
		\centering
		\labellist
		\small
		\pinlabel $\mu_2$ at 16 44
		\pinlabel $\mu_1$ at 118 100
		\pinlabel $1$ at 10 112																									
		\pinlabel $1$ at 70 88
	\endlabellist
	\includegraphics[scale=.6]{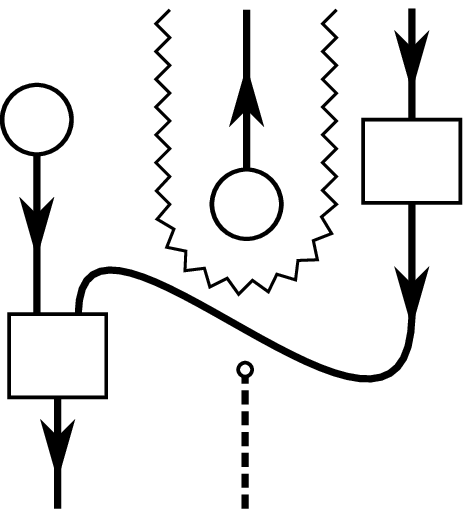}
	\caption{}
	\label{subfig:av_a_dh_details_4}
\end{subfigure}
\begin{subfigure}[b]{.32\linewidth}
	\centering
	\labellist
	\small
	\pinlabel $\mu_1$ at 16 44
	\pinlabel $1$ at 70 88
\endlabellist
\includegraphics[scale=.6]{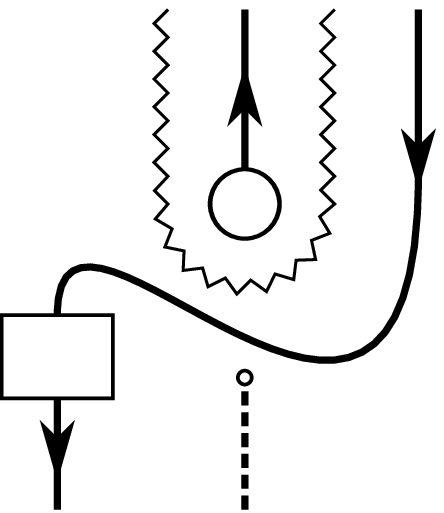}
\caption{}
\label{subfig:av_a_dh_details_5}
		\end{subfigure}
		\begin{subfigure}[b]{.32\linewidth}
			\centering
			\labellist
			\small
			\pinlabel $\mu_2$ at 16 44
			\pinlabel $1$ at 70 104
		\endlabellist
		\includegraphics[scale=.6]{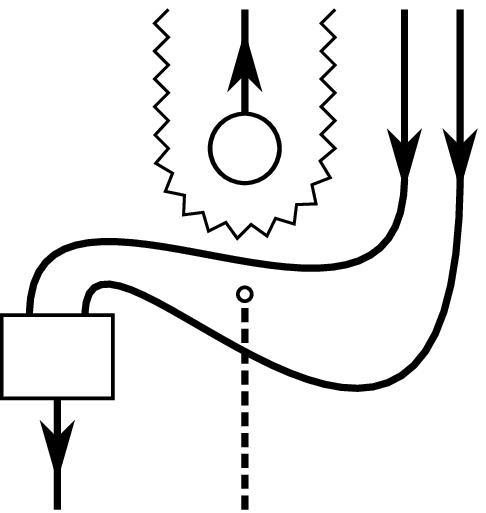}
		\caption{}
		\label{subfig:av_a_dh_details_6}
	\end{subfigure}
	\begin{subfigure}[b]{.32\linewidth}
		\centering
		\labellist
		\small
		\pinlabel $\mu_2$ at 106 100
		\pinlabel $1$ at 54 88
	\endlabellist
	\includegraphics[scale=.6]{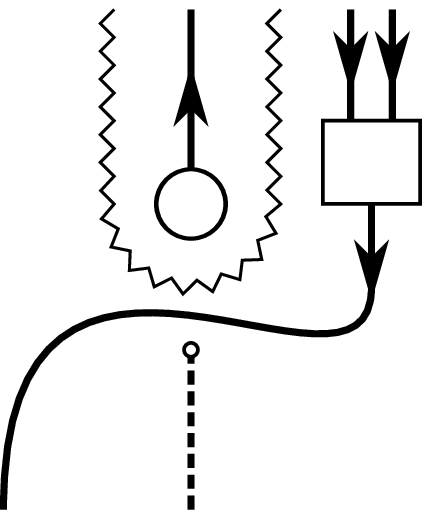}
	\caption{}
	\label{subfig:av_a_dh_details_7}
\end{subfigure}
\caption{Elementary terms of $\del h$.}
\label{fig:av_a_dh_details}
	\end{figure}

	Showing that $c_A$ is a homotopy equivalence is somewhat roundabout. 
	First we will show that the induced map
	\begin{equation*}
		\id_A\sqtens c_A\co A\sqtens(\lu{A}\II^A\sqtens A^\vee\sqtens\lu{A}\II^A\sqtens A)
		\to A\sqtens\lu{A}\II_A\cong A
	\end{equation*}
	is a homotopy equivalence. It is easy to see that the map is
	\begin{equation*}
		(\id_A\sqtens c_A)_{0|1|0} (a\sqtens *\sqtens b^\vee\sqtens *\sqtens c)=
		\begin{cases}
			a \cdot c &\textrm{if $b$ is an idempotent,}\\
			0 &\textrm{otherwise.}
		\end{cases}
	\end{equation*}

	In particular, it is surjective. Indeed, $\id_A\sqtens c_A(a\sqtens *\sqtens 1^\vee\sqtens *\sqtens 1)=a$ for all
	$a\in A$. Thus the induced map on homology is surjective. But the source and domain are homotopy equivalent for
	topological reasons (both represent $\BSAA(\TW_{\ol\F,-})$). This implies that $\id_A\sqtens c_A$ is a quasi-isomorphism,
	and a homotopy equivalence. But $(\II\sqtens A^\vee \sqtens\II)\sqtens A\simeq\II$ and
	$A\sqtens(\II\sqtens A^\vee \sqtens\II)\simeq\II$ for topological reasons, so $A\sqtens\cdot$ is an equivalence of derived categories.
	Thus, $c_A$ itself must have been a homotopy equivalence, which finishes the proof of the lemma.
\end{proof}

	We will now use  \Lemma~\ref{lem:av_a_homotopy_equivalence},
	to show that for any $\Y$ there is a homotopy $\Psi_\W(\cdot,\Delta_M)\simeq\id_{\SFC(\Y)}$. 
	Let $c_A$ be the homotopy equivalence from the lemma.
	There is a sequence of homomorphisms as follows.

	\centerline{ \xymatrix{
		\II\sqtens M
		\ar[d]^{\id_{\II\sqtens M}\otimes\Delta_M}\\
		\II\sqtens M\otimes\D(M)\cong\II\sqtens M\otimes M^\vee\sqtens\II\sqtens A\sqtens\II\sqtens M
		\ar[d]^{\id_{\II}\sqtens\nabla_M\sqtens\id_{\II\sqtens A\sqtens\II\sqtens M}}\\
		\II\sqtens A^\vee \sqtens\II\sqtens A \sqtens\II\sqtens M
		\ar[d]^{c_A\sqtens\id_{\II\sqtens M}} \\
		\II\sqtens M 
		}}
	
	The composition of these maps is shown in \Figure~\ref{fig:nabla_delta}. As we can see from the diagram, it is equal to
	$\id_\II\sqtens\id_M$. If $U=\BSD(\Y\setminus\W)$, then by applying the functor $\id_U\sqtens\cdot$ to both
	homomorphisms, we see that
	\begin{equation*}
		(\id_U\sqtens c_A\sqtens\id_{\II\sqtens M})\circ\Psi_M\circ(\id_{\SFC(\Y)}\sqtens\Delta_M)=\id_{\SFC(\Y)},
	\end{equation*}
	which is equivalent to \Equation~(\ref{eq:join_identity}).
\end{proof}
\begin{figure}
	\labellist
	\small
	\pinlabel $=$ at 250 150
	\pinlabel $=$ at 512 150
	\pinlabel $=$ at 746 150
	\tiny
	\pinlabel $1$ at 352 100
	\pinlabel $1$ at 426 232
	\pinlabel $1$ at 524 104
	\pinlabel $1$ at 600 100
	\pinlabel $1$ at 674 232
	\pinlabel $1$ at 772 144
	\pinlabel $\mu_A$ at 24 52
	\pinlabel $\mu_A$ at 280 52
	\pinlabel $\mu_A$ at 528 52
	\pinlabel $\ol\delta_\II$ at 46 144
	\pinlabel $\ol\delta_\II$ at 302 144
	\pinlabel $\ol\delta_{\II\sqtens{A}\sqtens\II\sqtens{M}}$ at 182 144
	\pinlabel $\ol\delta_{\II\sqtens{A}\sqtens\II\sqtens{M}}$ at 438 144
	\pinlabel $\ol\delta_{\II\sqtens{M}}$ at 204 88
	\pinlabel \rotatebox[origin=c]{90}{$m_M$} at 352 192
	\pinlabel \rotatebox[origin=c]{90}{$m_M$} at 600 192
	\pinlabel $\pi_A$ at 16 104
	\pinlabel $\pi_A$ at 272 104
	\pinlabel $\Delta_M$ at 170 228
	\pinlabel $\nabla_M$ at 98 144
	\pinlabel $c_A$ at 108 88
	\hair=1.5pt
	\pinlabel $\II$ [b] at 46 284
	\pinlabel $M$ [b] at 68 284
	\pinlabel $\II$ [b] at 302 284
	\pinlabel $M$ [b] at 324 284
	\pinlabel $\II$ [b] at 550 284
	\pinlabel $M$ [b] at 572 284
	\pinlabel $\II$ [b] at 794 284
	\pinlabel $M$ [b] at 816 284
	\pinlabel $A$ [t] at 24 16
	\pinlabel $\II$ [t] at 194 16
	\pinlabel $M$ [t] at 216 16
	\pinlabel $A$ [t] at 280 16
	\pinlabel $\II$ [t] at 450 16
	\pinlabel $M$ [t] at 472 16
	\pinlabel $A$ [t] at 528 16
	\pinlabel $\II$ [t] at 698 16
	\pinlabel $M$ [t] at 720 16
	\pinlabel $A$ [t] at 772 16
	\pinlabel $\II$ [t] at 794 16
	\pinlabel $M$ [t] at 816 16
	\endlabellist
	\includegraphics[scale=.43]{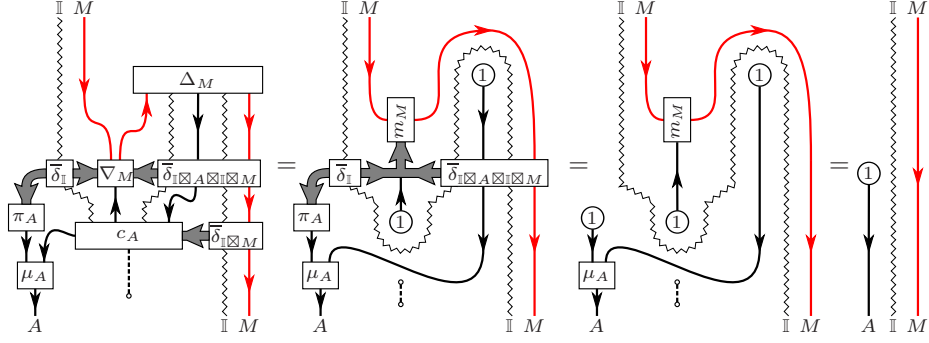}
	\caption{Proof that $\Psi_M(\cdot,\Delta_M)\simeq\id$.}
	\label{fig:nabla_delta}
\end{figure}

\subsection{Self-join and self-gluing}

So far we have been talking about the join or gluing of two disjoint sutured manifolds. However, we can extend these
notions to a self-join or self-gluing of a single manifold. For example if there is an embedding $(\W\sqcup-\W)\into\Y$,
then we can define the \emph{self-join} of $\Y$ along $\W$ to be the concatenation
\begin{equation*}
	\Y \Cup_{\W,\circlearrowleft}=(Y\setminus(\W\sqcup-\W))\cup_{\F\cup\ol\F}\TW_{\F,+}\cong
	\Y\Cup_{\W\sqcup-\W}\D(\W).
\end{equation*}
It is easy to see that if $\W$ and $-\W$ embed into different components of $\Y$, this is the same as the regular join.

Similarly, we can extend the join map to a \emph{self-join map}
\begin{equation*}
	\Psi_{\W,\circlearrowleft}\co\SFC(\Y)\to\SFC(\Y\Cup_{\W\sqcup-\W}\D(\W))\simeq\SFC(\Y\Cup_{\W,\circlearrowleft}),
\end{equation*}
by setting
\begin{equation*}
	\Psi_{\W,\circlearrowleft}=\Psi_{\W\sqcup-\W}(\cdot,\Delta_\W).
\end{equation*}

Again, if $\W$ and $-\W$ embed into disjoint components of $\Y$, $\Psi_{\W,\circlearrowleft}$ is, up to homotopy, the
same as the regular join map $\Psi_{\W}$. This follows quickly from properties~(\ref{subthm:join_associative})
and~(\ref{subthm:join_identity}) in \Theorem~\ref{thm:join_nice_properties}.

\section{The bordered invariants in terms of \texorpdfstring{\SFH}{SFH}}
\label{sec:reinterpretation}

In this section we give a (partial) reinterpretation of bordered and bordered sutured invariants in terms of $\SFH$ and
the gluing map $\Psi$. This is a more detailed version of \Theorem~\ref{thm:intro_bordered_via_SFH}.

\subsection{Elementary dividing sets}
\label{sec:elementary_div_sets}

Recall \Definition~\ref{def:dividing_set} of a dividing set. Suppose we have a sutured surface $\F=(F,\Lambda)$
parametrized by an arc diagram $\Z=(\ZZZ,\aaa,M)$ of rank $k$. We will define a set of $2^k$
distinguished dividing sets.

Before we do that, recall the way an arc diagram parametrizes a sutured surface, from \Section~\ref{sec:arc_diagrams}.
There is an embedding of the graph $G(\Z)$ into $F$, such that $\del\ZZZ=\Lambda$ (Recall
\Figure~\ref{fig:parametrized_surfaces}).
We will first define the elementary dividing sets in the cases that $\Z$ is of $\alpha$--type. In that case the image of
$\ZZZ$ is a push-off of $S_+$ into the interior of $F$. Denote the regions between $S_+$ and $\ZZZ$ by $R_0$. It is a
collection of discs, one for each component of $S_+$. The images of the arcs $e_i\subset G(\Z)$ are in the complement
$F\setminus R_0$.

\begin{defn}
	\label{def:elementary_dividing_set}
	Let $I\subset\{1,\ldots,k\}$.  The \emph{elementary dividing set for $\F$} associated to $I$ is the dividing set
	$\Gamma_I$ constructed as follows.
	Let $R_0$ be the region defined above. Set
	\begin{equation*}
		R_+ = R_0 \cup \bigcup_{i\in I}\nu(e_i) \subset F.
	\end{equation*}
	Then $\Gamma_I=(\del R_+)\setminus S_+$.
\end{defn}

If $\Z$ is of $\beta$--type, repeat the same procedure, substituting $R_-$ for $R_+$ and $S_-$ for $S_+$. For example
the region $R_0$ consists of discs bounded by $S_-\cup\ZZZ$.
Examples of both cases are given in \Figure~\ref{fig:elem_div_sets}.

\begin{figure}
	\begin{subfigure}[b]{.48\linewidth}
		\centering
		\labellist
		\small\hair=1.5pt
		\pinlabel $1$ [tl] at 26 30
		\pinlabel $2$ [bl] at 26 30
		\pinlabel $3$ [tl] at 26 78
		\pinlabel $4$ [bl] at 26 78
		\pinlabel $1$ [tr] at 92 20
		\pinlabel $2$ [bl] at 144 36
		\pinlabel $3$ [tl] at 144 72
		\pinlabel $4$ [br] at 92 88
		\endlabellist
		\includegraphics[scale=.59]{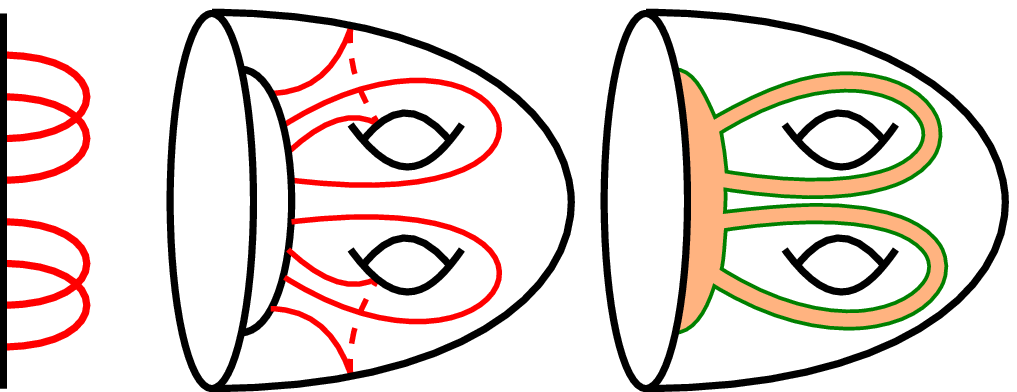}
		\caption{$\alpha$--type diagram.}
		\label{subfig:elem_div_set_alpha}
	\end{subfigure}
	\begin{subfigure}[b]{.48\linewidth}
		\centering
		\labellist
		\small\hair=1.5pt
		\pinlabel $1$ [tr] at 10 30
		\pinlabel $2$ [br] at 10 30
		\pinlabel $3$ [tr] at 10 78
		\pinlabel $4$ [br] at 10 78
		\pinlabel $1$ [tl] at 120 20
		\pinlabel $2$ [br] at 68 36
		\pinlabel $3$ [tr] at 68 72
		\pinlabel $4$ [bl] at 120 88
		\endlabellist
		\includegraphics[scale=.59]{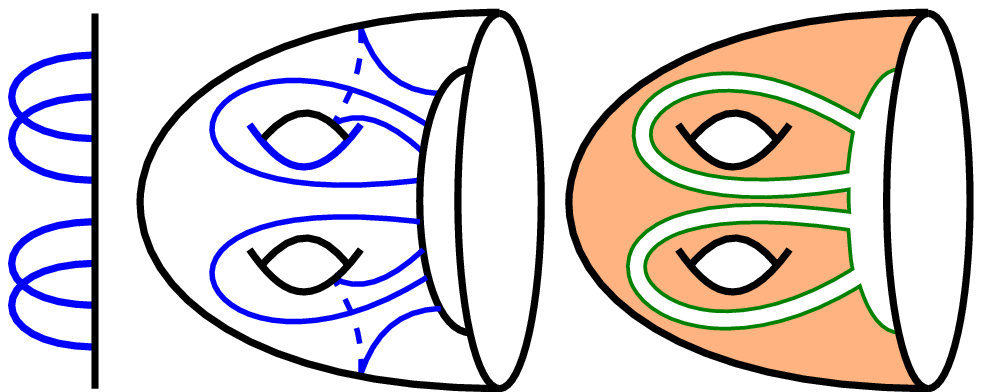}
		\caption{$\beta$--type diagram.}
		\label{subfig:elem_div_set_beta}
	\end{subfigure}
	\caption{Elementary dividing sets for an arc diagram. In each case we show the arc diagram, its embedding into
	the surface, and the dividing set $\Gamma_{\{2,3\}}$. The shaded regions are $R_+$.}
	\label{fig:elem_div_sets}
\end{figure}

We refer to the collection of $\Gamma_I$ for all $2^k$--many subsets of $\{1,\ldots,k\}$ as \emph{elementary dividing
sets for $\Z$}. The reason they are important is the following proposition.

\begin{prop}
	\label{prop:elem_div_sets}
	Let $\Z$ be an arc diagram of rank $k$, and let $I\subset\{1,\ldots,k\}$ be any subset. Let
	$\iota_I$ be the idempotent for $A=\A(\Z)$ corresponding to horizontal arcs at all $i\in I$, and let
	$\iota_{I^c}$ be the idempotent corresponding to the complement of $I$. Let $\W_I$ be the cap associated to
	the elementary dividing set $\Gamma_I$.

	Then the following hold:
	\begin{itemize}
		\item $\lu{A}\BSD(\W)$ is (represented by) the elementary type--$D$ module for $\iota_I$.
		\item $\li{_A}\BSA(\W)$ is (represented by) the elementary type--$A$ module for $\iota_{I^c}$.
	\end{itemize}
\end{prop}
\begin{proof}
	The key fact is that there is a particularly simple Heegaard diagram $\HH$ for $\W_I$. For concreteness we will
	assume $\Z$ is a type--$\alpha$ diagram, though the case of a type--$\beta$ diagram is completely analogous.

	The diagram $\HH=(\Sigma,\balpha,\bbeta,\Z)$ contains no $\alpha$--circles, exactly one $\alpha$--arc
	$\alpha^a_i$ for each matched pair $M^{-1}(i)$, and $k-\# I$ many $\beta$--circles. Each $\beta$--circle has
	exactly one intersection point on it, with one of $\alpha^a_i$, for $i\notin I$. This implies that there is
	exactly one generator $\xxx\in\G(\HH)$, that occupies the arcs for $I^c$.  This implies that $\BSD(\W_I)$ and
	$\BSA(\W_I)$ are both $\{\xxx,0\}$ as $\ZZ/2$--modules. The actions of the ground ring are $\iota_I\cdot\xxx=\xxx$
	for $\BSD(\W_I)$ and $\iota_{I^c}\cdot\xxx=\xxx$ for $\BSA(\W_I)$. This was one of the two requirements for an
	elementary module.
	
	The connected components of $\Sigma\setminus(\balpha\cup\bbeta)$ are in 1--to--1 correspondence with components of
	$\del R_+$. In fact each such region is adjacent to exactly one component of $\del\Sigma\setminus\ZZZ$. Therefore,
	there are only boundary regions and no holomorphic curves are counted for the definitions of $\BSD(\W_I)$ and
	$\BSA(\W_I)$. This was the other requirement for an elementary module, so the proof is complete. The diagram $\HH$
	can be seen in \Figure~\ref{fig:elementary_diagram}.
\end{proof}

\begin{figure}
	\labellist
	\small
	\pinlabel $A$ at 17 24
	\pinlabel \reflectbox{$A$} at 59 24
	\pinlabel $B$ at 136 24
	\pinlabel \reflectbox{$B$} at 177 24
	\endlabellist
	\includegraphics[scale=.8]{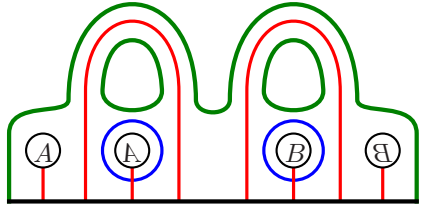}
	\caption{Heegaard diagram $\HH$ for the cap $\W_{2,3}$ corresponding to the dividing set from
	\Figure~\ref{subfig:elem_div_set_alpha}.}
	\label{fig:elementary_diagram}
\end{figure}

We will define one more type of object. Let $\F$ be a sutured surface parametrized by some arc diagram $\Z$. Let $I$ and
$J$ be two subsets of $\{1,\ldots,k\}$. Consider the sutured manifold $-\W_I\cup\TW_{-\ol\F,-}\cup\W_J$. Since $-\W_I$ and 
$\W_J$ are caps, topologically this is $F\times[0,1]$. The dividing set can be described as follows. Along
$F\times\{0\}$ it is $\Gamma_I\times\{0\}$, along $F\times\{1\}$ it is $\Gamma_J\times\{1\}$, and along
$\del F\times[0,1]$ it consists of arcs in the $[0,1]$ direction with a partial negative twist.

\begin{defn}
	\label{def:gamma_I_J}
	Let $\Gamma_{I\to J}$ denote the dividing set on $\del(F\times [0,1])$, such that
	\begin{equation*}
		(F\times[0,1],\Gamma_{I\to J})=-\W_I\cup\TW_{\ol\F,-}\cup\W_J.
	\end{equation*}
\end{defn}

\subsection{Main results}
\label{sec:main_result}

The main results of this section are the following two theorems. We will give the proofs in the next subsection.

\begin{thm}
	Let $\F$ be a sutured surface parametrized by an arc diagram $\Z$.
	\label{thm:algebra_as_SFH}
	The homology of $A=\A(\Z)$ decomposes as the sum
	\begin{equation}
		\label{eq:algebra_decomp}
		H_*(A)=\bigoplus_{I,J\subset\{1,\ldots,k\}}\iota_I\cdot H_*(A)\cdot \iota_J
		=\bigoplus_{I,J\subset\{1,\ldots,k\}}H_*(\iota_I\cdot A\cdot \iota_J),
	\end{equation}
	where
	\begin{equation}
		\label{eq:algebra_as_SFH}
		\iota_I\cdot H_*(A)\cdot\iota_J\cong\SFH(F\times[0,1],\Gamma_{I\to J}).
	\end{equation}

	Multiplication $\mu_2$ descends to homology as
	\begin{equation}
		\label{eq:algebra_as_SFH_multiplication}
		\begin{split}
			\mu_H=\Psi_{(F,\Gamma_{J})}\co\SFH\left(F\times[0,1],\Gamma_{I\to J}\right)\otimes
			\SFH\left(F\times[0,1],\Gamma_{J\to K}\right)\\
			\to\SFH\left(F\times[0,1],\Gamma_{I\to K}\right),
		\end{split}
	\end{equation}
	and is 0 on all other summands.
\end{thm}

\begin{thm}
	\label{thm:BSA_as_SFH}
	Let $\Y=(Y,\Gamma,\F)$ be a bordered sutured manifold where $\F$ parametrized by $\Z$.
	Then there is a decomposition
	\begin{equation}
		\label{eq:BSA_decomposition}
		\begin{split}
			H_*\left(\BSA(Y)_A\right)&=\bigoplus_{I\subset\{1,\ldots,k\}}H_*\left(\BSA(\Y)\right)\cdot\iota_I\\
			&=\bigoplus_{I\subset\{1,\ldots,k\}}H_*\left(\BSA(Y)\cdot\iota_I\right),
		\end{split}
	\end{equation}
	where
	\begin{equation}
		\label{eq:BSA_as_SFH}
		H_*\left(\BSA(Y)\right)\cdot\iota_I \cong \SFH(Y,\Gamma\cup\Gamma_I).
	\end{equation}
	
	Moreover, the $m_{1|1}$ action of $A$ on $\BSA$ descends to the following action on homology:
	\begin{equation}
		\label{eq:BSA_as_SFH_action}
		m_H=\Psi_{(F,\Gamma_I)}
		\co\SFH(Y,\Gamma\cup\Gamma_I)\otimes\SFH(F\times I,\Gamma_{I\to J})
		\to\SFH(Y,\Gamma\cup\Gamma_J),
	\end{equation}
	and $m_H=0$ on all other summands.

	Similar statements hold for left $A$--modules $\li{_A}\BSA(\Y)$, and for bimodules $\li{_A}\BSAA(\Y)_B$.
\end{thm}

\Theorem~\ref{thm:algebra_as_SFH} and~\ref{thm:BSA_as_SFH}, give us an alternative way to think about
bordered sutured Floer homology, or pure bordered Floer homology. (Recall that as shown in~\cite{Zar:BSFH}, the bordered
invariants $\CFD$ and $\CFA$ are special cases of $\BSD$ and $\BSA$.) More remarkably, as we show in~\cite{Zar:gluing2},
$H_*(A)$, $\mu_H$, and $m_H$ can be expressed in purely contact-geometric terms.

For practical purposes, $A$ and $\BSA$ can be replaced by the $\Ainf$--algebra $H_*(A)$ and the $\Ainf$--module
$H_*(\BSA)$ over it. For example, the pairing theorem will still hold. This is due to the fact that (using
$\ZZ/2$--coefficients), an $\Ainf$--algebra or module is always homotopy equivalent to its homology.

We would need, however, the higher multiplication maps of $H_*(A)$, and the higher actions of $H_*(A)$ on $H_*(\BSA)$.
The maps $\mu_H$ and $m_H$ that we just computed are only single terms of those higher operations. (Even though $A$ is a
DG-algebra, $H_*(A)$ usually has nontrivial higher multiplication.)

\subsection{Proofs}
\label{sec:SFH_equiv_proofs}

In this section we prove \Theorems~\ref{thm:algebra_as_SFH} and~\ref{thm:BSA_as_SFH}. Since there is a lot of overlap of
the two results and the arguments, we will actually give a combined proof of a mix of statements from both theorems.
The rest follow as corollary.

\begin{proof}[Combined proof of \Theorem~\ref{thm:algebra_as_SFH} and \Theorem~\ref{thm:BSA_as_SFH}]
	First, note that \Equation~(\ref{eq:algebra_decomp}) and \Equation~(\ref{eq:BSA_decomposition}) follow directly from the
	fact that the idempotents generate the ground ring over $\ZZ/2$.

	We will start by proving a generalization of \Equation~(\ref{eq:algebra_as_SFH}) and \Equation~(\ref{eq:BSA_as_SFH}).
	The statement is as follows. Let $\F$ and $\F'$ be two sutured surfaces parametrized by the arc diagrams
	$\Z$ and $\Z'$ of rank $k$ and $k'$, respectively. Let $A=\A(\Z)$ and $B=\A(\Z')$. Let
	$\Y=(Y,\Gamma,\ol{\F}\sqcup\F')$ be a
	bordered sutured manifold, and let $M=\li{_A}\BSAA(\Y)_B$.

	Fix $I\subset\{1,\ldots,k\}$ and $J\subset\{1,\ldots,k'\}$. Let $\W_I$ and $\W'_J$ be
	the respective caps associated to the dividing sets $\Gamma_I$ on $\F$ and $\Gamma'_J$ on $\F'$. Then the
	following homotopy equivalence holds.
	\begin{equation}
		\label{eq:BSAA_as_SFH}
		\iota_I\cdot\BSAA(\Y)\cdot\iota_J\simeq\SFC(Y,\Gamma_I\cup\Gamma\cup\Gamma'_J).
	\end{equation}
	
	The proof is easy. Notice that the sutured manifold $(Y,\Gamma_I\cup\Gamma\cup\Gamma'_J)$ is just
	$-\W_I\cup\Y\cup\W'_J$. By the pairing theorem,
	$\SFC(Y,\Gamma_I\cup\Gamma\cup\Gamma'_J)\simeq\BSD(-\W_I)\sqtens\BSAA(\Y)\sqtens\BSD(\W'_J)$.
	But by \Proposition~\ref{prop:elem_div_sets}, $\BSD(-\W_I)=\{x_I,0\}$ is the elementary module corresponding to
	$\iota_I$, while $\BSD(\W'_J)=\{y_J,0\}$ is the elementary idempotent corresponding to $\iota'_J$. Thus we have
	\begin{equation*}
		\begin{split}
			\BSD(-\W_I)\sqtens\BSAA(\Y)\sqtens\BSD(\W'_J)
			&= x_I\sqtens\BSAA(\Y)\sqtens y_J\\
			&\cong \iota_I\cdot\BSAA(\Y)\cdot\iota'_J.
		\end{split}
	\end{equation*}

	\Equation~(\ref{eq:algebra_as_SFH}) follows from \Equation~(\ref{eq:BSAA_as_SFH}) by substituting the empty sutured surface
	$\varnothing=(\varnothing,\varnothing)$ for $\F$. Its algebra is $\A(\varnothing)=\ZZ/2$, so
	$\li{_{\ZZ/2}}\BSAA(\Y)_B$ and $\BSA(\Y)_B$ can be identified.

	\Equation~(\ref{eq:BSA_as_SFH}) follows from \Equation~(\ref{eq:BSAA_as_SFH}) by substituting $\F(\Z)$ for both $\F$
	and $\F'$, and $\TW_{-\ol\F,-}$ for $\Y$. Indeed, $\BSAA(-\TW_{\ol\F,-})\simeq\A(\Z)$, as a bimodule over itself, by
	\Proposition~\ref{prop:twisting_slice_invariants}.

	Next we prove \Equation~(\ref{eq:BSA_as_SFH_action}). Let $U_A$ be a DG-type representative for $\BSA(\Y)_A$, and let
	$M_I$ be the elementary representative for $\li{_A}\BSA(\W_I)$. Since both are DG-type, we can form the
	associative product
	\begin{equation*}
		\begin{split}
			U\sqtens\lu{A}\II^A\sqtens M_I&\simeq\BSA(Y)\sqtens\BSD(\W_I)\\
			&\simeq\SFC(Y,\Gamma\cup\Gamma_I).
		\end{split}
	\end{equation*}
	Similarly, pick $M_J$ to be the elementary representative for $\li{_A}\BSA(\W_J)$. We also know that
	$\li{_A}A_A$ is a DG-type representative for $\li{_A}\BSAA(\TW_{-\ol\F,-})_A$. We have the associative
	product
	\begin{equation*}
		\begin{split}
			{M_I}^\vee\sqtens\lu{A}\II^A\sqtens A\sqtens\lu{A}\II^A\sqtens M_J
			&\simeq\BSD(-\W_I)\sqtens A\sqtens\BSD(\W_J)\\
			&\simeq\SFC(F\times[0,1],\Gamma_{I\to J}).
		\end{split}
	\end{equation*}
	
	Gluing the two sutured manifolds along $(F,\Gamma_I)$ results in
	\begin{equation*}
		\Y\cup\TW_{\F,+}\cup\TW_{-\ol\F,-}\cup\W_J\cong\Y\cup\W_J=(Y,\Gamma\cup\Gamma_J),
	\end{equation*}
	so we get the correct manifold.

	The gluing map can be written as the composition of
	\begin{align*}
		\Psi_{M_I}
		&\co (U\sqtens\II)\sqtens M_I \otimes {M_I}^\vee \sqtens(\II\sqtens A\sqtens\II\sqtens M_J)\\
		&\to (U\sqtens\II)\sqtens A^\vee\sqtens(\II\sqtens A\sqtens\II\sqtens M_J),\\
		\id_U\sqtens c_A\sqtens \id_{\II\sqtens M_J}
		&\co U\sqtens(\II\sqtens A^\vee\sqtens\II\sqtens A)\sqtens\II\sqtens M_J
		\to U\sqtens\II\sqtens M_J,
	\end{align*}
	where $c_A$ is the homotopy equivalence from \Lemma~\ref{lem:av_a_homotopy_equivalence}.

	Luckily, since $M_I$ is elementary, $\Psi_{M_I}$ takes the simple form from
	\Proposition~\ref{prop:join_for_elementary}. In addition, since $U$ and $M_J$ are DG-type, $\id\sqtens h\sqtens\id$
	is also very simple. As can be seen in \Figure~\ref{fig:BSA_SFH_action}, the composition is in fact
	\begin{equation*}
		u\sqtens *\sqtens x_{I^c}\otimes {x_{I^c}}^\vee\sqtens *\sqtens a \sqtens *\sqtens x_{J^c}\mapsto
		m_{1|1}(u,a)\sqtens *\sqtens x_{J^c}.
	\end{equation*}

	Since $\cdot\sqtens *\sqtens x_{I^c}$ corresponds to $\cdot\iota_I$, this translates to the map
	\begin{equation*}
		\begin{split}
			\Psi_{(F,\Gamma_I)}\co (U\cdot\iota_I) \otimes(\iota_I\cdot A\cdot\iota_J)&
			\to U\cdot\iota_J,\\
			(u\cdot\iota_I)\otimes(\iota_I\cdot a\cdot\iota_J).
			&\mapsto m(u,a)\cdot\iota_J
		\end{split}
	\end{equation*}

	\begin{figure}
		\labellist
		\pinlabel $=$ at 224 108
		\small
		\pinlabel $\id_{U\sqtens\II}\sqtens\nabla_{M_I}\sqtens\id_{\II\sqtens{A}\sqtens\II\sqtens{M_J}}$ at 102 140
		\pinlabel $\id_U\sqtens{c_A}\sqtens\id_{\II\sqtens{M_J}}$ at 102 76
		\pinlabel $m_U$ at 256 60
		\pinlabel $\iota_I^\vee$ at 298 146
		\pinlabel $1$ at 298 94
		\hair=1.5pt
		\pinlabel $U$ [b] at 16 196
		\pinlabel $\II$ [b] at 38 196
		\pinlabel $M_I$ [b] at 60 196
		\pinlabel $M_I$ [b] at 104 196
		\pinlabel $\II$ [b] at 126 196
		\pinlabel $A$ [b] at 148 196
		\pinlabel $\II$ [b] at 170 196
		\pinlabel $M_J$ [b] at 192 196
		\pinlabel $U$ [b] at 256 196
		\pinlabel $\II$ [b] at 274 196
		\pinlabel $M_I$ [b] at 288 196
		\pinlabel $M_I$ [b] at 308 196
		\pinlabel $\II$ [b] at 322 196
		\pinlabel $A$ [b] at 340 196
		\pinlabel $\II$ [b] at 358 196
		\pinlabel $M_J$ [b] at 376 196
		\pinlabel $U$ [t] at 16 20
		\pinlabel $\II$ [t] at 170 20
		\pinlabel $M_J$ [t] at 192 20
		\pinlabel $U$ [t] at 256 20
		\pinlabel $\II$ [t] at 358 20
		\pinlabel $M_J$ [t] at 376 20
		\endlabellist
		\includegraphics[scale=.7]{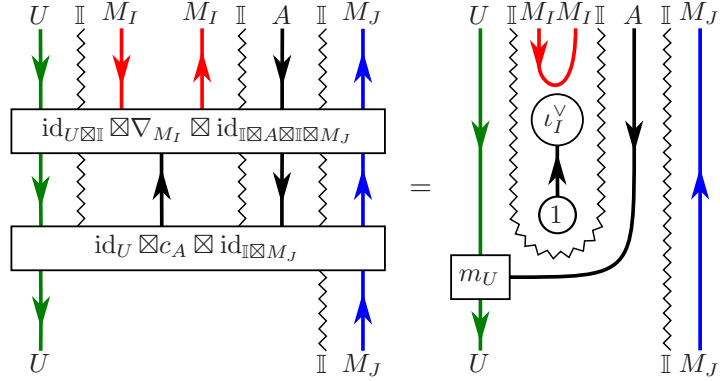}
		\caption{The gluing map $\Psi_{M_I}$ on $\SFC(Y,\Gamma_I)\otimes\SFC(F\times[0,1],\Gamma_{I\to J})$,
		followed by the chain homotopy equivalence $\id\sqtens c_A\sqtens\id$.}
		\label{fig:BSA_SFH_action}
	\end{figure}

	Note that even though we picked a specific representative for $\BSA(\Y)_A$, the group $H_*(\BSA(\Y))$ and the induced
	action $m_H$ of $H_*(A)$ do not depend on this choice. Finally, \Equation~(\ref{eq:algebra_as_SFH_multiplication})
	follows by treating $A$ as a right module over itself.
\end{proof}

\appendix
\section{Calculus of diagrams}
\label{sec:diagrams}

This appendix summarizes the principles of the diagrammatic calculus we have used throughout the
paper. First we describe the algebraic objects we work with, and the necessary assumptions on them. Then we describe the
diagrams representing these objects.

\subsection{Ground rings}
\label{sec:ground_rings}

The two basic objects we work with are a special class of rings, and bimodules over them. We call these rings
\emph{ground rings}.

\begin{defn}
	\label{def:ground_ring}
	A \emph{ground ring} $K$ is a finite dimensional $\ZZ/2$--algebra with a distinguished basis $(e_1,\ldots,e_k)$ such
	that multiplication is given by the formula
	\begin{equation*}
		e_i\cdot e_j =
		\begin{cases}
			e_i &\textrm{if $i=j$,}\\
			0 &\textrm{otherwise.}
		\end{cases}
	\end{equation*}

	Such a basis for $K$ is called a \emph{canonical basis}.
\end{defn}

The canonical basis elements are uniquely determined by the property that $e_i$ cannot be written as a sum $u+v$, where
$u$ and $v$ are nonzero and $u\cdot v =0$. Each element of $K$ is an idempotent, while $1_K=e_1+\cdots+e_k$ is an
identity element.

We consider only finite dimensional bimodules $\li{_K}M_L$ over ground rings $K$ and $L$, and collections
$(\li{_K}{M_i}_L)_{i\in I}$ where $I$ is a countable index set (usually $I=\{0,1,2,\ldots\}$, or some Cartesian power of
the same),
and each $M_i$ is a finite-dimensional $K,L$--bimodule. It is often useful to think of the collection $(M_i)$ as the
direct sum $\bigoplus_{i\in I}M_i$, but that sometimes leads to problems, so we will not make this identification.

There are some basic properties of bimodules over ground rings as defined above.

\begin{prop}
	\label{prop:ground_ring_modules}
	Suppose $K$, $L$, and $R$ are ground rings with canonical bases $(e_1,\ldots,e_k)$, $(e'_1,\ldots,e'_l)$, and 
	$(e''_1,\ldots,e''_r)$, respectively.
	\begin{itemize}
		\item A bimodule $\li{_K}M_L$ is uniquely determined by the collection of $\ZZ/2$--vector spaces
			\begin{align*}
				e_i\cdot M\cdot e'_j,& &i\in\{1,\ldots,k\},j\in\{1,\ldots,l\},
			\end{align*}
			which we will call the \emph{components} of $M$.
		\item A $K,L$--bilinear map $f\co M\to N$ is determined by the collection of $\ZZ/2$--linear maps
			\begin{equation*}
				f|_{e_i\cdot M\cdot e'_j}\co e_i\cdot M\cdot e'_j \to e_i\cdot N\cdot e'_j.
			\end{equation*}
		\item The tensor product $(\li{_K}M_L)\otimes_L(\li{_L}N_R)$ has components
			\begin{equation*}
				e_i\cdot(M\otimes_L N)\cdot e''_j=\bigoplus_{p=1}^{l}(e_i\cdot M\cdot e'_p)\otimes_{\ZZ/2}(e'_p\cdot
				N\cdot e''_j).
			\end{equation*}
		\item The dual $\li{_L}{M^\vee}_K$ of $\li{_K}M_L$ has components
			\begin{equation*}
				e_i\cdot{M^\vee}\cdot e'_j\cong (e'_j\cdot M\cdot e_i)^\vee,
			\end{equation*}
			and the double dual $(M^\vee)^\vee$ is canonically isomorphic to $M$.
	\end{itemize}
\end{prop}
\begin{proof}
	These follow immediately. The fact that $M^{\vee\vee}\cong M$ is due to the fact the fact that each component is a finite
	dimensional	vector space.
\end{proof}

Finally, when dealing with countable collections we introduce the following conventions. For consistency we can think of
a single
module $M$ as a collection $(M_i)$ indexed by the set $I=\{1\}$.

\begin{defn}
	\label{def:module_collections}
	Let $K$, $L$, and $M$ be as in \Proposition~\ref{prop:ground_ring_modules}.
	\begin{itemize}
		\item An \emph{element} of $(M_i)_{i\in I}$ is a collection $(m_i)_{i\in I}$ where $m_i\in M_i$.
		\item A \emph{bilinear map} $f\co(\li{_K}{M_i}_L)_{i\in I}\to(\li{_K}{N_j}_L)_{j\in J}$ is a collection
			\begin{align*}
				f_{(i,j)}&\co M_i\to N_j &(i,j)\in I\times J.
			\end{align*}
			Equivalently, a map $f$ is an element of the collection
			\begin{equation*}
				\Hom_{K,L}((M_i)_{i\in I},(N_j)_{j\in J})=(\Hom(M_i,N_j))_{(i,j)\in I\times J}.
			\end{equation*}
		\item The \emph{tensor} $\li{_K}(M_i)_L\otimes\li{_L}(N_j)_R$ is the collection
			\begin{equation*}
			((M\otimes N)_{(i,j)})_{(i,j)\in I\times J}=(M_i\otimes N_j)_{(i,j)\in I\times J}.
			\end{equation*}
		\item The \emph{dual} $((M_i)_{i\in I})^\vee$ is the collection $({M_i}^\vee)_{i\in I}$.
		\item  Given bilinear maps $f\co (M_i)\to (N_j)$ and $g\co (N_j)\to (P_p)$, their composition
			$g\circ f\co (M_i)\to (P_p)$ is the collection
			\begin{equation*}
				(g\circ f)_{(i,p)}=\sum_{j\in J}g_{(j,p)}\circ f_{(i,j)}.
			\end{equation*}
	\end{itemize}
\end{defn}

Note that the composition of maps on collections may not always be defined due to a potentially infinite sum.
On the other hand, the double dual $(M_i)^{\vee\vee}$ is still canonically isomorphic to $(M_i)$.

\subsection{Diagrams for maps}
\label{sec:map_diagrams}

We will use the following convention for our diagram calculus. There is a TQFT-like structure, where to decorated planar
graphs we assign bimodule maps.

\begin{prop}
	\label{prop:hom_space_equivalence}
	Suppose $K_0,K_1,\ldots,K_n=K_0$ are ground rings, $n\geq 0$, and $\li{_{K_{i-1}}}{M_i}_{K_i}$ are bimodules, or
	collections of bimodules. Then the following $\ZZ/2$--spaces are canonically isomorphic.
	\begin{align*}
		A_i &=M_i\otimes M_{i+1}\otimes\cdots\otimes M_n\otimes M_1\otimes\cdots \otimes M_{i-1}/\sim,\\
		B_{i,j}
		&=\Hom_{K_i,K_j}(M_i^\vee\otimes\cdots\otimes M_1^\vee\otimes M_n^\vee\otimes\cdots\otimes M_{j+1}^\vee,~M_{i+1}\otimes\cdots\otimes
		M_j),\\
		C_{i,j}&=\Hom_{K_j,K_i}(M_j^\vee\otimes\cdots\otimes M_{i+1}^\vee,~M_{j+1}\otimes\cdots\otimes M_n\otimes
		M_1\otimes\cdots\otimes M_i),
	\end{align*}
	for $0\leq i\leq j\leq n$, where the relation $\sim$ in the definition of $A_i$ is $k\cdot x\sim x\cdot k$, for
	$k\in K_{i-1}$.
\end{prop}
\begin{proof}
	The proof is straightforward. If all $M_i$ are single modules, then we are only dealing with finite-dimensional
	$\ZZ/2$--vector spaces. If some of them are collections, then the index sets for $A_i$, $B_{i,j}$ and $C_{i,j}$ are
	all the same, and any individual component still consists of finite dimensional vector spaces.
\end{proof}

This property is usually referred to as \emph{Frobenius duality}. Our bimodules behave similar to a pivotal tensor
category. Of course we do not have a real category, as even compositions are not always defined.

\begin{defn}
	\label{def:composite_diagram}
	A \emph{diagram} is a planar oriented graph, embedded in a disc, with some degree--$1$ vertices on the boundary of the disc
	There are labels as follows.
	\begin{itemize}
		\item Each planar region (and thus each arc of the boundary) is labeled by a ground ring $K$. 
		\item Each edge is labeled by a bimodule $\li{_K}M_L$, such that when traversing the edge in its direction, the
			region on the left is labeled by $K$, while the one on the right is labeled by $L$.
		\item An internal vertex with all outgoing edges labeled by $M_1,\ldots,M_n$, in cyclic counterclockwise order, is labeled
			by an element of one of the isomorphic spaces in \Proposition~\ref{prop:hom_space_equivalence}.
		\item If any of the edges adjacent to a vertex are incoming, we replace the corresponding modules by their duals.
	\end{itemize}
\end{defn}

When drawing diagrams we will omit the bounding disc, and the boundary vertices.
We will usually interpret diagrams consisting of a single internal vertex having several incoming edges
$M_1,\ldots,M_m$ ``on top'', and several outgoing edges $N_1,\ldots,N_n$ ``on the bottom'', as a bilinear map in
$\Hom(M_1\otimes\cdots\otimes{}M_m,~N_1\otimes\cdots\otimes{}N_n)$. See \Figure~\ref{fig:calculus_example} for an example.

\begin{figure}
	\centering
	\labellist
	\pinlabel $\leftrightarrow$ at 158 62
	\pinlabel $\leftrightarrow$ at 274 62
	\pinlabel $F$ at 62 66
	\pinlabel $F$ at 206 62
	\pinlabel \rotatebox[origin=c]{180}{$F$} at 338 62
	\small\hair 0.5pt
	\pinlabel $M_1$ [t] at 12 16
	\pinlabel $M_2$ [t] at 44 16
	\pinlabel $M_3$ [t] at 76 16
	\pinlabel $M_4$ [t] at 108 16
	\pinlabel $M_5$ [t] at 140 16
	\pinlabel $M_2$ [t] at 188 16
	\pinlabel $M_3$ [t] at 220 16
	\pinlabel $M_4$ [t] at 252 16
	\pinlabel $M_5$ [t] at 308 16
	\pinlabel $M_1$ [t] at 340 16
	\pinlabel $M_1^\vee$ [b] at 204 116
	\pinlabel $M_5^\vee$ [b] at 236 116
	\pinlabel $M_4^\vee$ [b] at 292 116
	\pinlabel $M_3^\vee$ [b] at 324 116
	\pinlabel $M_2^\vee$ [b] at 356 116
	\endlabellist
	\includegraphics[scale=.5]{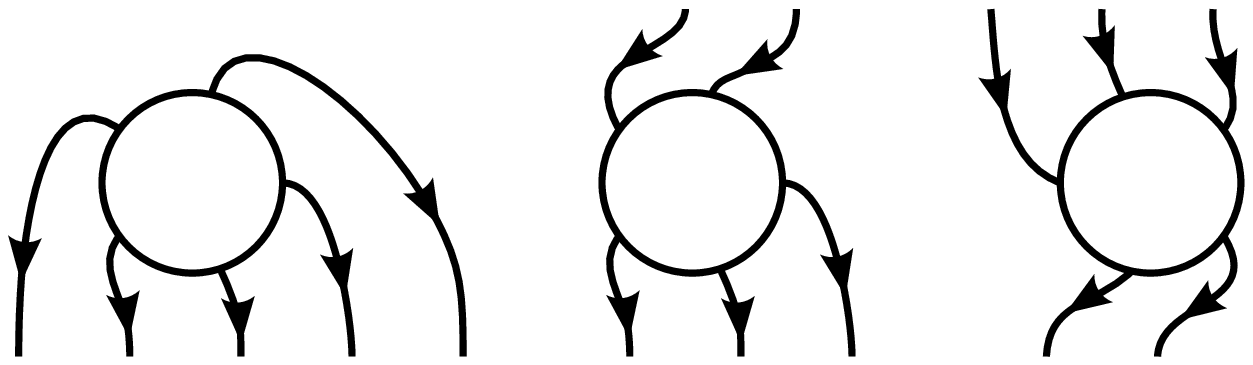}
	\caption{Three equivalent diagrams with a single vertex. The label $F$ is interpreted as an element of
	$A_1=M_1\otimes\cdots\otimes M_5/\sim$, $B_{1,4}=\Hom(M_1^\vee\otimes M_5^\vee,M_2\otimes M_3\otimes M_4)$, and
	$C_{1,4}\Hom(M_4^\vee\otimes M_3^\vee\otimes M_2^\vee,M_5\otimes M_1)$, respectively.}
	\label{fig:calculus_example}
\end{figure}

Under some extra assumptions, discussed in \Section~\ref{sec:boundedness}, a diagram with more vertices can also be
evaluated, or interpreted as an element of some set, corresponding to all outgoing edges. The most common example is
having two diagrams $\D_1$ and $\D_2$ representing linear maps

\centerline{\xymatrix{
M \ar[r]^{f_1} & N \ar[r]^{f_2} & P.
}}

Stacking the two diagrams together, feeding the outgoing edges of $\D_1$ into the incoming edges of $\D_2$, we get a new
diagram $\D$, corresponding to the map $f_2\circ f_1\co M\to P$. More generally, we can ``contract'' along all internal
edges, pairing the elements assigned to the two ends of an edge. As an example we will compute the diagram $D$ in
\Figure~\ref{fig:calculus_complicated_example}. Suppose the values of the vertices $F$, $G$,
and $H$ are as follows:
\begin{align*}
	F&=\sum_i m_i\otimes q_i\otimes s_i\in M\otimes Q \otimes S,\\
	G&=\sum_j s'_j\otimes r'_j\otimes p'_j\in S^\vee\otimes R^\vee\otimes P^\vee,\\
	H&=\sum_k q'_k\otimes n_k \otimes r_k \in Q^\vee\otimes N \otimes R.
\end{align*}
Then the value of $D$ is given by
\begin{equation*}
	D=\sum_{i,j,k} \left<q_i,q'_k\right>_Q\cdot \left<s_i,s'_j\right>_S\cdot \left<r_k,r'_j\right>_R
	\cdot m_i\otimes n_k \otimes p'_j\in M\otimes N\otimes P^\vee.
\end{equation*}

\begin{figure}
	\labellist
	\pinlabel $=$ at 112 66
	\pinlabel $D$ at 36 92
	\pinlabel $F$ at 176 116
	\pinlabel $G$ at 296 116
	\pinlabel $H$ at 236 60
	\small\hair=0.5pt
	\pinlabel $M$ [t] at 12 16
	\pinlabel $N$ [t] at 36 16
	\pinlabel $P$ [t] at 60 16
	\pinlabel $M$ [t] at 160 16
	\pinlabel $N$ [t] at 236 16
	\pinlabel $P$ [t] at 312 16
	\pinlabel $Q$ [tr] at 188 42
	\pinlabel $R$ [tl] at 284 40
	\pinlabel $S$ [b] at 236 98
	\endlabellist
	\includegraphics[scale=.5]{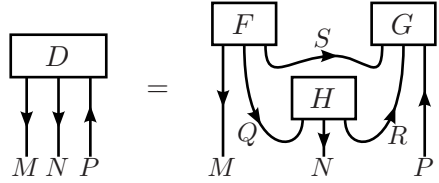}
	\caption{Evaluation of a complex diagram.}
	\label{fig:calculus_complicated_example}
\end{figure}

Edges that go from boundary to boundary and closed loops
can be interpreted as having an identity vertex in the middle. As with individual vertices, we can rotate a diagram to
interpret it as an element of different spaces, or different linear maps.

Note that the above construction might fail if any of the internal edges corresponds to a collection, since there might be an
infinite sum involved. The next section discusses how to deal with this problem.

\subsection{Boundedness}
\label{sec:boundedness}

When using collections of modules we have to make additional assumptions to avoid infinite sums. We use the concept of
boundedness of maps and diagrams.

\begin{defn}
	\label{def:boundedness}
	An element $(m_i)_{i\in I}$ of the collection $(M_i)_{i\in I}$ is called \emph{bounded} if only finitely many of its components $m_i$
	are nonzero. Equivalently, the bounded elements of $(M_i)$ can be identified with the elements of $\bigoplus_i M_i$.

	For a collection $(M_{i,j})_{i\in I,j\in J}$ there are several different concepts of boundedness. An element
	$(m_{i,j})$ is \emph{totally bounded} if it is bounded in the above sense, considering $I\times J$ as a single
	index-set. A weaker condition is that $(m_{i,j})$ is \emph{bounded in $J$ relative to $I$}. This means that for each $i\in
	I$, there are only finitely many $j\in J$, such that $m_{i,j}$ is nonzero. Similarly, an element can be bounded in
	$I$ relative to $J$.
\end{defn}

Note that $f\co (M_i)\to (N_j)$ is bounded in $J$ relative to $I$ exactly when $f$ represents a map from $\bigoplus_i
M_i$ to $\bigoplus_j N_j$. In computations relatively bounded maps are more common than totally bounded ones.
For instance the identity map $\id\co(M_i)\to(M_i)$ and the natural pairing $\left<\cdot,\cdot\right>\co(M_i)^\vee\otimes(M_i)\to
K$ are not totally bounded, but are bounded in each index relative to the other.
 
To be able to collapse an edge labeled by a collection $(M_i)_{i\in I}$ in a diagram, at least one of the two
adjacent vertices needs to be labeled by an element relatively bounded in the $I$--index.
For a given diagram $D$ we can ensure that it has a well-defined evaluation by imposing enough
 boundedness conditions on individual vertices. (There is usually no unique minimal set of conditions.)
Total or relative boundedness of $D$ can also be achieved by a stronger set of conditions. For example, if all vertices
are totally bounded, the entire diagram is also totally bounded.

\section{\texorpdfstring{$\Ainf$}{A-infinity}--algebras and modules}
\label{sec:algebra}

In this section we will present some of the background on $\Ainf$--algebras and modules, and the way they are used in
the bordered setting. A more thorough treatment is given in~\cite{LOT:bimodules}. 

As in \Appendix~\ref{sec:diagrams}, we always work with $\ZZ/2$--coefficients which avoids dealing with signs.
Everything is expressed in terms of the diagram calculus of \Appendix~\ref{sec:diagrams}. As described there, all
modules are finite dimensional, although we also deal with countable collections of such modules. There is essentially
only one example of collections that we use, which is presented below.

\subsection{The bar construction}
\label{sec:bar_definition}

Suppose $K$ is a ground ring and $\li{_K}M_K$ is a bimodule over it.

\begin{defn}
	\label{def:bar}
	The \emph{bar of $M$} is the collection
	\begin{equation*}
		\BBar M = (M^{\otimes i})_{i=0,\ldots,\infty},
	\end{equation*}
	of tensor powers of $M$.
\end{defn}

There are two important maps on the bar of $M$.

\begin{defn}
	\label{def:split_and_merge}
	The \emph{split} on $\BBar M$ is the map $s\co\BBar M\to\BBar M\otimes\BBar M$ with components
	\begin{equation*}
		s_{(i,j,k)}=
		\begin{cases}
			\id\co M^{\otimes i}\to(M^{\otimes j})\otimes(M^{\otimes k}) &\textrm{if $i=j+k$,}\\
			0 &\textrm{otherwise.}
		\end{cases}
	\end{equation*}

	The \emph{merge} map $\BBar M\otimes\BBar M\to \BBar M$ is similarly defined.
\end{defn}

Merges and splits can be extended to more complicated situations where any combination of copies of $\BBar M$ and $M$
merge into $\BBar M$, or split from $\BBar M$. All merges are associative, and all splits are coassociative.

Like the identity map, splits and merges are bounded in incoming indices, relative to outgoing, and vice versa. To
simplify diagrams, we draw merges and splits as merges ans splits of arrows, respectively, without using a box for the
corresponding vertex (see \Figure~\ref{fig:split_and_merge}).

\begin{figure}
	\begin{subfigure}[b]{.48\linewidth}
		\centering
		\labellist
		\tiny\hair 2pt
		\pinlabel $\BBar{}M$ [b] at 52 112
		\pinlabel $\BBar{}M$ [b] at 152 112
		\pinlabel $\BBar{}M$ [b] at 296 112
		\pinlabel $\BBar{}M$ [t] at 24 20
		\pinlabel $\BBar{}M$ [t] at 80 20
		\pinlabel $M$ [t] at 124 20
		\pinlabel $\BBar{}M$ [t] at 184 20
		\pinlabel $\BBar{}M$ [t] at 240 20
		\pinlabel $\BBar{}M$ [t] at 296 20
		\pinlabel $\BBar{}M$ [t] at 352 20
		\endlabellist
		\includegraphics[scale=.45]{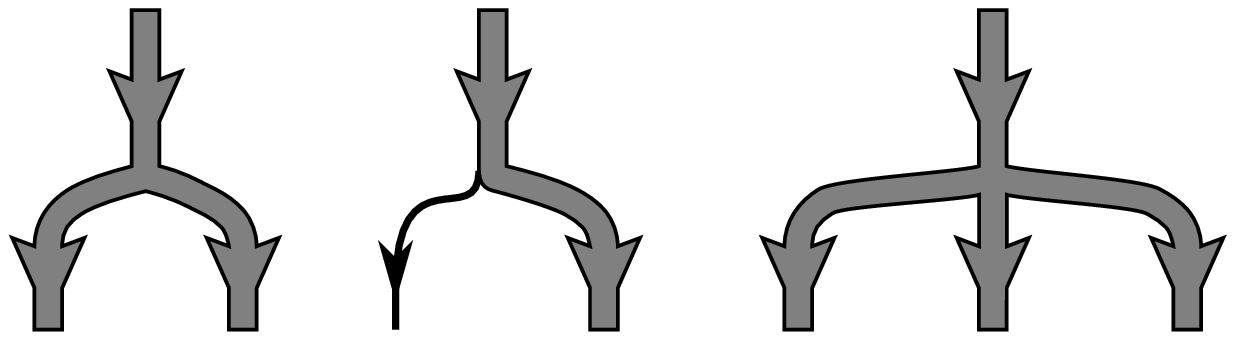}
		\caption{Split maps.}
		\label{subfig:splits}
	\end{subfigure}
	\begin{subfigure}[b]{.48\linewidth}
		\centering
		\labellist
		\tiny\hair 2pt
		\pinlabel $\BBar{}M$ [t] at 52 20
		\pinlabel $\BBar{}M$ [t] at 152 20
		\pinlabel $\BBar{}M$ [t] at 296 20
		\pinlabel $\BBar{}M$ [b] at 24 112
		\pinlabel $\BBar{}M$ [b] at 80 112
		\pinlabel $M$ [b] at 124 112
		\pinlabel $\BBar{}M$ [b] at 184 112
		\pinlabel $\BBar{}M$ [b] at 240 112
		\pinlabel $M$ [b] at 296 112
		\pinlabel $\BBar{}M$ [b] at 352 112
		\endlabellist
		\includegraphics[scale=.45]{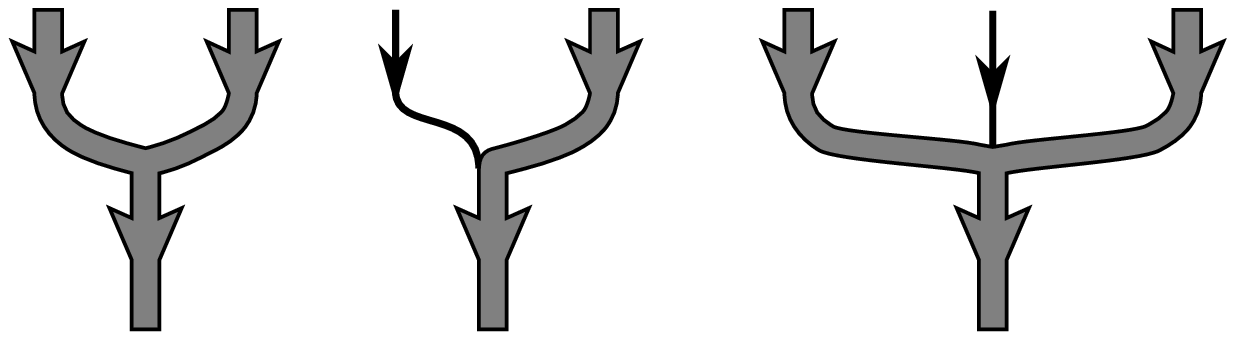}
		\caption{Merge maps.}
		\label{subfig:merges}
	\end{subfigure}
	\caption{}
	\label{fig:split_and_merge}
\end{figure}

\subsection{Algebras and modules}
\label{sec:algebras_and_modules}

The notion of an $\Ainf$--algebra is a generalization of that of a differential graded (or DG) algebra. While the algebras that arise in the context 
of bordered Floer homology are only DG, we give the general definition for completeness. We will omit grading shifts.

\begin{defn}
	\label{def:ainf_algebra}
	An \emph{$\Ainf$--algebra $A$ over the base ring $K$} consists a $K$--bimodule $\li{_K}A_K$, together with a
	collection of linear maps $\mu_i\co A^{\otimes i}\to A$, $i\geq 1$, satisfying certain compatibility conditions. By adding the trivial map
	$\mu_0=0\co K\to A$, we can regard this as a map $\mu=(\mu_i)\co\BBar A\to A$. This induces a map
	$\ol{\mu}\co \BBar A\to \BBar A$, given by splitting $\BBar M$ into three copies of itself, applying $\mu$ to the
	middle one, and merging again (see \Figure~\ref{subfig:mubar}).

	The compatibility condition is $\ol\mu\circ\mu=0$, or equivalently $\ol\mu\circ\ol\mu=0$
	(see \Figure~\ref{subfig:mu_condition}).

	The algebra is \emph{unital} if there is a map $1\co K\to A$ (which we draw as a circle labeled ``$1$'' with an
	outgoing arrow labeled ``$A$''), such that $\mu_2(1,a)=\mu_2(a,1)=a$, and $\mu_i(\ldots,1,\ldots)=0$ if $i\neq 2$.

	The algebra $A$ is \emph{bounded} if $\mu$ is bounded, or equivalently if $\ol\mu$ is relatively bounded in both
	directions.
\end{defn}

Notice that a DG-algebra with multiplication $m$ and differential $d$ is just an $\Ainf$ algebra with $\mu_1=d$,
$\mu_2=m$, and $\mu_i=0$ for $i\geq 3$. Moreover, DG-algebras are always bounded.

Since DG-algebras are associative, there is one more operation that is specific to them.

\begin{defn}
	\label{def:pi}
	The \emph{associative multiplication} $\pi\co\BBar A\to A$ for a DG-algebra $A$ is the map with components
	\begin{equation*}
		\pi_i(a_1\otimes\cdots\otimes a_i)=
		\begin{cases}
			a_1 a_2 \cdots a_i &i>0,\\
			1 &i=0.
		\end{cases}
	\end{equation*}
\end{defn}

\begin{figure}
	\begin{subfigure}[b]{.39\linewidth}
		\labellist
		\pinlabel $\ol\mu$ at 20 78
		\pinlabel $\mu$ at 132 78
		\pinlabel $=$ at 64 78
		\endlabellist
		\centering
		\includegraphics[scale=.5]{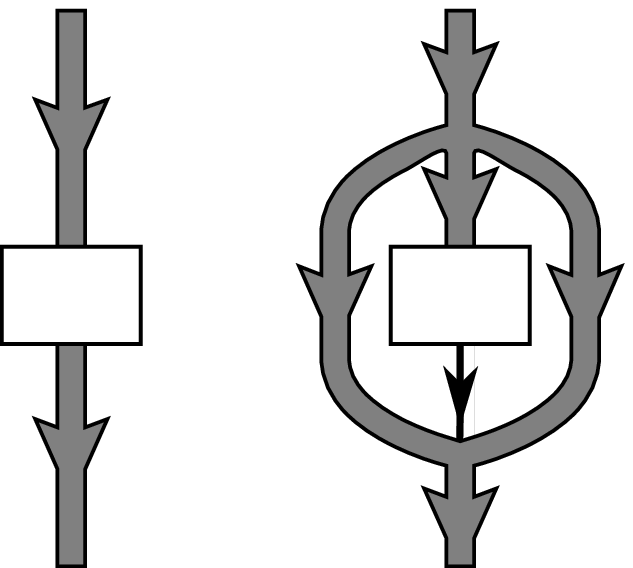}
		\caption{$\bar\mu$ in terms of $\mu$.}
		\label{subfig:mubar}
	\end{subfigure}
	\begin{subfigure}[b]{.59\linewidth}
		\labellist
		\pinlabel $=~0$ [l] at 51 80
		\pinlabel $\Longleftrightarrow$ at 138 80
		\pinlabel $=~0$ [l] at 239 80
		\pinlabel $\ol\mu$  at 20 110
		\pinlabel $\ol\mu$  at 208 110
		\pinlabel $\mu$  at 20 46
		\pinlabel $\ol\mu$  at 208 46
		\endlabellist
		\centering
		\includegraphics[scale=.5]{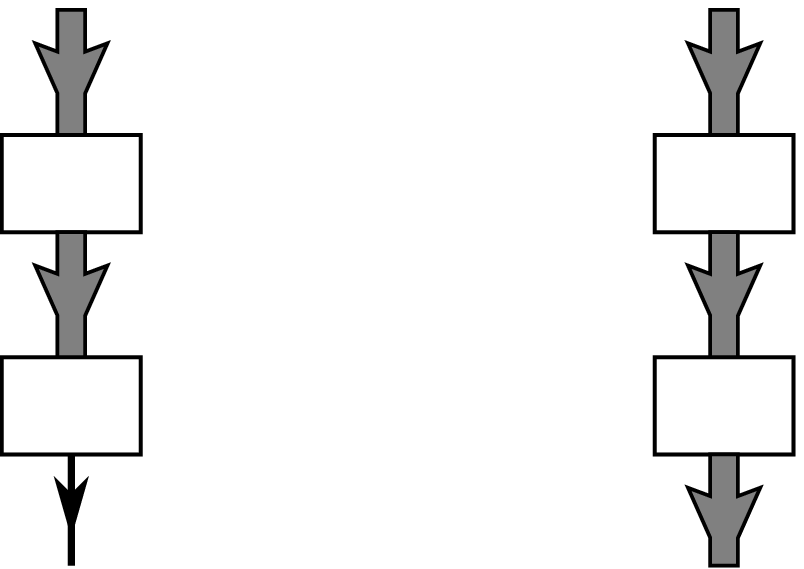}
		\caption{Compatibility conditions.}
		\label{subfig:mu_condition}
	\end{subfigure}
	\caption{Definition of $\Ainf$--algebras}
	\label{fig:ainf_algebra_def}
\end{figure}

There are two types of modules: type--$A$, which is the usual notion of an $\Ainf$--module, and type--$D$. There are
four types of bimodules: type--$AA$, type--$DA$, etc. These can be extend to tri-modules and so on. We describe several
of the bimodules. Other cases can be easily deduced.

Suppose $A$ and $B$ are unital $\Ainf$--algebras with ground rings $K$ and $L$, respectively. We use the following
notation. A type--$A$ module over $A$ will have $A$ as a lower index. A type--$D$ module over $A$ will have $A$ as an
upper index. Module structures over the ground rings $K$ and $L$ are denoted with the usual lower index notation.

\begin{defn}
	\label{def:aa_module}
	A type--$AA$ bimodule $\li{_A}M_B$ consists of a bimodule $\li{_K}M_L$ over the ground rings, together with a map
	$m=(m_{i|1|j})\co\BBar A\otimes M\otimes \BBar B\to M$. The compatibility conditions for $m$ are given in
	\Figure~\ref{fig:aa_condition}.
\end{defn}

The bimodule $M$ is \emph{unital} if $m_{1|1|0}(1_A,m)=m_{0|1|1}(m,1_B)=m$, and $m_{i|1|j}$ vanishes
in all other cases where one of the inputs is $1_A$ or $1_B$.

The bimodule can be \emph{bounded}, \emph{bounded only in $A$}, \emph{relatively bounded in $A$ with respect to
$B$}, etc. These are defined in terms of the index sets of $\BBar A$ and $\BBar B$.

\begin{figure}
	\labellist
	\pinlabel $m$ at 60 110
	\pinlabel $m$ at 60 46
	\pinlabel $\ol\mu_A$ at 176 110
	\pinlabel $m$ at 224 46
	\pinlabel $\ol\mu_B$ at 424 110
	\pinlabel $m$ at 376 46
	\pinlabel $+$ at 142 80
	\pinlabel $+$ at 300 80
	\pinlabel $=~0$ [l] at 456 80
	\endlabellist
	\includegraphics[scale=.5]{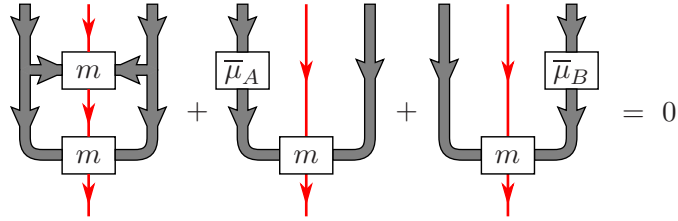}
	\caption{Structure equation for a type--$AA$ module.}
	\label{fig:aa_condition}
\end{figure}

\begin{defn}
	\label{def:da_module}
	A type--$DA$ bimodule $\lu{A}M_B$ consists of a bimodule $\li{_K}M_L$ over the ground rings, together with a map
	$\delta=(\delta_{1|1|j})\co M\otimes \BBar B\to A \otimes M$. This induces another map
	$\ol\delta=(\delta_{i|1|j})\co M\otimes \BBar B\to \BBar A\otimes M$, by splitting $\BBar B$ into $i$ copies, and
	applying $\delta$ $i$--many times (see \Figure~\ref{subfig:deltabar}). The compatibility conditions for $\delta$ and $\ol\delta$ are given in
	\Figure~\ref{subfig:da_condition}.
\end{defn}

The bimodule $M$ is \emph{unital} if $\delta_{1|1|1}(m,1_B)=1_A\otimes m$, and $\delta_{1|1|i}$ vanishes for
$i>1$ if one of the inputs is $1_B$.

Again, there are various boundedness conditions that can be imposed.

\begin{figure}
	\begin{subfigure}[b]{.30\linewidth}
		\centering
		\labellist
		\pinlabel $\vdots$ at 126 90
		\pinlabel $=$ at 98 80
		\pinlabel $\vdots$ at 160 87
		\pinlabel $\vdots$ at 208 87
		\small
		\pinlabel $\ol\delta$ at 46 78
		\pinlabel $\delta$ at 160 125
		\pinlabel $\delta$ at 160 40
		\endlabellist
		\includegraphics[scale=.45]{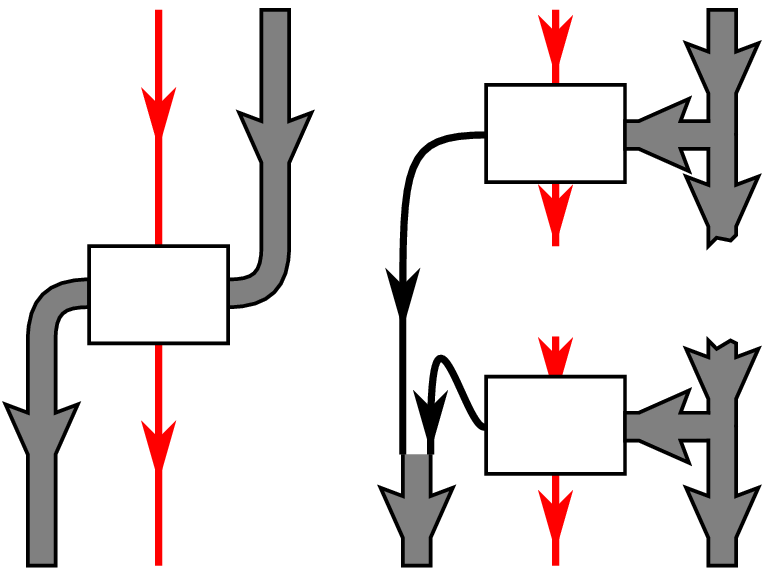}
		\caption{$\ol\delta$ in terms of $\delta$.}
		\label{subfig:deltabar}
	\end{subfigure}
	\begin{subfigure}[b]{.67\linewidth}
		\labellist
		\pinlabel $+$ at 102 80
		\pinlabel $+$ at 400 80
		\pinlabel $=0$ [l] at 188 80
		\pinlabel $=0$ [l] at 480 80
		\pinlabel $\Longleftrightarrow$ at 272 80
		\small
		\pinlabel $\ol\mu_B$ at 68 110
		\pinlabel $\ol\delta$ at 168 110
		\pinlabel $\ol\mu_B$ at 360 110
		\pinlabel $\ol\delta$ at 460 110
		\pinlabel $\delta$ at 36 46
		\pinlabel $\mu_A$ at 136 46
		\pinlabel $\ol\delta$ at 328 46
		\pinlabel $\ol\mu_A$ at 428 46
		\endlabellist
		\centering
		\includegraphics[scale=.45]{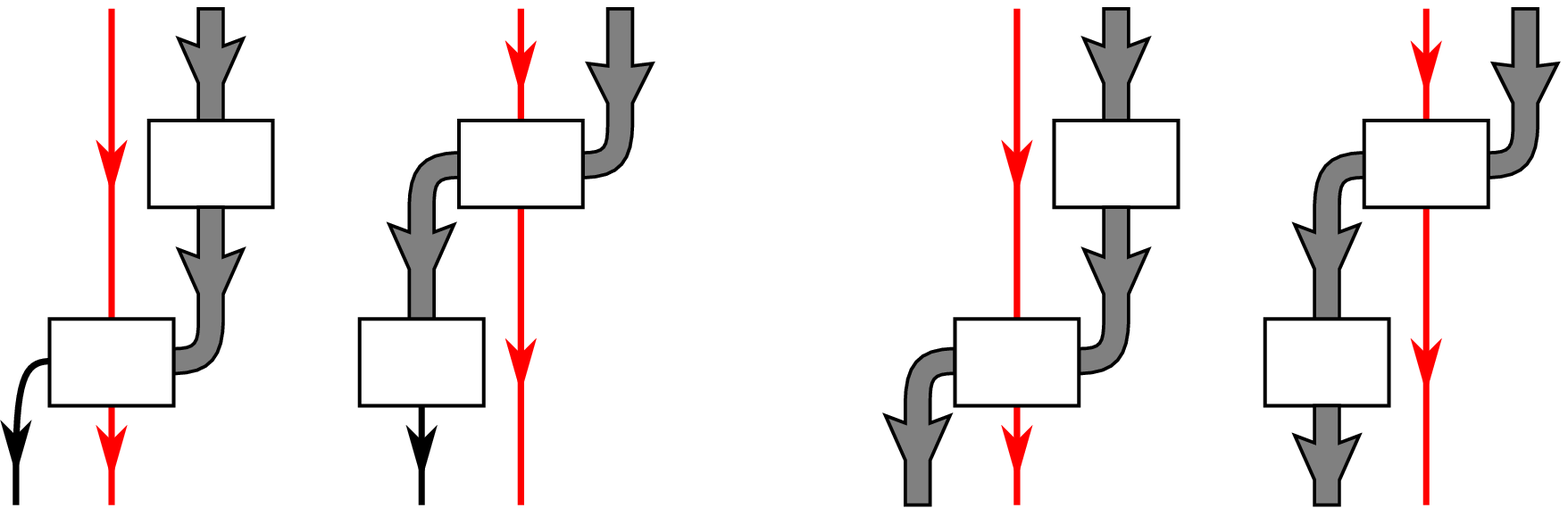}
		\caption{Structure equation for a type--$DA$ module.}
		\label{subfig:da_condition}
	\end{subfigure}
		\caption{}
		\label{fig:da_def}
\end{figure}

Type--$DD$ modules only behave well if the algebras involved are DG, so we only give the definition for that case.

\begin{defn}
	\label{def:dd_module}
	Suppose $A$ and $B$ are DG-algebras.
	A type $DD$--module $\lu{A}M^B$ consists of a bimodule $\li{_K}M_L$ over the ground rings, together with a map
	$\delta_{1|1|1}\co M\to A\otimes M\otimes B$ satisfying the condition in \Figure~\ref{fig:dd_condition}.
\end{defn}

\begin{figure}
	\labellist
	\pinlabel $+$ at 108 80
	\pinlabel $+$ at 232 80
	\pinlabel $=0$ at 372 80
	\small
	\pinlabel $\delta$ at 52 114
	\pinlabel $\delta$ at 164 114
	\pinlabel $\delta$ at 300 122
	\pinlabel $\delta$ at 300 74
	\pinlabel $\mu_A$ at 20 50
	\pinlabel $\mu_B$ at 196 50
	\pinlabel $\mu_A$ at 268 34
	\pinlabel $\mu_B$ at 332 34
	\endlabellist
	\includegraphics[scale=.5]{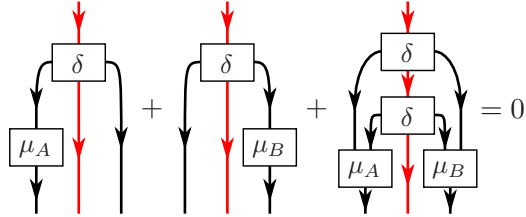}
	\caption{Structure equation for a type--$DD$ module.}
	\label{fig:dd_condition}
\end{figure}

We omit the definition of one-sided type--$A$ and type--$D$ modules, as they can be regarded as special cases of
bimodules. Type--$A$ modules over $A$ can be interpreted as type--$AA$ bimodules over $A$ and $B=\ZZ/2$. Similarly,
type--$D$ modules are type $DA$--modules over $\ZZ/2$.

\subsection{Tensor products}
\label{sec:tensors}

There are two types of tensor products for $\Ainf$--modules. One is the more traditional derived tensor product
$\dtens$. It is generally hard to work with, as $M\dtens N$ is infinite dimensional over $\ZZ/2$ even when $M$ and $N$
are finite dimensional. This is bad for computational reasons, as well as when using diagrams---it violates some of the
assumptions of \Appendix~\ref{sec:diagrams}. Nevertheless, we do use it in a few places throughout the paper.

Throughout the rest of this section assume that $A$, $B$, and $C$ are DG-algebras over the ground rings $K$, $L$, and
$P$, respectively.

\begin{defn}
	\label{def:dtens}
	Suppose $\li{_A}M_B$ and $\li{_B}N_C$ are two type--$AA$ bimodules.
	The \emph{derived tensor product} $(\li{_A}M_B)\dtens_B(\li{_B}N_C)$ is a type--$AA$ bimodule $\li{_A}(M\dtens N)_B$
	defined as follows. Its underling bimodule over the ground rings is
	\begin{equation*}
		\begin{split}
			\li{_K}(M\dtens N)_P&=(\li{_K}M_L)\otimes_L\left(\bigoplus_{i=0}^\infty {\li{_L}B_L}^{\otimes
			i}\right)\otimes_L(\li{_L}N_P)\\
			&=M\otimes_L \BBar B\otimes_L N.
		\end{split}
	\end{equation*}
	Here we're slightly abusing notation in identifying $\BBar B$ with a direct sum. The structure map as an
	$\Ainf$--bimodule over $A$ and $C$ is $m_{M\dtens N}$, as shown in \Figure~\ref{subfig:aaaa_dtens_def}.
\end{defn}

Similarly, we can take the derived tensor product of a $DA$ module and an $AA$ module, or a $DA$ module and an $AD$
module. The former is demonstrated in \Figure~\ref{subfig:daaa_dtens_def}.

\begin{figure}
	\begin{subfigure}[b]{\linewidth}
		\centering
		\labellist
		\pinlabel $=$ at 138 50
		\pinlabel $+$ at 288 50
		\pinlabel $+$ at 432 50
		\small
		\pinlabel $m_{M\dtens{}N}$ at 60 50
		\pinlabel $m_M$ at 200 50
		\pinlabel $m_N$ at 376 50
		\pinlabel $\ol\mu_B$ at 484 50
		\endlabellist
		\includegraphics[scale=.6]{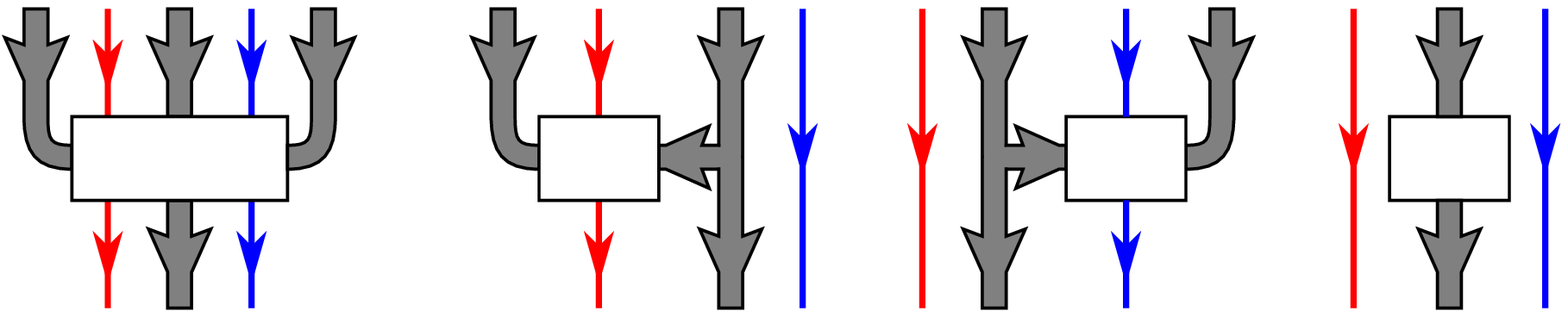}
		\caption{$\li{_A}M_B\dtens\li{_B}N_C$}
		\label{subfig:aaaa_dtens_def}
	\end{subfigure}
	\begin{subfigure}[b]{\linewidth}
		\centering
		\labellist
		\pinlabel $=$ at 127 50
		\pinlabel $+$ at 266 50
		\pinlabel $+$ at 430 50
		\small
		\pinlabel $\delta_{M\dtens{}N}$ at 54 50
		\pinlabel $\delta_M$ at 180 50
		\pinlabel $m_N$ at 376 50
		\pinlabel $\ol\mu_B$ at 504 50
		\pinlabel $1$ at 288 50
		\pinlabel $1$ at 452 50
		\endlabellist
		\includegraphics[scale=.6]{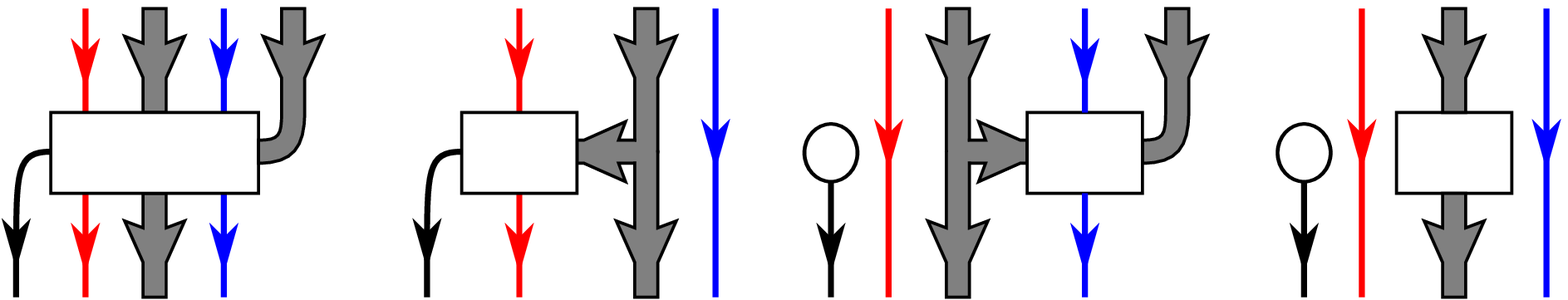}
		\caption{$\lu{A}M_B\dtens\li{_B}N_C$}
		\label{subfig:daaa_dtens_def}
	\end{subfigure}
	\caption{Structure maps for two types of $\dtens$ products.}
	\label{fig:dtens_def}
\end{figure}

The other type of tensor product is the square tensor product $\sqtens$. It is asymmetric, as it requires one side to be
a type--$D$ module, and the other to be a type--$A$ module. The main advantage of $\sqtens$ over $\dtens$ is that
$M\sqtens N$ is finite dimensional over $\ZZ/2$ whenever $M$ and $N$ are. Its main disadvantage is that $M\sqtens N$ is
only defined subject to some boundedness conditions on $M$ and $N$.

\begin{defn}
	\label{def:sqtens}
	Suppose $\li{_A}M_B$ is a type--$AA$ bimodule and $\lu{B}N_C$ is a type--$DA$ bimodule, such that at least one of
	$M$ and $N$ is relatively bounded in $B$. The \emph{square tensor product}
	$(\li{_A}M_B)\sqtens_B(\li{_B}N_C)$ is a type--$AA$ bimodule $\li{_A}(M\sqtens N)_C$ defined as follows.
	Its underlying bimodule over the ground rings is
	\begin{equation*}
		\li{_K}(M\sqtens N)_P=(\li{_K}M_L)\sqtens_L(\li{_L}N_P),
	\end{equation*}
	and its structure map is $m_{M\sqtens N}$ as shown in \Figure~\ref{subfig:aada_sqtens_def}.
\end{defn}

There are three other combinations depending on whether the modules are of type $D$ or $A$ with respect 
$A$ and $C$. All combinations are shown in \Figure~\ref{fig:sqtens_def}.

\begin{figure}
	\begin{subfigure}[b]{.48\linewidth}
		\centering
		\labellist
		\pinlabel $=$ at 120 50
		\small
		\pinlabel $m_{M\sqtens{}N}$ at 55 50
		\pinlabel $m_M$ at 176 50
		\pinlabel $\ol\delta_N$ at 244 50
		\endlabellist
		\includegraphics[scale=.5]{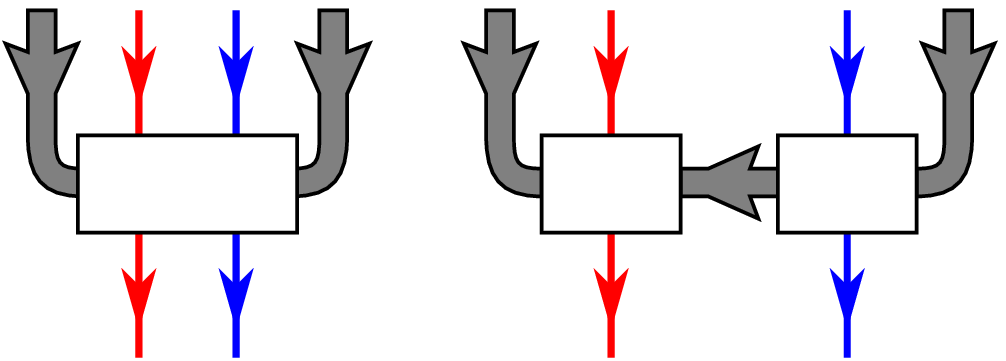}
		\caption{$\li{_A}M_B\sqtens\lu{B}N_C$}
		\label{subfig:aada_sqtens_def}
	\end{subfigure}
	\begin{subfigure}[b]{.48\linewidth}
		\centering
		\labellist
		\pinlabel $=$ at 120 50
		\small
		\pinlabel $\delta_{M\sqtens{}N}$ at 54 50
		\pinlabel $\delta_M$ at 176 50
		\pinlabel $\ol\delta_N$ at 244 50
		\endlabellist
		\includegraphics[scale=.5]{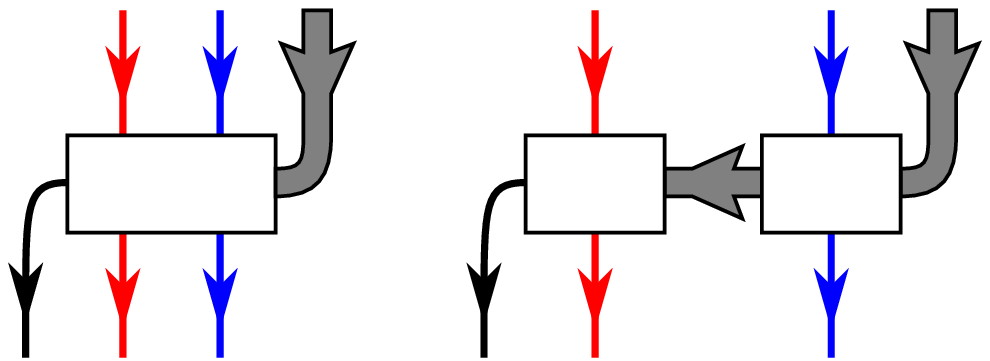}
		\caption{$\lu{A}M_B\sqtens\lu{B}N_C$}
		\label{subfig:dada_sqtens_def}
	\end{subfigure}
	\begin{subfigure}[b]{.48\linewidth}
		\centering
		\labellist
		\pinlabel $=$ at 120 50
		\small
		\pinlabel $\delta_{M\sqtens{}N}$ at 54 50
		\pinlabel $m_M$ at 176 50
		\pinlabel $\ol\delta_N$ at 244 50
		\pinlabel $\pi_C$ at 300 34
		\endlabellist
		\includegraphics[scale=.5]{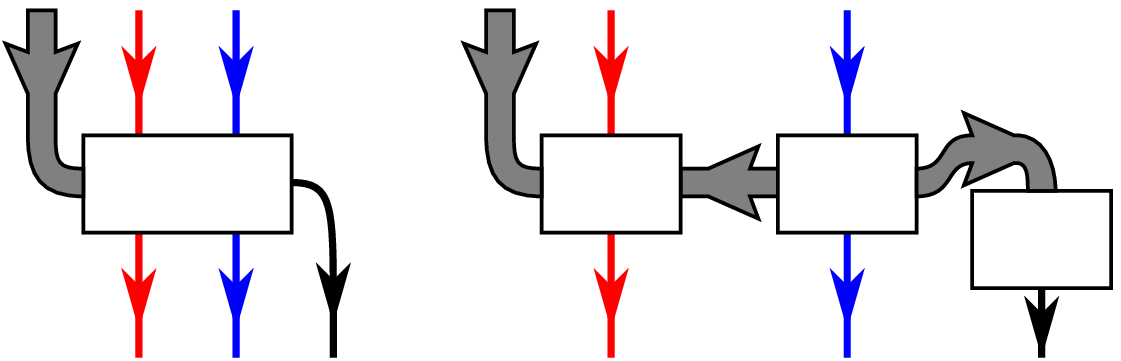}
		\caption{$\li{_A}M_B\sqtens\lu{B}N^C$}
		\label{subfig:aadd_sqtens_def}
	\end{subfigure}
	\begin{subfigure}[b]{.48\linewidth}
		\centering
		\labellist
		\pinlabel $=$ at 120 50
		\small
		\pinlabel $\delta_{M\sqtens{}N}$ at 54 50
		\pinlabel $\delta_M$ at 176 50
		\pinlabel $\ol\delta_N$ at 244 50
		\pinlabel $\pi_C$ at 300 34
		\endlabellist
		\includegraphics[scale=.5]{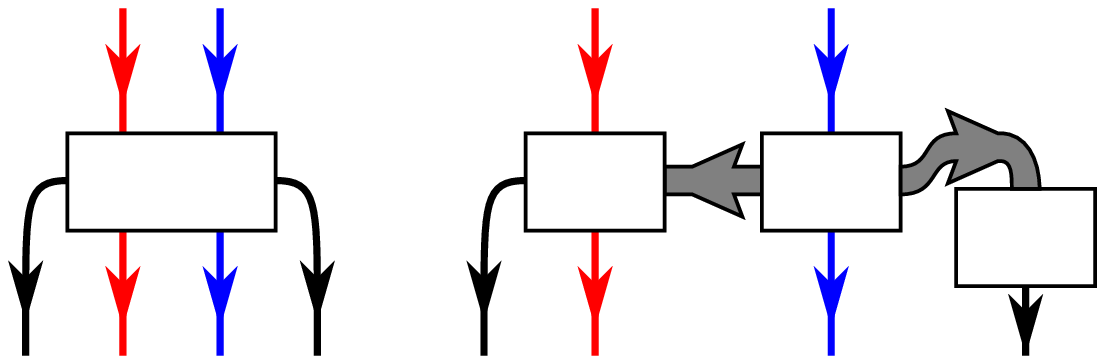}
		\caption{$\lu{A}M_B\sqtens\lu{B}N^C$}
		\label{subfig:dadd_sqtens_def}
	\end{subfigure}
	\caption{Structure maps for the four types of $\sqtens$ products.}
	\label{fig:sqtens_def}
\end{figure}

\subsection{Morphisms and homomorphisms}
\label{sec:morphisms_and_premorphisms}

There are two different notions of morphisms when working with $\Ainf$--modules and bimodules. The more natural one is
that of \emph{homomorphisms}, which generalize chain maps. However, if we work only with homomorphisms, too much
information is lost. For this reason we also consider the more general \emph{morphisms}. These generalize linear maps of
chain complexes, which do not necessarily respect differentials.

\begin{defn}
	\label{def:morphism}
	A \emph{morphism} $f\co M\to N$ between two bimodules $M$ and $N$ of the same type is a collection of maps
	of the same type as the structure maps for $M$ and $N$. For example, $f\co \li{_A}M_B\to \li{_A}N_B$ has components
	$f_{i|1|j}\co\BBar A\otimes M\otimes\BBar B\to N$. The spaces of morphisms are denoted by $\li{_A}\Mor_B(M,N)$, etc.
\end{defn}

Suppose $A$ and $B$ are DG-algebras.
The bimodules of each type, e.g. $\li{_A}\Mod_B$, form a DG-category, with morphism spaces $\li{_A}\Mor_B$, etc.
The differentials and composition maps for each type are shown in \Figures~\ref{fig:premorphism_differentials} and~\ref{fig:premorphism_compositions},
respectively.

\begin{defn}
	\label{def:homorphism}
	A \emph{homomorphism} $f\co M\to N$ of bimodules is a morphism $f$ which is a cycle, i.e., $\del f=0$. A
	\emph{null-homotopy} of $f$ is a morphism $H$, such that $\del H=f$. The space of homomorphisms up to homotopy is
	denoted by $\li{_A}\Hom_B$, etc.
\end{defn}

Notice that the homomorphism space $\li{_A}\Hom_B(M,N)$ is exactly the homology of $\li{_A}\Mor_B(M,N)$. This gives us a new
category of bimodules.

Having homomorphisms and homotopies allows us to talk about homotopy equivalences of modules. For example, if
$\li{_A}M_B$ is a bimodule, then $A\dtens M\simeq M\simeq M\dtens B$, via canonical homotopy equivalences. For example,
there is $h_M\co A\dtens M\to M$, which we used in several places.

\begin{figure}
	\begin{subfigure}[b]{\linewidth}
		\centering
		\labellist
		\pinlabel $=$ at 100 80
		\pinlabel $+$ at 244 80
		\pinlabel $+$ at 388 80
		\pinlabel $+$ at 500 80
		\small
		\pinlabel $\del{}f$ at 44 74
		\pinlabel $m_M$ at 172 110
		\pinlabel $f$ at 316 110
		\pinlabel $\mu_A$ at 412 110
		\pinlabel $\mu_B$ at 588 110
		\pinlabel $f$ at 172 46
		\pinlabel $m_N$ at 316 46
		\pinlabel $f$ at 444 46
		\pinlabel $f$ at 556 46
		\endlabellist
		\includegraphics[scale=.55]{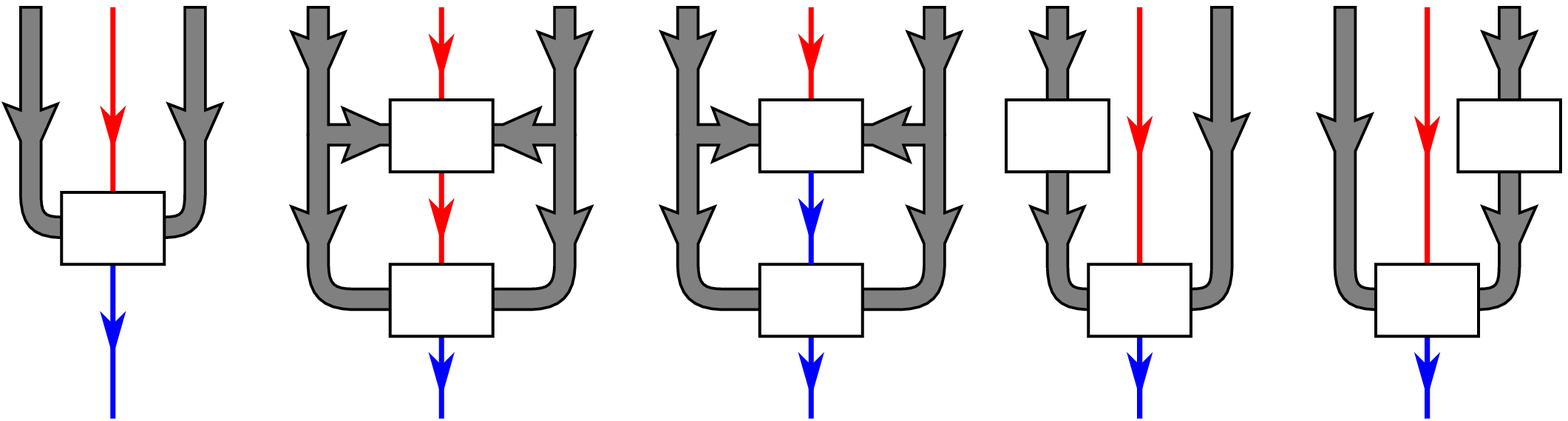}
		\caption{Type--$AA$.}
		\label{subfig:aa_premorphism_differential}
	\end{subfigure}
	\begin{subfigure}[b]{\linewidth}
		\centering
		\labellist
		\pinlabel $=$ at 92 80
		\pinlabel $+$ at 224 80
		\pinlabel $+$ at 356 80
		\pinlabel $+$ at 464 80
		\small
		\pinlabel $\del{f}$ at 40 74
		\pinlabel $f$ at 152 74
		\pinlabel $f$ at 284 122
		\pinlabel $f$ at 408 114
		\pinlabel $f$ at 484 46
		\pinlabel $\delta_M$ at 152 122
		\pinlabel $\delta_N$ at 284 74
		\pinlabel $\mu_A$ at 120 34
		\pinlabel $\mu_A$ at 252 34
		\pinlabel $\mu_A$ at 376 50
		\pinlabel $\ol\mu_B$ at 516 110
		\endlabellist
		\includegraphics[scale=.55]{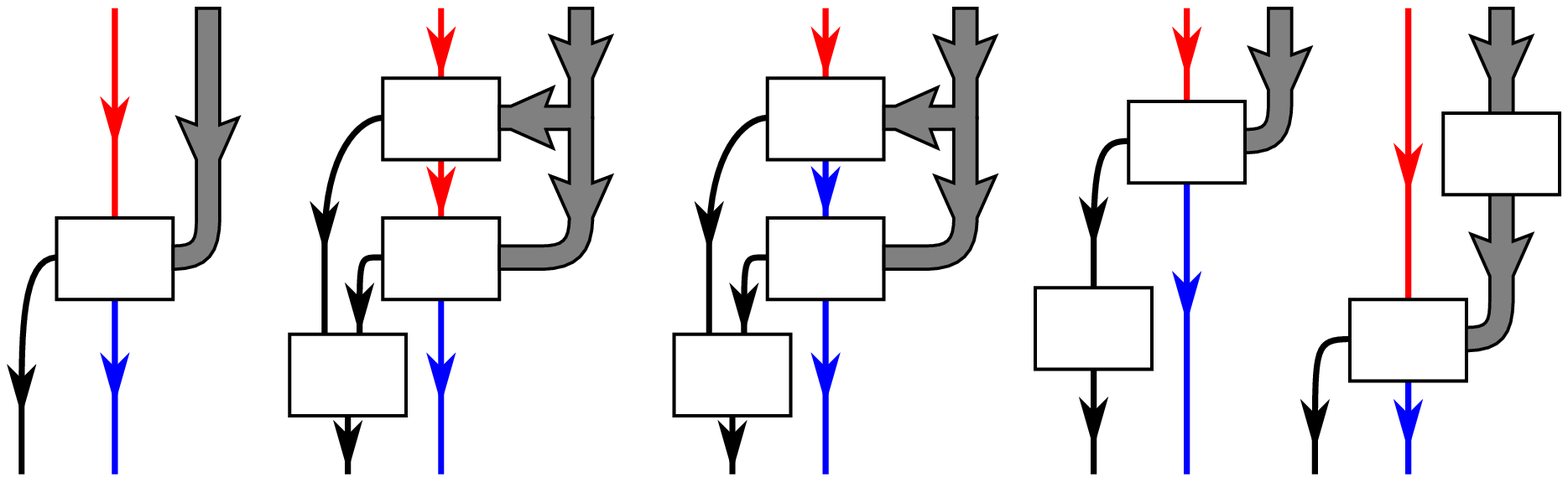}
		\caption{Type--$DA$.}
		\label{subfig:da_premorphism_differential}
	\end{subfigure}
	\begin{subfigure}[b]{\linewidth}
		\centering
		\labellist
		\pinlabel $=$ at 92 80
		\pinlabel $+$ at 212 80
		\pinlabel $+$ at 332 80
		\pinlabel $+$ at 444 80
		\small
		\pinlabel $\del{f}$ at 40 74
		\pinlabel $f$ at 152 74
		\pinlabel $f$ at 272 122
		\pinlabel $f$ at 392 114
		\pinlabel $f$ at 496 114
		\pinlabel $\delta_M$ at 152 122
		\pinlabel $\delta_N$ at 272 74
		\pinlabel $\mu_A$ at 120 34
		\pinlabel $\mu_B$ at 184 34
		\pinlabel $\mu_A$ at 240 34
		\pinlabel $\mu_B$ at 304 34
		\pinlabel $\mu_A$ at 360 50
		\pinlabel $\mu_B$ at 528 50
		\endlabellist
		\includegraphics[scale=.55]{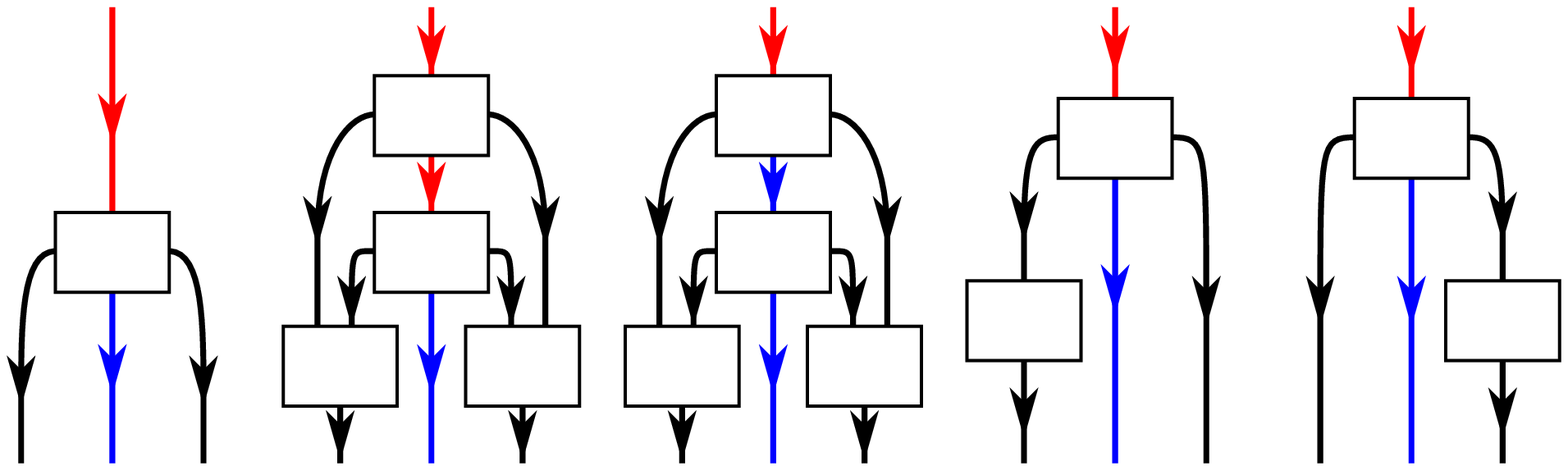}
		\caption{Type--$DD$.}
		\label{subfig:dd_premorphism_differential}
	\end{subfigure}
	\caption{Differentials of the different types of morphisms.}
	\label{fig:premorphism_differentials}
\end{figure}
\begin{figure}
	\begin{subfigure}[b]{.34\linewidth}
		\centering
		\labellist
		\pinlabel $=$ at 100 80
		\small
		\pinlabel $g\!\circ\!\!f$ at 44 74
		\pinlabel $f$ at 172 110
		\pinlabel $g$ at 172 46
		\endlabellist
		\includegraphics[scale=.5]{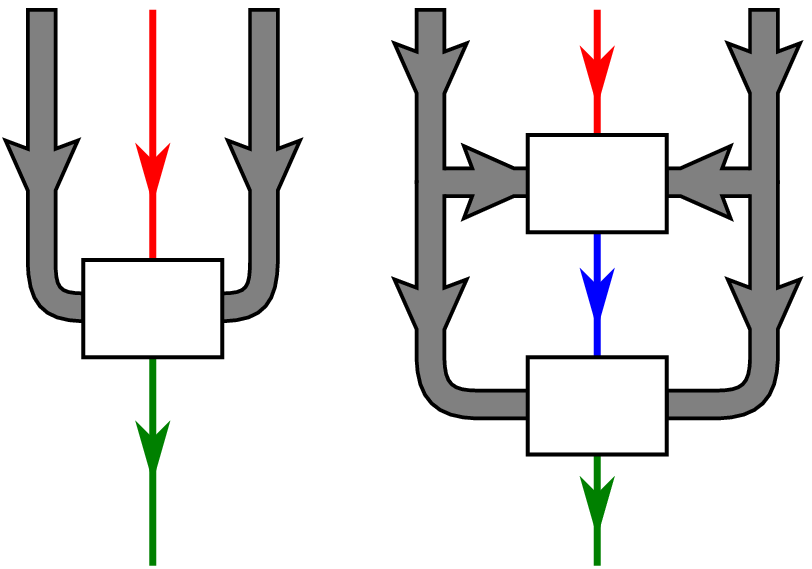}
		\caption{Type--$AA$.}
		\label{subfig:aa_premorphism_composition}
	\end{subfigure}
	\begin{subfigure}[b]{.30\linewidth}
		\centering
		\labellist
		\pinlabel $=$ at 92 80
		\small
		\pinlabel $g\!\circ\!\!f$ at 40 74
		\pinlabel $g$ at 152 74
		\pinlabel $f$ at 152 122
		\pinlabel $\mu_A$ at 120 34
		\endlabellist
		\includegraphics[scale=.5]{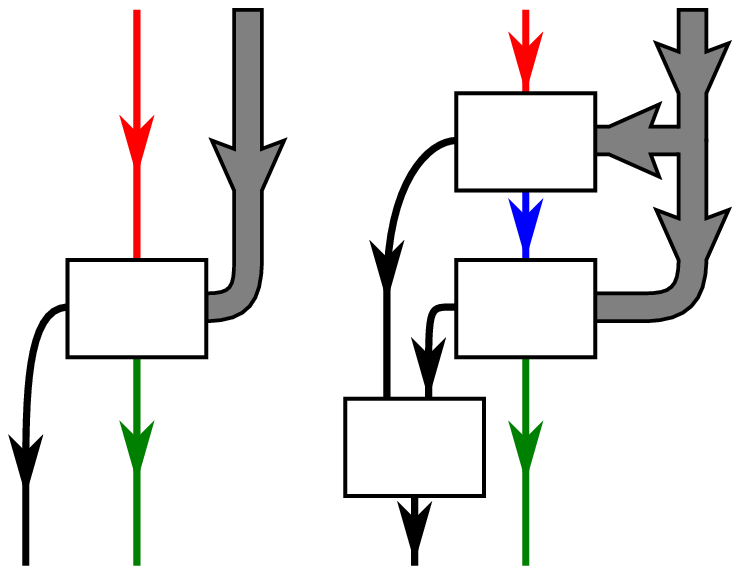}
		\caption{Type--$DA$.}
		\label{subfig:da_premorphism_composition}
	\end{subfigure}
	\begin{subfigure}[b]{.30\linewidth}
		\centering
		\labellist
		\pinlabel $=$ at 92 80
		\small
		\pinlabel {$g\!\circ\!\!f$} at 40 74
		\pinlabel $g$ at 152 74
		\pinlabel $f$ at 152 122
		\pinlabel $\mu_A$ at 120 34
		\pinlabel $\mu_B$ at 184 34
		\endlabellist
		\includegraphics[scale=.5]{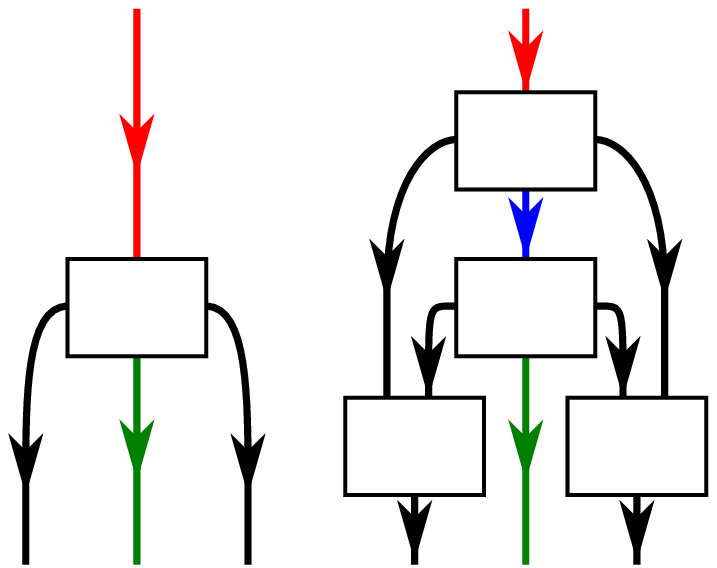}
		\caption{Type--$DA$.}
		\label{subfig:dd_premorphism_composition}
	\end{subfigure}
	\caption{Compositions of the different types of morphisms.}
	\label{fig:premorphism_compositions}
\end{figure}

\subsection{Induced morphisms}
\label{sec:induced_maps}

Suppose $f\co M\to N$ is a bimodule morphism. This induces morphisms
\begin{align*}
	f\dtens\id&\co M\dtens P\to N\dtens P &f\sqtens\id&\co M\sqtens P\to N\sqtens P,\\
\end{align*}
whenever the tensor products are defined. The main types of induced morphisms are shown in \Figure~\ref{fig:induced_maps}.
The functors $\cdot\sqtens\id$ and $\cdot\dtens\id$ are DG-functors. That is, they preserve homomorphisms, homotopies, and
compositions.

\begin{figure}
	\begin{subfigure}[b]{.44\linewidth}
		\centering
		\labellist
		\pinlabel $=$ at 120 76
		\small
		\pinlabel $f\!\sqtens\!\id_N$ at 54 76
		\pinlabel $f$ at 176 76
		\pinlabel $\ol\delta_N$ at 244 76
		\endlabellist
		\includegraphics[scale=.54]{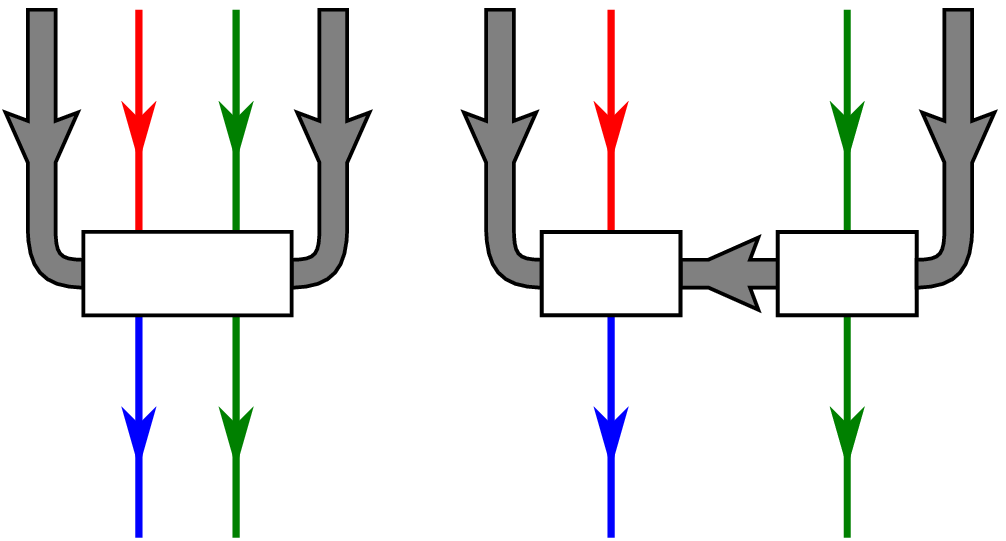}
		\caption{}
		\label{subfig:a_to_d_induced_map}
	\end{subfigure}
	\begin{subfigure}[b]{.50\linewidth}
		\centering
		\labellist
		\pinlabel $=$ at 120 76
		\small
		\pinlabel $\id_M\!\sqtens\!f$ at 54 76
		\pinlabel $m_M$ at 176 76
		\pinlabel $\ol\delta_N$ at 268 120
		\pinlabel $f$ at 268 76
		\pinlabel $\ol\delta_P$ at 268 32
		\endlabellist
		\includegraphics[scale=.54]{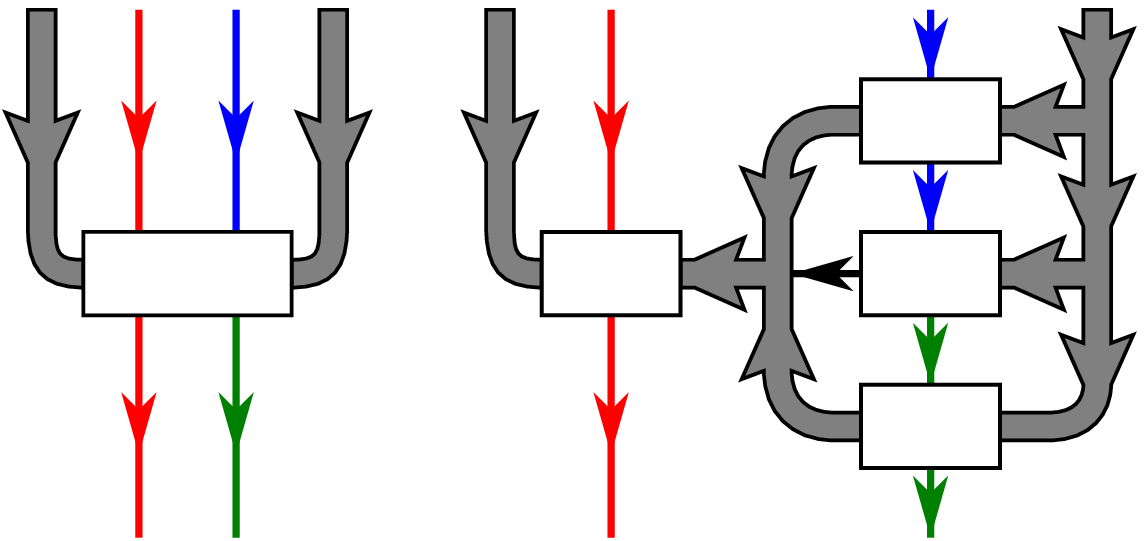}
		\caption{}
		\label{subfig:d_to_a_induced_map}
	\end{subfigure}
	\begin{subfigure}[b]{.48\linewidth}
		\centering
		\labellist
		\pinlabel $=$ at 132 50
		\small
		\pinlabel $f\!\dtens\!\id_N$ at 60 50
		\pinlabel $f$ at 188 50
		\endlabellist
		\includegraphics[scale=.55]{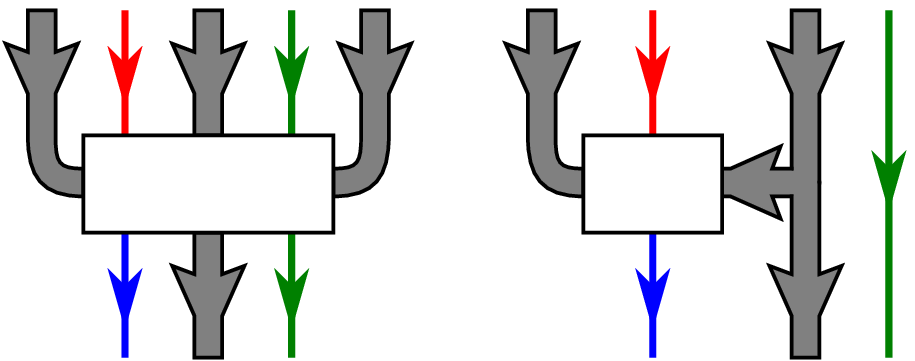}
		\caption{}
		\label{subfig:dtens_induced_map}
	\end{subfigure}
	\caption{Three types of induced maps on tensor products.}
	\label{fig:induced_maps}
\end{figure}

\subsection{Duals}
\label{sec:duals}
There are two operations on modules, which can be neatly expressed by diagrams. One is the operation of turning a bimodule
$\li{_A}M_B$ into a bimodule $\li{_{B^{\op}}}M_{A^{\op}}$. (Similarly, type--$DA$ bimodules become type--$AD$ bimodules,
etc.) Diagrammatically this is achieved by reflecting diagrams along the vertical
axis. See \Figure~\ref{fig:opposite_module} for an example.

The other operation is dualizing modules and bimodules. If $\li{_A}M_B$ has an underlying bimodule $\li{_K}M_L$ over the
ground rings, then
its dual $\li{_B}{M^\vee}_A$ has an underlying bimodule $\li{_L}{M^\vee}_K=(\li{_K}M_L)^\vee$. Diagrammatically this is achieved
by rotating diagrams by 180 degrees. Again, there are variations for type--$D$ modules. See
\Figure~\ref{fig:dual_module} for an example.

Since the structure equations are symmetric, it is immediate that both of
these operations send bimodules to bimodules, as long as we restrict to modules finitely generated over $\ZZ/2$.

\begin{figure}
	\labellist
	\pinlabel $=$ at 384 52
	\small\hair=0.5pt
	\pinlabel $A$ [t] at 16 20
	\pinlabel $M$ [t] at 48 20
	\pinlabel $A^{\op}$ [t] at 340 20
	\pinlabel $M$ [t] at 308 20
	\pinlabel $A^{\op}$ [t] at 500 20
	\pinlabel $M$ [t] at 468 20
	\pinlabel $B$ [b] at 80 86
	\pinlabel $M$ [b] at 48 86
	\pinlabel $B^{\op}$ [b] at 276 86
	\pinlabel $M$ [b] at 308 86
	\pinlabel $B^{\op}$ [b] at 436 86
	\pinlabel $M$ [b] at 468 86
	\pinlabel $m_M$ at 48 52
	\pinlabel \reflectbox{$m_M$} at 308 52
	\pinlabel $m_M^{\op}$ at 468 52
	\endlabellist
	\includegraphics[scale=.5]{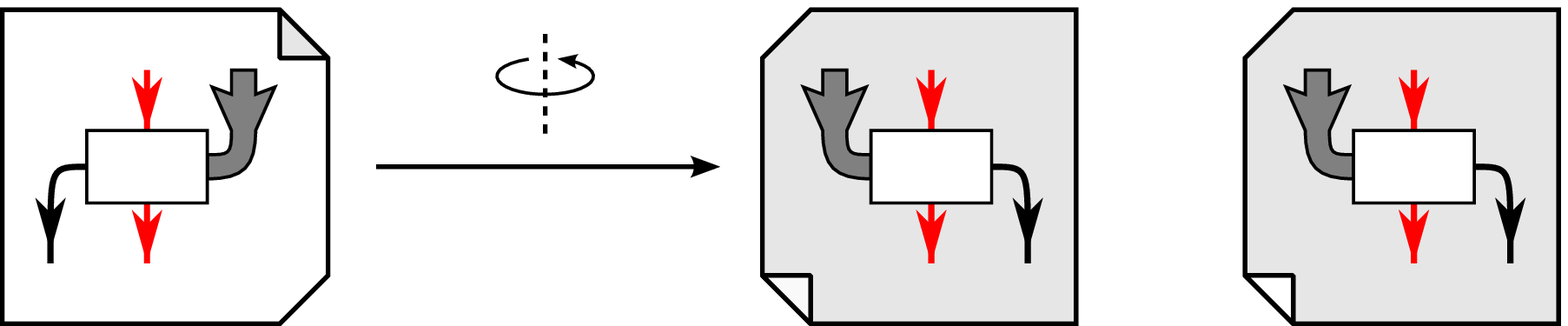}
	\caption{Passing from $\lu{A}\Mod_B$ to $\li{_{B^{\op}}}\Mod^{A^{\op}}$ by reflection.}
	\label{fig:opposite_module}
\end{figure}

\begin{figure}
	\labellist
	\pinlabel $=$ at 384 52
	\small\hair=0.5pt
	\pinlabel $A$ [t] at 16 20
	\pinlabel $M$ [t] at 48 20
	\pinlabel $A$ [b] at 344 86
	\pinlabel $M$ [t] at 312 20
	\pinlabel $A$ [t] at 500 20
	\pinlabel $M$ [t] at 468 20
	\pinlabel $B$ [b] at 80 86
	\pinlabel $M$ [b] at 48 86
	\pinlabel $B$ [t] at 280 20
	\pinlabel $M$ [b] at 312 86
	\pinlabel $B$ [b] at 436 86
	\pinlabel $M$ [b] at 468 86
	\pinlabel $m_M$ at 48 52
	\pinlabel \rotatebox[origin=c]{180}{$m_M$} at 312 52
	\pinlabel $m_{M^\vee}$ at 470 52
	\endlabellist
	\includegraphics[scale=.5]{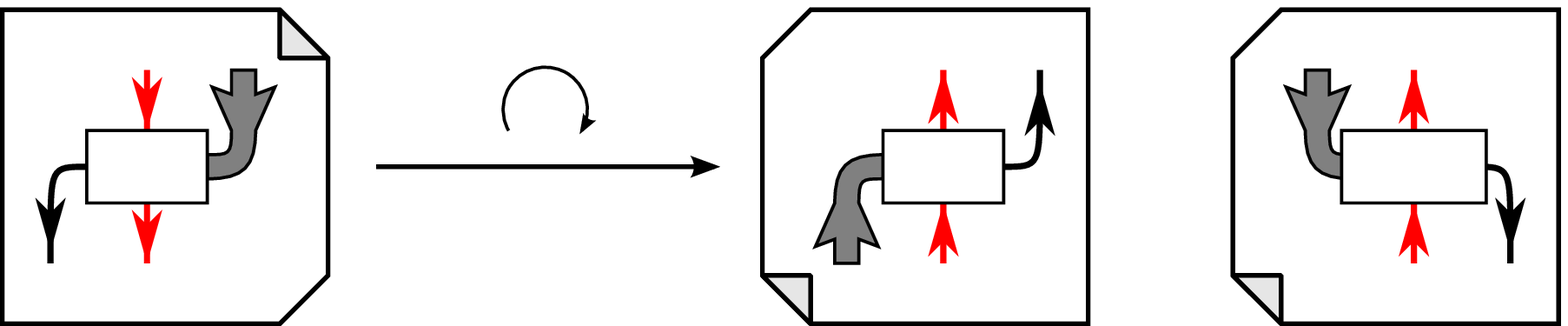}
	\caption{Passing from $\lu{A}\Mod_B$ to $\li{_B}\Mod^A$ by rotation.}
	\label{fig:dual_module}
\end{figure}

This gives equivalences of the DG-categories
\begin{equation*}
	\lu{A}\Mod_B\cong \li{_{B^{\op}}}\Mod^{A^{\op}} \cong\left(\li{_B}\Mod^A\right)^{\op},
\end{equation*}
etc. One can check that both constructions extend to tensors, induced morphisms, etc.

\bibliographystyle{hamsalpha2}
\bibliography{../bibliography/all}

\end{document}